\documentclass[reqno,12pt]{amsart}
\usepackage[utf8]{inputenc}
\usepackage{amsmath,amssymb,amsthm,mathrsfs,color,times,textcomp,verbatim,yfonts}

\usepackage[left=1.2in, right=1.2in, top=1in]{geometry}

\usepackage{subcaption}
\usepackage[export]{adjustbox}
\usepackage{esint}
\usepackage{xcolor}
\usepackage{array}
\usepackage[colorlinks=true]{hyperref}
\hypersetup{urlcolor=blue, citecolor=red, linkcolor=blue}
\usepackage{mathtools}
\usepackage[shortlabels]{enumitem}
\usepackage[square,numbers]{natbib}

\def \iS{\int_{\mathbb{S}^1}}

\numberwithin{equation}{section}
\theoremstyle{plain} 
\newtheorem{theorem}{Theorem}[section]
\newtheorem{proposition}[theorem]{Proposition}
\newtheorem{lemma}[theorem]{Lemma}

\newtheorem{corollary}[theorem]{Corollary}

\usepackage{graphicx}

\theoremstyle{definition}
\newtheorem{definition}[theorem]{Definition}

\newtheorem*{notation}{\textbf{Notation}}

\newtheorem{remark}[theorem]{Remark}

\title[Ancient solution to $\alpha$-CSF]{Classification of ancient flows by sub-affine-critical powers of curvature in $\mathbb{R}^2$}
\author{Kyeongsu Choi}
\address{School of Mathematics, Korea Institute for Advanced Study, 85 Hoegiro, Dongdaemun-gu, Seoul 02455, Republic of Korea.}
\email{choiks@kias.re.kr}
\author{Liming Sun}
\address{Academy of Mathematics and Systems Science, the Chinese Academy of Sciences, Beijing 100190, China.}
\email{lmsun@amss.ac.cn}
\date{\today}

\begin{document}

\maketitle
\begin{abstract}
We classify closed convex ancient $\alpha$-curve shortening flows for sub-affine-critical powers $\alpha \leq \frac{1}{3}$. In addition, we show that closed convex smooth finite entropy ancient $\alpha$-curve shortening flows with $\frac{1}{3}<\alpha$ are shrinking circles.  After rescaling, the ancient flows satisfying the above conditions converge exponentially fast to smooth closed convex shrinkers as the time goes to negative infinity. In particular, when $\alpha=\frac{1}{k^2-1}$ with $3\leq k \in \mathbb{N}$, the round circle shrinker has non-trivial Jacobi fields, but the ancient flows asymptotic to shrinking circles do not evolve along the Jacobi fields.
\end{abstract}

\tableofcontents

\setlength{\parskip}{0.5em}
\section{Introduction}

We call a family of complete convex embedded curves $\Gamma_t \subset \mathbb{R}^2$ the $\alpha$-curve shortening flow ($\alpha$-CSF) if the position vector $\mathbf{X}(\cdot,t)$ of $\Gamma_t$ satisfies 
\begin{align}\label{eq:main-flow}
    \tfrac{\partial }{\partial t}\mathbf{X}(p,t)=\kappa^{\alpha}(p,t)\mathbf{N}(p,t),
\end{align}
where $\kappa$ is the curvature and $\mathbf{N}$ is the inward pointing unit normal vector of $\Gamma_t$. In particular, if a flow $\Gamma_t$ exists for $t\in (-\infty,T)$ for some $T\in \mathbb{R}\cup \{+\infty\}$, then $\Gamma_t$ is an ancient flow.

Among the $\alpha$-curve shortening flows, the $\frac{1}{3}$-CSF is known as the affine normal flow, because the flow remains as a solution to \eqref{eq:main-flow} with $\alpha=\frac{1}{3}$ under volume-preserving affine transformations of $\mathbb{R}^2$. This affine-critical case $\alpha=\frac{1}{3}$ indeed plays a crucial role in studying ancient flows and singularities of the $\alpha$-CSFs, because their asymptotic behaviors dramatically change at the critical case. 

In the super-affine-critical case $\alpha >\frac{1}{3}$, Daskalopoulos-Hamilton-Sesum  showed  in  \cite{daskalopoulos2010} that a convex closed ancient CSF ($\alpha=1$) must be a shrinking circle or an Angenent oval. This result was extended for $\alpha \in (\frac{2}{3},1)$ in \cite{bourni2020ancient} by Bourni-Clutterbuck-Nguyen-Stancu-Wei-Wheeler. 

In the affine-critical case $\alpha =\frac{1}{3}$, Chen \cite{chen2015classifying} showed that an ancient closed convex affine normal flow must be a shrinking ellipse.  See an alternate proof by \citet{ivaki2016classification}. See also the higher dimensional analogue by Loftin and Tsui in \cite{loftin2008ancient}.  

Therefore, if a convex closed ancient flow with $\alpha=\frac{1}{3}$ or $\alpha \in (\frac{2}{3},1]$ converges to a closed shrinker after rescaling as $t\to -\infty$, then it is a self-shrinking flow. 

On the other hand,  the authors recently \cite{choisun} discovered infinitely many non-homothetic convex closed ancient flows with sub-affine-critical powers $\alpha <\frac{1}{3}$, which converge to closed shrinkers  as $t\to -\infty$  after rescaling.

\begin{theorem}[Choi-Sun \cite{choisun}]\label{thm:existence}
Given $\frac{1}{k^2-1}\leq \alpha <\frac{1}{(k-1)^2-1}$ with $3 \leq k\in \mathbb{N}$, there exist, up to rigid motions and dilation, a $(2k-5)$-parameter family of closed convex ancient $\alpha$-curve shortening flows converging to shrinking circles as $t\to -\infty$ and  a $(2m-4)$-parameter family of  closed convex ancient $\alpha$-curve shortening flows converging to shrinking $m$-fold symmetric flows\footnote{See Remark \ref{rmk:k_fold_sym} for the definition of the $m$-fold symmetry, and see Andrews \cite{andrews2003classification} for the classification of smooth strictly convex closed shrinkers to the $\alpha$-curve shortening flows for all $\alpha>0$. Readers are also encouraged to take a look at  the illustrations in the authors' previous paper \cite{choisun}.} as $t\to -\infty$, for each $m\in [3,k)\cap \mathbb{N}$.
\end{theorem}
 
The main goal of this paper is to show that ancient flows in Theorem \ref{thm:existence} are the only convex closed ancient $\alpha$-CSFs in the sub-affine-critical case.

\begin{theorem}\label{thm:subaffine.classification}
A convex closed ancient $\alpha$-curve shortening flow with $\alpha< \frac{1}{3}$ must be one of the flows listed in Theorem \ref{thm:existence}.  
\end{theorem}

 \bigskip
 
To prove Theorem \ref{thm:subaffine.classification}, we recall the entropy $\mathcal{E}_\alpha$ of the $\alpha$-CSF from Andrews-Guan-Ni \cite{andrews2016flow}. See Section 2 for its definition and more discussion. In the sub-critical case, it is important to use the fact in Proposition \ref{prop:ent-upper-bound} that a convex closed ancient $\alpha$-CSF $\Gamma_t$ with $\alpha\leq \frac{1}{3}$ has \textbf{finite entropy}, namely
\begin{equation}\label{eq:def-fin.ent}
\lim_{t\to -\infty}\mathcal{E}_{\alpha}(\Gamma_t)<\infty,
\end{equation}
which is not available for $\alpha >\frac{1}{3}$. In the super-critical case, we can establish an analogous theorem under the finite entropy condition.

\begin{theorem}\label{thm:superaffine.classification}
A convex closed smooth\footnote{By \cite{andrews1998evolving}, a weakly convex closed curve immediately becomes strictly convex and smooth under the $\alpha$-CSF with $\alpha\leq 1$. Hence, a convex closed ancient $\alpha$-CSF with $\alpha\leq 1$ is always strictly convex and smooth. However, for $\alpha>1$ there exists an ancient flow which is neither strictly convex nor smooth. See Figure \ref{fig:test2}. } finite entropy ancient $\alpha$-curve shortening flow with $ \alpha>\frac{1}{3}$ must be a shrinking circle.
\end{theorem} 

 \bigskip

We remind that Theorem \ref{thm:subaffine.classification} says that a closed ancient flow with $\alpha\leq \frac{1}{3}$ must converge to a shrinker as $t\to -\infty$ after rescaling. In the super-affine-critical case $\alpha\geq \frac{1}{3}$, a closed convex $\alpha$-CSF must converge to a round circle after rescaling at its singularity by the results in \cite{andrews1996contraction,andrews1998evolving,andrews2003classification,gage1986heat,sapiro1994affine}. (See the analogous theorem in higher dimensions  \cite{andrews1996contraction,andrews1999gauss,andrews2016flow,brendle2017asymptotic,choi-daskalopoulos16,guan2017entropy}.) We can describe this property shortly as follows: \textit{closed convex ancient $\alpha$-CSFs with $\alpha \leq \frac{1}{3}$ are Type I ancient flows, while closed convex $\alpha$-CSFs with $\alpha \geq \frac{1}{3}$ develop Type I singularities.} For readers' convenience, we recall the definition of Type I singularities and ancient flows.
\bigskip

\begin{definition}\label{def:TypeII}
Given a closed $\alpha$-curve shortening flow $\partial \Omega_t$, we denote by $\mathcal{A}(\Omega_t)$ and $\mathcal{P}(\Omega_t)$ the area  and the perimeter of the convex set $\Omega_t\subset \mathbb{R}^2$, respectively.  We say that an $\alpha$-CSF $\partial \Omega_t$  develops a \textit{Type II singularity} at the singular time $t=T$ if 
\begin{equation}
\limsup_{t\to T}\mathcal{P}(\Omega_t)^2/\mathcal{A}(\Omega_t)=+\infty.
\end{equation} 
Otherwise, we say that it develops a \textit{Type I singularity}. Similarly, we call an ancient flow $\partial\Omega_t$ a Type II ancient flow if 
\begin{equation}
\limsup_{t\to -\infty}\mathcal{P}(\Omega_t)^2/\mathcal{A}(\Omega_t)=+\infty.
\end{equation}
Otherwise, it is a Type I ancient flow.
\end{definition}

We will show in Proposition \ref{prop:iso-peri} that the isoperimetric ratio $\mathcal{P}^2(\Omega)/\mathcal{A}(\Omega)$ is bounded from above by the entropy $\mathcal{E}_\alpha(\Omega)$ with $\alpha >\frac{1}{3}$. Hence, we can restate Theorem \ref{thm:superaffine.classification} as follows.

\begin{corollary}\label{cor:TypeI.classification}
A convex closed smooth Type I ancient $\alpha$-curve shortening flow with $\alpha>\frac{1}{3}$ must be a shrinking circle.
\end{corollary} 

Notice that strictly convex closed  smooth Type II ancient $\alpha$-CSFs with $\alpha \in (\frac{1}{2},1]$ were discovered in \cite{angenent1992shrinking,bourni2020ancient}.  The Type II ancient flows in \cite{angenent1992shrinking,bourni2020ancient} are asymptotic to two parallel lines, and they converge to translators at their ends. 

For $\alpha\in (\frac{1}{3},\frac{1}{2}]$, it has been conjectured that there are strictly convex closed Type II ancient flows sweeping the entire plane, because the translators are entire graphs when $\alpha\in (\frac{1}{3},\frac{1}{2}]$. We remind that the translators with $\alpha >\frac{1}{2}$ are graphs defined on finite intervals.
 
 For $\alpha>1$,  the translators are not smooth. To be specific, the translators for $\alpha>1$ are of class $C^2$ and they have two flat sides as in Figure \ref{fig:test2}. See also \citet{U98GCFsoliton} and \cite[Remark 1.3]{choi2020uniqueness}. Therefore, we can simply obtain ancient flows with flat sides by gluing two translators as in  Figure \ref{fig:test2}, which look like paper clips. Clearly, these ancient paper clips with $\alpha>1$ are not smooth as the translators. In higher dimensions, the flows with flat sides have been studied to understand the shape of worn stones as a free boundary problem. See \cite{CDLflat,DH99GCFflat,DL04GCFflat,H93GCFflat,KLRflat}. 
 
\begin{figure}[ht]
\centering
  \includegraphics[width=0.7\linewidth]{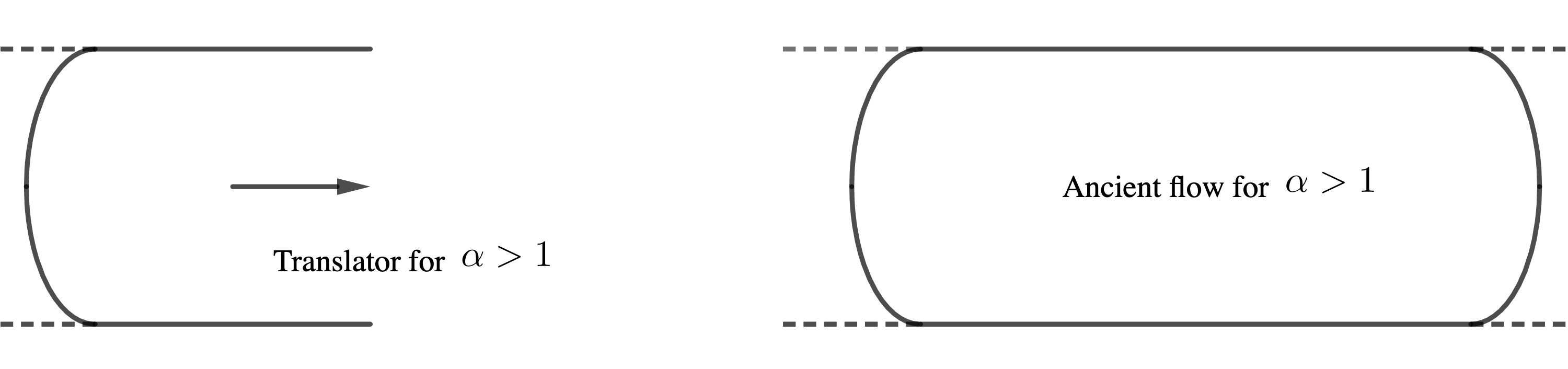}
   \caption{Ancient paper clips.}
  \label{fig:test2}
\end{figure}

\bigskip

\noindent\textbf{Open problems}

 By Theorem \ref{thm:subaffine.classification}, the convex closed ancient flows in the sub-affine-critical case are completely understood. Here, we summarize some conjectures for super-affine-critical powers.

\noindent \textit{Conjecture 1. There exists a one-parameter family (up to rigid motions and dilations) of strictly convex closed smooth Type II ancient $\alpha$-curve shortening flows with $\alpha \in (\frac{1}{3},\frac{1}{2}]$ sweeping the entire plane. Moreover, these flows are the only strictly convex closed smooth Type II ancient flows.}

\noindent \textit{Conjecture 2. The oval-shaped ancient flows in \cite{bourni2020ancient} are the only strictly convex closed smooth Type II ancient flows with $\alpha \in (\frac{1}{2},\frac{2}{3}]$.}

\noindent \textit{Conjecture 3. The ancient paper clips in Figure \ref{fig:test2} are the only convex closed Type II ancient flows with $\alpha >1$ in a weak sense. In particular, there is no smooth closed convex Type II ancient flow.}

 \bigskip

\noindent \textbf{Outline of the paper}

 In Section 2, we provide some preliminaries on  the entropy and the support function. In particular, we will show the sub-sequential convergence of rescaled ancient flows to shrinkers as the time goes to negative infinity. Moreover, the backward limit shrinker is unique up to rotations. In addition, we prove Theorem \ref{thm:superaffine.classification}.
 
 In Section 3, we consider the difference of $v$ between the support $\bar u$ of the rescaled ancient flow and the support $h$ of a nearby shrinker. The linearized operator $\mathcal{L}$ of the evolution equation of $v$ has the spectral decomposition, and therefore we can decompose the associated $L^2$-space with the norm $\|\cdot \|_h$ by the spaces spanned by unstable, neutral, and stable eigenfunctions. We consider the projection $P_-,P_*,P_+$ of $v$ into the unstable, neutral, stable spaces, respectively. Then, we will study basic dynamics of $\|P_-v\|_h,\|P_+v\|_h,\|P_*v\|_h $ by using the ODE method employed in recent researches of ancient geometric flows \cite{angenent2019unique,angenent2019unique-B,angenent2020uniqueness,brendle2018uniqueness,brendle2019uniqueness,chodosh2020mean,choi2018ancient,choi2019ancient,choi2019ancient-white}. 
  
  In Section 4,  we consider the flows converging to the round circle. The linearized operator has the kernel spanned by $\sin k\theta$ and $\cos k\theta$ if $\alpha=\frac{1}{k^2-1}$. We will study the dynamics of projections $A_k(\tau)=( v(\theta,\tau),\cos k\theta)_h$, $B_k(\tau)=( v(\theta,\tau),\sin k\theta)_h$, and some minor terms $A_{0}(\tau),A_{2k}(\tau),B_{2k}(\tau)$, to derive a contradiction for the case $\|P_-v\|_h+\|P_+v\|_h=o(\|P_*v\|_h)$ considered in Theorem \ref{thm:dichotomoy}. Hence, we obtain $\|P_-v\|_h+\|P_*v\|_h\leq C\|v\|_{C^4}\|P_+v\|_h$ so that we can show the exponential decay of $\|v(\cdot,\tau)\|_{h}$.

Notice that the linearized operators in previous researches of geometric flows \cite{angenent2019unique-B,angenent2019unique,choi2019ancient,choi2019ancient-white} have only one non-trivial neutral eigenfunction up to normalization factors. Hence, one can reduce the dynamics of projections to a single ODE. However, we deal with the dynamics of five projections in this paper.

In Section 5, we consider the ancient flows subsequentially converging to non-radial shrinkers. The kernel of the linearized operator is generated by rotations of the limit shrinker. Hence, by modifying the method of Allard-Almgren for minimal surfaces \cite{allard1981radial}, we can show the uniqueness of the tangent flow and the exponentially fast convergence in the rescaled time parameter $\tau$.
    
In Section 6,  we classify all convex closed ancient $\alpha$-curve shortening flows with $\alpha<\frac{1}{3}$ so that we prove Theorem \ref{thm:subaffine.classification}.

\bigskip

\noindent \textbf{Acknowledgements.} The authors are grateful to Shibing Chen, Beomjun Choi, Mohammad Ivaki, John Loftin, and  Christos Mantoulidis, for fruitful discussions and comments. K. Choi is supported by KIAS Individual Grant MG078901.

\bigskip

\section{Uniqueness of tangent flow at infinity}
\subsection{Preliminary}
Suppose $\Gamma$ is a convex closed curve. It is the boundary of a bounded convex domain $\Omega\subset \mathbb{R}^2$, namely $\Gamma=\partial \Omega$. The \textbf{support function} of $\Gamma$ in the direction $(\cos\theta,\sin\theta)$ with respect to some fixed point $z_0\in \mathbb{R}^2$ is defined by 
\begin{align}\label{eq:support-def}
    u_{z_0}(\theta):=\max_{z\in \Omega}\langle (\cos \theta,\sin\theta),z-z_0\rangle.
\end{align}
When $z_0$ is the origin, we write $u_{z_0}=u$ for the purpose of simplicity.

If $\Gamma$ is closed and differentiable, then each $p\in \Gamma$ has a unique inward pointing unit normal $\mathbf{N}(p)$. We say that $\theta$ the \textit{normal angle} of $p\in \Gamma$ if $\mathbf{N}(p)=-(\cos\theta,\sin\theta)$. Since $\Gamma$ is strictly convex, there is a one-to-one map between the point in $\Gamma$ and the normal angle module $2\pi$. Hence, we can parametrize $\Gamma$ by the normal angle, namely we denote the position vector by $\boldsymbol{X}(\theta)\in \Gamma$, the normal vector by $\mathbf{N}(\theta)=-(\cos \theta,\sin\theta)$, and the tangent vector by $\mathbf{T}(\theta)=(-\sin\theta,\cos\theta)$. Then, we can parametrize the \textbf{support function} of strictly convex closed differentiable curve $\Gamma$ as
\begin{align}\label{eq:support_position}
u (\theta)=-\langle  \mathbf{N}(\theta),\mathbf{X}(\theta) \rangle.
\end{align}
We can derive $u_{\theta} =\langle \mathbf{T}, \boldsymbol{X} \rangle$ from \eqref{eq:support_position}, and directly calculate
\begin{align}\label{eq:curvature-support}
    u_{\theta\theta}+u=\langle\tfrac{d}{d\theta}\mathbf{X} ,\mathbf{T}\rangle=\langle \tfrac{ds}{d\theta}\mathbf{T},\mathbf{T}\rangle=\kappa^{-1},
\end{align}
where $s$ is the arc-length parameter of $\Gamma$.

Now, we consider an evolution $\Gamma_t$ which moves by \eqref{eq:main-flow}. Then, by using \eqref{eq:curvature-support} we can obtain the evolution equation of the support function $u(\theta,t)$ of $\Gamma_t$ that
\begin{align}\label{eq:un-sppt-flow}
    u_t=-\kappa^\alpha=-(u_{\theta\theta}+u)^{-\alpha}.
\end{align}
Also, differentiating \eqref{eq:un-sppt-flow} by $t$ yields  the quasilinear  parabolic  equation 
\begin{equation}\label{eq:speed_evolution}
(\kappa^\alpha)_t= \alpha (\kappa^\alpha)^{1+\frac{1}{\alpha}} (\kappa^\alpha)_{\theta\theta} +\alpha (\kappa^\alpha)^{2+\frac{1}{\alpha}}
\end{equation}
of  the positive speed $\kappa^\alpha$.
 
Next, from \eqref{eq:un-sppt-flow} we can deduce that given a self-similarly shrinking flow  $\Gamma'_t$ there exist  $z_0\in \mathbb{R}^2$, $t_0\in \mathbb{R}$, and a static curve $\overline{\Gamma}'$ satisfying $\Gamma'_t=z_0+(t_0-t)^{-\frac{1}{\alpha+1}}(\overline{\Gamma}'-z_0)$. Therefore, given an ancient $\alpha$-CSF $\Gamma_t$,  we define the \textbf{rescaled flow} $\overline{\Gamma}_\tau$ by
\begin{align}\label{eq:un2nm2}
   \overline{\Gamma}_\tau=(1+\alpha)^{-\frac{1}{1+\alpha}}(-t)^{-\frac{1}{1+\alpha}}\Gamma_t,\quad \text{where}\,\;\tau=-\tfrac{1}{1+\alpha} \log(-t).
\end{align}
Then, the rescaled position vector $\bar{\boldsymbol{X}}=(1+\alpha)^{-\frac{1}{1+\alpha}}(-t)^{-\frac{1}{1+\alpha}}\boldsymbol{X}$ satisfies
\begin{align}\label{main-flow-nm2}
  \partial_\tau \bar{\boldsymbol{X}}  ={\bar \kappa^\alpha}\mathbf{N} +\bar {\boldsymbol{X}},
\end{align}
where $\bar \kappa(\cdot,\tau)=(1+\alpha)^{\frac{1}{1+\alpha}}(-t)^{\frac{1}{1+\alpha}}\kappa(\cdot,t)$ denotes the curvature of $\overline{\Gamma}_\tau$. Hence, we multiply the normal vector $-\mathbf{N}$ on both sides so that we obtain the evolution equation 
\begin{align}\label{eq:bar-u-tau}
    \bar u_\tau=-{\bar\kappa^\alpha}+\bar u
\end{align} 
of the   support function $\bar u(\cdot,\tau)$ of the rescaled flow $\overline{\Gamma}_\tau$. If the rescaled flow $\overline{\Gamma}_\tau$ is a static curve, then its support function $h(\theta)=\bar u(\theta,\tau)$ satisfies
\begin{align}\label{eq:h2}
   h_{\theta\theta}+h=h^{-1/\alpha}.
   \end{align}
   Hence, we say that a smooth strictly convex curve $\overline{\Gamma}$ is a \textbf{shrinker}  for the $\alpha$-curve shortening flow if its support function $h$ satisfies \eqref{eq:h2}.

\begin{remark}\label{rmk:k_fold_sym}
Since we consider embedded flows,    if a shrinker is closed, its support function $h$ is a $2\pi$-periodic solution to the ODE \eqref{eq:h2}. There, unless $h\equiv 1$, there exists some $k\in \mathbb{N}$ such that $2\pi/k$ is the fundamental period of $h$, and we say that the shrinker is \textbf{$k$-fold symmetric}. Also, if the fundamental period of a convex closed curve is  $2\pi/k$, then we say that the curve has the $k$-fold symmetry.  
\end{remark}

We recall an important classification result of Ben Andrews \cite{andrews2003classification}.

\begin{theorem}[Andrews  \cite{andrews2003classification}]\label{thm:class-ben}
Given a power  $\alpha \in (0,\frac{1}{8})$ and an integer $ k \in [3, \sqrt{1+1/\alpha} \,)$, the $\alpha$-curve shortening flow has a unique (up to rotation) closed $k$-fold symmetric shrinker embedded in $\mathbb{R}^2$.

If $0<  \alpha \neq \frac{1}{3}$, then a closed  embedded  shrinker for the $\alpha$-curve shortening flow must be the unit circle  or a $k$-fold symmetric shrinker with  $3\leq k< \sqrt{1+1/\alpha} $. 

If $\alpha=1/3$, then a closed embedded shrinker must be an ellipse.
\end{theorem}

\begin{definition}
We denote by $\Gamma_{k}^\alpha$ the $k$-fold symmetric shrinker for the $\alpha$-curve shortening flow whose support function attains its maximum at $\theta=0$.
\end{definition}

\begin{remark}
Notice that the unit circle is a shrinker. Hence, we frequently denote the unit circle by $\Gamma_{\infty}^\alpha$ to emphasize its role of the shrinker.
\end{remark}

Now, we recall  from  \cite{andrews2016flow}  the \textbf{entropy} functional $\mathcal{E}_\alpha(\Omega)$ of convex domains $\Omega$ defined by 
\begin{align}\label{def:ent}
    \mathcal{E}_\alpha(\Omega):=\sup_{z_0\in \Omega}\mathcal{E}_\alpha(\Omega,z_0)
\end{align}
where 
\begin{align}\label{def:ent-2}
\begin{split}
    \mathcal{E}_\alpha(\Omega,z_0)=\begin{dcases}
    \frac{\alpha}{\alpha-1}\log\left( \frac{1}{2\pi}\int_{\mathbb{S}^1}u_{z_0}^{1-\frac{1}{\alpha}}(\theta)d\theta\right)- \frac{1}{2}\log\frac{\mathcal{A}(\Omega)}{\pi}&\text{if }\alpha\neq 1,\\ \frac{1}{2\pi}\int_{\mathbb{S}^1}\log u_{z_0}(\theta)d\theta- \frac{1}{2}\log\frac{\mathcal{A}(\Omega)}{\pi}&\text{if }\alpha=1,
    \end{dcases}
    \end{split}
\end{align}
where $\mathcal{A}(\Omega)$ denotes the area of $\Omega$. 

\begin{remark}
 Here the definition of $\mathcal{E}_\alpha(\Omega,z_0)$ is a little bit different from that in \cite{andrews2016flow}, because we incorporate the term involving $\mathcal{A}(\Omega)$. By doing so, the entropy here is invariant under scaling, namely $\mathcal{E}_\alpha(\lambda\Omega)=\mathcal{E}_\alpha(\Omega)$ holds for any $\lambda>0$. 
\end{remark}

We remind that a closed convex curve $\Gamma$ is the boundary of a convex set $\Omega\subset \mathbb{R}^2$. Then, the following property explains the reason that we call $\mathcal{E}_\alpha$ the entropy.

\begin{proposition}[monotonicity]\label{prop:entropy-monotone}  Given   a smooth closed  $\alpha$-curve shortening flow   $\Gamma_t=\partial \Omega_t$, the entropy $\mathcal{E}_\alpha(\Omega_t)$ is strictly decreasing, unless $\partial \Omega_t$ is a self-similarly shrinking flow.
\end{proposition}

\begin{proof}
See   \cite{andrews1997monotone} or \cite[Theorem 3.1]{andrews2016flow}.
\end{proof}

We will frequently use the notation $\mathcal{E}_\alpha(\Gamma_t)$ to denote the entropy of the region bounded by the convex curve $\Gamma_t$, and define
\begin{align}
&\overline{\mathcal{E}}_\alpha(\Gamma_t)=\lim_{t\to-\infty}\mathcal{E}_\alpha(\Gamma_t) \in \mathbb{R}\cup \{+\infty\}, && \overline{\mathcal{E}}_\alpha(\Omega_t)=\lim_{t\to-\infty}\mathcal{E}_\alpha(\Omega_t) \in \mathbb{R}\cup \{+\infty\}.
\end{align}
As mentioned in the introduction, we say that a closed ancient  $\alpha$-curve shortening flow $\Gamma_t$  has \textbf{finite entropy} if \eqref{eq:def-fin.ent} holds, namely $\overline{\mathcal{E}}_\alpha(\Gamma_t) <+\infty$.

 \medskip
 
Indeed, any closed ancient flow with $\alpha <\frac{1}{3}$ has a  universal entropy bound. 
\begin{proposition}\label{prop:ent-upper-bound}  Given $\alpha \in (0,\frac{1}{3})$, there exists some constant $C_\alpha<+\infty$ such that 
\begin{equation}
\mathcal{E}_\alpha(\Omega)<C_\alpha,
\end{equation}
holds for every bounded convex open set $\Omega\subset \mathbb{R}^2$.
\end{proposition}

\begin{proof}
Since the entropy is invariant under scaling, we may assume  $\mathcal{A}(\Omega)=\pi$. Then, there exists a numeric constant  $C$ satisfying $\rho_-(\Omega) \leq C$, where $\rho_-(\Omega)$ is the inradius of $\Omega$. Hence,   \cite[Proposition 2.7. (ii)]{andrews2016flow} implies the desired result.
\end{proof}

\begin{remark}
By \cite[Proposition 2.1]{andrews2016flow}, we have $\mathcal{E}_\alpha(\Omega)\leq \mathcal{E}_{1/3}(\Omega)$ for $\alpha\leq \frac{1}{3}$. Then, we can show $\mathcal{E}_\alpha(\Omega)\leq \log{2}$  by using the John ellipse. However, Proposition \ref{prop:ent-upper-bound} is enough for this paper.   
\end{remark}

The following is an analogue of Proposition \ref{prop:ent-upper-bound} for $\alpha>\frac{1}{3}$.
\begin{proposition}[Corollary 2.2 \cite{andrews2016flow}]\label{prop:ent-low-bound}   $\mathcal{E}_\alpha(\Omega)\geq 0$ holds for $\alpha>1/3$. In addition, the equality holds if and only if $\Omega$ is a round disc.
\end{proposition}

 Next, we recall the \textbf{entropy point}.
\begin{proposition}\label{prop:ent-pt}
There is a unique point $z_e\in \Omega$ \rm (\em called \textbf{entropy point}\rm ) \em satisfying
\begin{equation}
\mathcal{E}_\alpha(\Omega)=\mathcal{E}_\alpha(\Omega,z_e).
\end{equation}  
 Also, the entropy and the entropy point are continuous. Namely, 
\begin{equation}
\lim_{d_H(\Omega,\Omega')\to 0}\left|\mathcal{E}_\alpha(\Omega)-\mathcal{E}_\alpha(\Omega')\right|+\left|z_e(\Omega)-z_e(\Omega')\right|=0
\end{equation}
holds, where $d_H(\Omega,\Omega')$ denotes the Hausdorff distance between $\Omega$ and $\Omega'$.
\end{proposition}

\begin{proof}
Lemma 2.5 and Lemma 4.3 in \cite{andrews2016flow}.
\end{proof}

\begin{proposition}[isoperimetric ratio]\label{prop:iso-peri}
Let $ \partial   \Omega_t$ be a   closed ancient \textbf{finite entropy} $\alpha$-curve shortening flow with $\alpha\neq \frac{1}{3}$. Then, there exist  some constant $C>0$   depending on $\alpha,\overline{\mathcal{E}}_\alpha(  \Omega_t)$ and  time $T\ll -1$ such that the circumradius $\rho_+( \Omega_t)$ and the inradius $\rho_-(  \Omega_t)$ of $\Omega_t$ satisfy
\begin{equation}\label{eq:radius-estimates}
C^{-1}\sqrt{\mathcal{A}(\Omega_t)}\leq \rho_-( \Omega_t)\leq \rho_+(  \Omega_t) \leq C\sqrt{\mathcal{A}(\Omega_t)} ,
\end{equation}
for $t\leq T$, and the entropy point $z_e( \Omega_t)$ with $t\leq T$ satisfies
\begin{equation}\label{eq:ent.po-distance}
\text{dist}(z_e( \Omega_t),\partial   \Omega_t)\geq C^{-1}\sqrt{\mathcal{A}(\Omega_t)}, 
\end{equation}
\end{proposition}

\begin{proof}
If  $\alpha>\frac{1}{3}$,   \cite[Proposition 2.7. (i)]{andrews2016flow} and Proposition \ref{prop:ent-low-bound} directly give \eqref{eq:radius-estimates}  for all $t$.  In the  case $\alpha<\frac{1}{3}$,  \cite[Proposition 2.7. (ii)]{andrews2016flow} yields a lower bound for $\frac{\rho_-}{\sqrt{\mathcal{A}}}$ for sufficiently negative $t$. Hence, we have an upper bound for $\frac{\rho_+}{\sqrt{\mathcal{A}}}$ .

 Next, we can obtain \eqref{eq:ent.po-distance} from \eqref{eq:radius-estimates} and  \cite[Lemma 4.4]{andrews2016flow}.
\end{proof}

\subsection{Backward convergence} In this subsection, we prove  the following theorem.

\begin{theorem}[unique backward shape]\label{thm:ent-limit}  Let $\overline{\Gamma}_\tau$ be a rescaled closed smooth ancient \textbf{finite entropy} $\alpha$-curve shortening flow with $\alpha\neq \frac{1}{3}$. Then, there exists a closed embedded shrinker $\Gamma^\alpha_k$  \rm(\em with $k\in \mathbb{N}\cup \{\infty\}$\rm) \em with the following significance:

 Given $ \varepsilon\in (0,1)$, there exists some rotation function $S(\tau):(-\infty,T]\to   SO(2)$ with $T\ll -1$  such that for each $\tau_0 \leq T$, the rotated flow $S(\tau_0)\overline{\Gamma}_\tau$  is $\varepsilon$-close to the static shrinker $\Gamma^\alpha_k$ in the $C^{[1/\varepsilon]}$-topology for $|\tau-\tau_0|\leq \varepsilon^{-1}$. 
\end{theorem} 

We notice that we can exclude the finite entropy condition when $\alpha<\frac{1}{3}$, thanks to Proposition \ref{prop:ent-upper-bound}. In addition, Theorem  \ref{thm:superaffine.classification}  immediately follows from Theorem \ref{thm:ent-limit}.
\begin{proof}[\textbf{Proof of Theorem \ref{thm:superaffine.classification}}]
By Theorem \ref{thm:ent-limit} and Theorem \ref{thm:class-ben}, the rescaled flow $\overline{\Gamma}_\tau$ converges to the unit circle $\Gamma^\alpha_\infty$ as $\tau\to -\infty$. Therefore, Proposition \ref{prop:entropy-monotone} implies
\begin{equation}
\mathcal{E}_\alpha(\Gamma_t) =\mathcal{E}_\alpha (\overline{\Gamma}_\tau)\leq \mathcal{E}_\alpha ( \Gamma_{\infty}^\alpha)=0, 
\end{equation}
for all $t$. Therefore, Proposition \ref{prop:ent-low-bound} yields the desired result.
\end{proof}

To begin with, given a bounded region $ \Omega $, we define a \textbf{normalized} region 
\begin{equation}\label{eq:normalized_region}
\hat \Omega = \lambda \Omega ,  \qquad \text{where}\quad \lambda=  \left[ \pi / \mathcal{A} ( \Omega ) \right]^{\frac{1}{2}},
\end{equation} 
which will be used only in this subsection. This normalization is useful, because we do not know yet if the area of the region enclosed by the rescaled flow $\overline{\Gamma}_\tau$ is bounded above and below. However, the normalized region $\hat\Omega_t$ always has the constant area $\mathcal{A}(\hat \Omega_t)= \pi$. 

Similarly, given a closed curve $\Gamma=\partial\Omega$ we define a \textbf{normalized} curve $\hat \Gamma=\partial \hat \Omega$.

\begin{proposition}\label{prop:area_div}
Any  closed ancient   $\alpha$-curve shortening flow  $ \partial   \Omega_t$ satisfies 
\begin{equation}\label{eq:area_divergence}
\lim_{t\to -\infty}\mathcal{A}(\Omega_t)\to +\infty.
\end{equation}
\end{proposition}
\begin{proof}
Since $\Omega_t$ monotonically shrinks, if \eqref{eq:area_divergence} fails, then there exists some large ball $B_R(0)$ including $\Omega_t$ for all $t$. However, $\partial B_{\rho(t)}(0)$ with $\rho(t)^{\alpha+1}=R^{\alpha+1}-(\alpha+1)(t-T)$ is an $\alpha$-CSF for any $T$, and we have $\Omega_T\subset B_{\rho(T)}(0)$. Thus, by the comparison principle, we have $\Omega_t \subset B_{\rho(t)}(0)$ for $t \geq T$, and therefore $\Omega_{t}=\emptyset$ for $t \geq T+(\alpha+1)^{-1}R^{\alpha+1}$. Hence, passing $T\to -\infty$ yields a contradiction.
\end{proof}

\medskip

\begin{lemma}[$C^0$ estimates]\label{lem:ent_pt_0}
Let $ \partial   \Omega_t$ be a smooth closed ancient \textbf{finite entropy} $\alpha$-curve shortening flow with $\alpha\neq \frac{1}{3}$. Then, there exist some constant $C$ depending on $\alpha,\overline{\mathcal{E}}_\alpha(\Omega_t)$ and  time $T\ll -1$ such that the support function $\hat u(\cdot,t)$ of the  \textbf{normalized} curve $\partial \hat \Omega_t$ satisfies 
\begin{equation}
  C \geq \hat u(\cdot,t)\geq C^{-1}>0,
 \end{equation}
for  $t\leq T$. Moreover, the entropy point $z_e(\hat\Omega_t)$ converges to the origin as $t\to -\infty$.
\end{lemma}

\begin{proof}
First of all, combining the upper bounds for $\rho_+$ in Proposition \ref{prop:iso-peri} and \eqref{eq:support_position} yields 
\begin{equation}\label{eq:supp_upper}
  \hat u(\cdot,t)\leq C.
 \end{equation}

Lower bounds can be obtained in a similar way to the proof of  \cite[Theorem 4.1]{andrews2016flow}. See also \cite{guan2017entropy}. To this end, we consider a decreasing sequence $\{t_i\}_{i\in \mathbb{N}}$ satisfying  $t_i\to -\infty$ and
\begin{align}\label{eq:Haus-conv}
\inf_{\mathbb{S}^1} \hat u(\cdot,t_i) =\liminf_{t\to -\infty} \min_{\mathbb{S}^1} u(\cdot,t).
\end{align}
Since we have \eqref{eq:supp_upper} and $\mathcal{A}(\hat \Omega_t)=\pi$, by the Blaschke selection theorem (cf.   \cite[Theorem 1.8.7]{schneider2014convex}),  there exists a subsequence $\{t_{i_m}\}$  such that $\{\hat \Omega_{t_{i_m}}\}$ converges to a convex body $\hat \Omega_*$ in the Hausdorff distance. We can replace $t_{i_m}$ by $t_i$ for simplicity so that $\{t_i\}$ satisfies \eqref{eq:Haus-conv} and $\hat \Omega_{t_{i}}\to \hat \Omega_*$. Then, the support function $\hat u_*$ of $\hat \Omega_*$ satisfies
\begin{equation}
\inf_{\mathbb{S}^1} \hat u_* =\liminf_{t\to -\infty} \min_{\mathbb{S}^1} u(\cdot,t),
\end{equation}
and thus it is enough to show that $dist(0,\partial\hat \Omega_*)\geq \varepsilon_0$   holds for some $\varepsilon_0>0$ depending on the entropy limit and $\alpha$. 

Now, we claim that the origin $0$ is the entropy point of $\hat \Omega_*$. If the claim is true, then Proposition \ref{prop:iso-peri} completes the proof. 

To prove the claim, for each  $i \in \mathbb{N}$ we define the shifted flow
\begin{equation}
\Omega_t^{i}=\Omega_t-z_e(  \Omega_{t_i})
\end{equation}
where $z_e(  \Omega_{t_i})$ denotes the entropy point  of $\Omega_{t_i}$. We consider the normalized curve $\partial\hat \Omega_t^{i}$ and its support function $\hat u^{i}(\cdot,t)$. Then,   Proposition \ref{prop:entropy-monotone} and  $\mathcal{A}(\hat \Omega^{i}_t)=\pi$  imply 
\begin{align}
\frac{\alpha}{\alpha-1}\log\fint_{\mathbb{S}^1}(\hat u^{i})^{1-\frac{1}{\alpha}}(\theta,t_j)d\theta=\mathcal{E}_\alpha(\hat  \Omega_{t_j}^i,0)=\mathcal{E}_\alpha(   \Omega_{t_j}^i,0) \geq  \mathcal{E}_\alpha(   \Omega_{t_i}^i,0)= \mathcal{E}_\alpha( \Omega_{t_i})
\end{align}
for $\alpha\neq 1$  and $t_j\leq t_i$.  Fixing any $i$, by Proposition \ref{prop:area_div}, the normalized shifted flows $\hat \Omega^i_{t_j}$ also converge to $\hat \Omega_*$ as $t_j\to -\infty$. Equivalently, $\lim_{j\to \infty}\hat u^{i}(\theta,t_j)=\hat u_*(\theta)$. 
\begin{align}
\lim_{j\to\infty}\left(\frac{\alpha}{\alpha-1}\log\fint_{\mathbb{S}^1}(\hat u^{i})^{1-\frac{1}{\alpha}}(\theta,t_j)d\theta\right)= \frac{\alpha}{\alpha-1}\log\fint_{\mathbb{S}^1}\hat u_*^{1-\frac{1}{\alpha}}(\theta)d\theta=\mathcal{E}_\alpha(\hat \Omega_*,0) 
\end{align}
holds for $\alpha>1$. If $\alpha<1$, the Fatou's lemma yields
\begin{align}
\liminf_{j\to\infty}\left(\log\fint_{\mathbb{S}^1}(\hat u^{i})^{1-\frac{1}{\alpha}}(\theta,t_j)d\theta\right)\geq   \log\fint_{\mathbb{S}^1}\hat u_*^{1-\frac{1}{\alpha}}(\theta)d\theta=\frac{\alpha-1}{\alpha }\mathcal{E}_\alpha(\hat \Omega_*,0).
\end{align}
Combining the inequalities above implies
\begin{equation}\label{eq:ent_pt_order}
\mathcal{E}_\alpha( \Omega_{t_i})\leq \mathcal{E}_\alpha(\hat \Omega_*,0)\leq \mathcal{E}_\alpha(\hat \Omega_*)=\lim\limits_{t\to -\infty}\mathcal{E}_\alpha(\hat\Omega_t)=\overline{\mathcal{E}}_\alpha(\Omega_t).
\end{equation}
When $\alpha=1$,  combining \eqref{eq:supp_upper} and the Fatou's lemma leads to
\begin{align}
\liminf_{j\to\infty}\left(\fint_{\mathbb{S}^1}-\log \hat u^{i} (\theta,t_j)d\theta\right)\geq   \fint_{\mathbb{S}^1}-\log \hat u_* (\theta )d\theta=-\mathcal{E}_1(\hat \Omega_*,0).
\end{align}
Therefore, by using Proposition \ref{prop:entropy-monotone} and \eqref{def:ent-2}, we also have \eqref{eq:ent_pt_order} for $\alpha=1$.

Hence, passing $t_i$ to $-\infty$ gives us $ \mathcal{E}_\alpha(\hat \Omega_*,0)= \mathcal{E}_\alpha(\hat \Omega_*)=\overline{\mathcal{E}}_\alpha(\Omega_t)$. Namely. the origin is the entropy point of $\hat \Omega_*$ as we claimed.
\end{proof}

\bigskip

\begin{lemma}[curvature estimates]\label{lem:ent_curvature}
Let $\partial     \Omega_t$ be a smooth closed ancient \textbf{finite entropy} $\alpha$-curve shortening flow with $\alpha\neq \frac{1}{3}$. Then, there exist some constants $C,\sigma>0$ depending on $\alpha$, $\overline{\mathcal{E}}_\alpha(  \Omega_t)$ and   $T\ll -1$ with the following significance: 

Given $t_0\leq T$, we consider an $\alpha$-curve shortening flow $\Omega_t^\mu:=\mu \Omega_{t_0+\mu^{-\alpha-1}t}$ with $\mu  = [\pi/\mathcal{A}(\Omega_{t_0})]^{\frac{1}{2}}$. Then, its support function $u^\mu(\cdot,t)$ and curvature $\kappa^\mu(\cdot,t)$ satisfy
\begin{align}
& C \geq   \kappa^\mu  \geq C^{-1} , && C \geq    u^\mu   \geq C^{-1} , && \text{for}\;\; t \in [-\sigma,\sigma].
\end{align}
\end{lemma}

\begin{proof}
The proof is similar to those of Theorem 5.1 and Theorem 5.2 in \cite{andrews2016flow}.

By $\mathcal{A}(\Omega^\mu_0)=\pi$ and Lemma \ref{lem:ent_pt_0}, for sufficiently negative $t_0$ there exist some $r_1,r_2>0$ depending on $\alpha,\overline{\mathcal{E}}_\alpha(  \Omega_t) $ such that  
\begin{align}
 r_1\leq   u^\mu (\cdot,0)\leq  \tfrac12 r_2.
\end{align}
Also, by Proposition \ref{prop:area_div} and Lemma \ref{lem:ent_pt_0}, there exists some $T_0<0$ such that
\begin{equation}
\min u^\mu(\cdot,T_0)=r_2.
\end{equation}
In addition, by Lemma \ref{lem:ent_pt_0}, there exists some $L$ depending on $\alpha,\overline{\mathcal{E}}_\alpha(  \Omega_t) $ such that
\begin{equation}
\max u^\mu(\cdot,T_0) \leq L \min u^\mu(\cdot,T_0)=L r_2=:r_3.
\end{equation}
See Figure \ref{fig:1} for illustration.

\begin{figure}[ht]
\centering
  \includegraphics[width=0.5\linewidth]{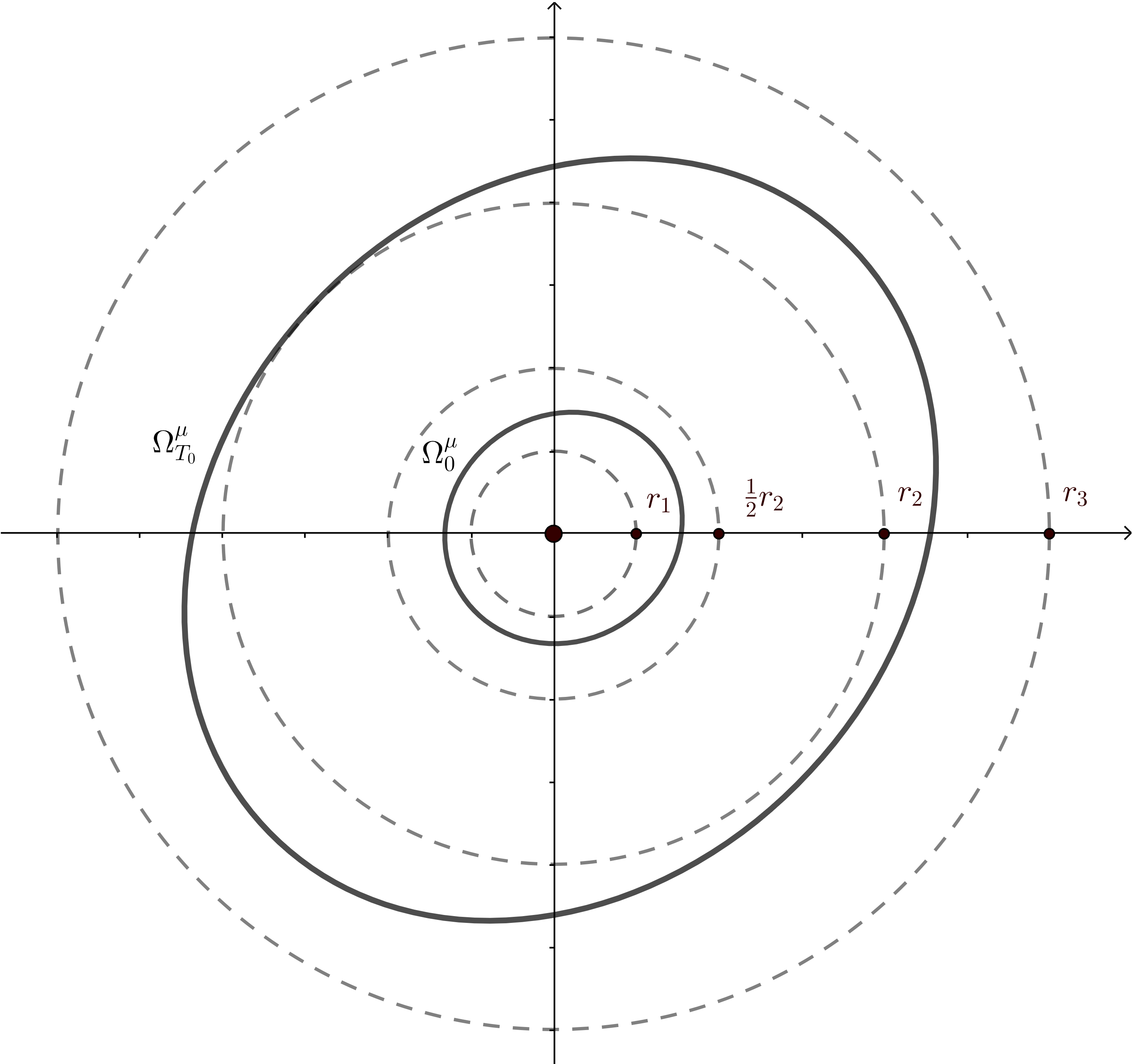}
   \caption{Four circles and time slices.}
  \label{fig:1}
\end{figure}

We observe that the shrinking circle $\partial B_{\rho_i(t)}(0)$ with $\rho_i(t)^{\alpha+1}= r_i^{\alpha+1}-(\alpha+1)t$ is the  $\alpha$-CSF for each $i=1,2,3$. Thus, by the comparison principle we have  $B_{\rho_1(t)}(0)\subset \Omega^\mu_t$ and   $   B_{\rho_2(t)}(0)\subset \Omega^\mu_{t+T_0} \subset B_{\rho_3(t)}(0)$ for $t\geq 0$.

Now, we define $\sigma_1>0$ by $\rho_1(2\sigma_1)=\frac{1}{2} r_1$. Then, $r_2 \geq 2r_1$ yield
\begin{equation}
\rho_2(2\sigma_1)^{\alpha+1}=r_2^{\alpha+1}-r_1^{\alpha+1}+(\tfrac12 r_1)^{\alpha+1} \geq 2 \left(\tfrac12 r_2\right)^{\alpha+1}-r_1^{\alpha+1} \geq  \left(\tfrac12 r_2\right)^{\alpha+1},
\end{equation}
namely $\Omega^\mu_0 \subset B_{\frac{1}{2} r_2}(0) \subset B_{\rho_2( 2\sigma_1 )}(0) \subset \Omega^\mu_{2\sigma_1 +T_0}$. Hence, we have 
\begin{align}
&T_0\leq -2\sigma_1, &&\tfrac12 r_1 \leq u^\mu(\cdot,t)\leq r_3,
\end{align}
for $t\in [   T_0, 2\sigma_1]$. Thus, by \cite[Theorem 6]{andrews2000motion} there exists some $C_1$ only depending on $\alpha,\overline{\mathcal{E}}_\alpha(\Omega_t)$ such that
\begin{equation}\label{eq:local_speed_control}
\kappa^\mu \leq C_1    
\end{equation}
holds for $t\in [-\sigma_1,2\sigma_1]$.

To derive lower bounds for $\kappa^\mu$, we first define the constant
\begin{equation}
    \sigma_2:=\min\{\sigma_1, \tfrac{1}{4} C_1^{-\alpha}r_2\},
\end{equation}
which only depends on $\alpha,\overline{\mathcal{E}}_\alpha(\Omega_t)$. Then, $\partial_t u^\mu=-(\kappa^{\mu})^\alpha$ and \eqref{eq:local_speed_control} imply 
\begin{equation}
    u^\mu(\theta,-\sigma_2)\leq \sigma_2 \max_{t\in [-\sigma_1,0]} |u^\mu_t(\theta,t)| +   u^\mu(\theta,0)\leq \sigma_2 C_1^\alpha+\tfrac12 r_2\leq \tfrac34 r_2
\end{equation}
for $\theta\in \mathbb{S}^1$. Thus, the Harnack inequality \cite[Theorem 14]{andrews2013contracting} yields
\begin{equation}
  (\kappa^\mu)^\alpha (\theta,t)\geq  \frac{u^\mu(\theta,T_0)-u^\mu(\theta,t)}{(1+\alpha)(t-T_0)}\geq    \frac{\frac{1}{4} r_2}{(1+\alpha)(2\sigma_1-T_0)}
\end{equation}
for $\theta\in \mathbb{S}^1$ and $t \in [-\sigma_2,2\sigma_1]$. Moreover, $B_{\rho_1(2\sigma_1)}(0)\subset  \Omega^\mu_{2\sigma_1} \subset  B_{\rho_3(2\sigma_1-T_0)}(0)$ leads to
\begin{equation}
 r_3^{\alpha+1}-(1+\alpha)(2\sigma_1-T_0)=\rho_3^{\alpha+1}(2\sigma_1-T_0)>0.
\end{equation}
Namely, $(\kappa^\mu)^\alpha \geq \frac{1}{4(1+\alpha)}r_2r_3^{-(\alpha+1)}$.  This completes the proof.
\end{proof}

\begin{proof}[Proof of Theorem \ref{thm:ent-limit}]
We consider the un-normalized flow $\Gamma_t=\partial\Omega_t$ and choose any sequence $\{t_i\}_{i=1}^{\infty}$ diverging to $-\infty$. We define  $\Omega_t^i:=\mu_i\Omega_{t_i+\mu_i^{-(\alpha+1)}t}$ where $\mu_i =[\pi/\mathcal{A}(\Omega_{t_i})]^{\frac{1}{2}}$. Then, by Lemma \ref{lem:ent_curvature} their support functions $u^i$ and curvatures $\kappa^i$ satisfy
\begin{align}
& C \geq   \kappa^i \geq C^{-1} , && C \geq    u^i  \geq C^{-1} , 
\end{align}
for $t\in [-\sigma,\sigma]$ and sufficiently large $i$, where $C,\sigma $ are some positive constants  only depending on $\alpha$ and the entropy limit.  Hence, the standard interior estimates (cf. \cite[Theorem 4.19]{wang1992regularity}) for the quasilinear equation \eqref{eq:speed_evolution} yield that given $\beta\in (0,1)$ there exists some constant  $C(\alpha,\overline{\mathcal{E}}_\alpha(  \Omega_t))$ satisfying
\begin{equation}
\|\kappa^i\|_{C^{\beta}(\mathbb{S}^1\times [0,\sigma])}\leq C.
\end{equation}
Since $\kappa^i=(u^i_{\theta\theta}+u^i)^{-1}$, the sequence $\{t_i\}$ has a subsequence $\{t_{i_m}\}$ such that $\Omega_t^{i_m}$ converges to $\Omega^*_t$ for $t\in [0,\sigma]$ in the $C^{2,\beta}$-topology of the support functions. Hence, $\partial \Omega^*_t$ is an $\alpha$-CSF of class $C^{2,\beta}$ for $t\in [0,\sigma]$. Moreover, we have $\mathcal{E}_\alpha(\Omega^*_s)=\overline{\mathcal{E}}_\alpha(  \Omega_t)$  for $s\in [0,\sigma]$, and therefore $\Omega^*_t$ is a self-similarly shrinking flow by Proposition \ref{prop:entropy-monotone}. Namely, 
\begin{equation}
\partial\Omega^*_t=z_*+(\alpha+1)^{\frac{1}{\alpha+1}}(t_*-t)^{\frac{1}{\alpha+1}}(S_*\Gamma_k^\alpha-z_*)
\end{equation}
for some $S_*\in SO(2)$, $z_*\in \mathbb{R}^2$, $t_* >\sigma$, and some closed embedded shrinker $ \Gamma^\alpha_k$ with $k\in \mathbb{N}\cup \{\infty\}$. Since we have $\mathcal{A}(\Omega^*_0)=\lim \mathcal{A}(\Omega^{i_m}_0)=\pi$, there is a constant $C^\alpha_k$ only depending on $\alpha$ and $\Gamma^\alpha_k$ such that $t_*=C^\alpha_k$. In addition, we have $z_*=0$, because the entropy point of $\Omega^*_t$ is the origin by Lemma \ref{lem:ent_pt_0}. Moreover, Theorem \ref{thm:ent-cmp} says that the shape of the shrinker is uniquely determined by the entropy $\mathcal{E}_\alpha(\Omega^*_0)=\overline{\mathcal{E}}_\alpha(  \Omega_t)$. Therefore, there exist rotations $S_{i_m}\in SO(2)$ such that 
 \begin{equation}\label{Siit}
\lim_{i_m\to +\infty}S_{i_m} \partial\Omega_t^{i_m} = (\alpha+1)^{\frac{1}{\alpha+1}}(C^\alpha_k-t)^{\frac{1}{\alpha+1}} \Gamma_k^\alpha,
\end{equation}
in the $C^{2,\beta}$-topology for $t\in [0,\sigma]$. Here using $\mathcal{A}(\Omega_0^{i_m})=\pi$, one gets \begin{align}\label{Cak=pi}
    (\alpha+1)^{\frac{2}{1+\alpha}}(C^\alpha_k)^{\frac{2}{1+\alpha}}\mathcal{A}(\Gamma_k^\alpha)=\pi.
\end{align}
Switching back to the original flow, as $i_m\to \infty$,
\begin{equation}\label{Sit-ak}
    S_{i_m}\partial  \Omega_t \to (\alpha+1)^{\frac{1}{\alpha+1}} [ C^\alpha_k\mu_{i_m}^{-(\alpha+1)} + t_{i_m}-t ]^{\frac{1}{\alpha+1}}\Gamma^\alpha_k,
\end{equation}
 holds for $t\in [t_{i_m},t_{i_m}+\mu_{i_m}^{-(\alpha+1)}\sigma]$. Since rotations preserve the area, we have
\begin{align}
       \mathcal{A}(\Omega_t)=& (\alpha+1)^{\frac{2}{\alpha+1}}[ C^\alpha_k\mu_{i_m}^{-(\alpha+1)} + t_{i_m}-t ]^{\frac{2}{\alpha+1}}\mathcal{A}(\Gamma^\alpha_k)+o(\mu_{i_m}^{-2}),\\       
       \tfrac{d}{dt}\mathcal{A}(\Omega_t)=&-2 (\alpha+1)^{-\frac{\alpha-1}{\alpha+1}}[ C^\alpha_k\mu_{i_m}^{-(\alpha+1)} + t_{i_m}-t ]^{-\frac{\alpha-1}{\alpha+1}}\mathcal{A}(\Gamma^\alpha_k)+o(\mu_{i_m}^{ \alpha-1}),
\end{align}
for $t\in [t_{i_m},t_{i_m}+\mu_{i_m}^{-(\alpha+1)}\sigma]$. Namely,
\begin{equation}\label{eq:area_derivative}
\tfrac{d}{dt}\mathcal{A}(\Omega_{t}) =-2\mathcal{A}(\Omega_{t})^{-\frac{\alpha-1}{2}}\left[\mathcal{A}(\Gamma^\alpha_k)^{\frac{\alpha+1}{2}}+o(1)\right] 
\end{equation}
for $t\in [t_{i_m},t_{i_m}+\mu_{i_m}^{-(\alpha+1)}\sigma]$, where $o(1)$ depends on $\mu_{i_m}^{-1}$.
Since $\mu_i\to +\infty$, we have \eqref{eq:area_derivative} for any subsequence $\{t_{i_m}\}$ with the convergent $\{\Omega_t^{i_m}\}$, and therefore we have \eqref{eq:area_divergence} for $t\in [t_i,t_i+\mu_i^{-(\alpha+1)}\sigma]$. Moreover, \eqref{eq:area_derivative} holds for any sequence $\{t_i\}$ diverging to $-\infty$, because $\lim_{t\to-\infty} \mathcal{A}(\Omega_t)=+\infty$. Hence
\begin{equation}
   \lim_{t\to -\infty}\tfrac{d}{dt}\mathcal{A}(\Omega_{t})^{\frac{\alpha+1}{2}}=-(\alpha+1)\mathcal{A}(\Gamma^\alpha_k)^{\frac{\alpha+1}{2}}. 
\end{equation}
Simple integration reveals that $\mathcal{A}(\Omega_t)=(1+\alpha)^{\frac{2}{\alpha+1}}\mathcal{A}(\Gamma^\alpha_k)|t|^{\frac{2}{1+\alpha}}+o(|t|^{\frac{2}{1+\alpha}})$ as $t\to -\infty$. Since $\mu_{i_m}=[\pi/\mathcal{A}(\Omega_{t_{i_m}})]^{1/2}$,  the identity \eqref{Cak=pi} implies 
\begin{equation}\label{Cak=-1}
    \lim_{i\to \infty} C^\alpha_k\mu_{i_m}^{-(\alpha+1)}/t_{i_m}=-1.
\end{equation}
Hence, by the rescaling in \eqref{eq:un2nm2}, the convergence \eqref{Sit-ak} infers that 
\begin{equation}
    S_{i_m}\partial \overline{\Omega}_\tau  \to  [e^{(1+\alpha)\tau}(C^\alpha_k\mu_{i_m}^{-(\alpha+1)}+t_{i_m}) +1]^{\frac{1}{\alpha+1}}\Gamma^\alpha_k
\end{equation}
for $\tau\in [-(1+\alpha)^{-1}\log (-t_{i_m}),-(1+\alpha)^{-1}\log(-t_{i_m}+\mu_{i_m}^{-(\alpha+1)}\sigma)]$.

 Hence, by \eqref{Cak=-1}, there exist  some constant $\bar \sigma>0$ depending on $\sigma,\alpha,k$ and rotations $  S(\tau)\in SO(2)$ such that $  S(\tau_0)\partial \overline{ \Omega}_\tau$ with $\tau\in [\tau_0,\tau_0+2\bar \sigma]$ converges to the static shrinker $\Gamma_k^\alpha$ as $\tau_0\to -\infty$. In addition, $S(\tau_0)\partial \overline{\Omega}_{\tau}$ gets closer and closer to $S(\tau_0+\bar \sigma)\partial \overline{\Omega}_{\tau}$ for $\tau \in [\tau_0+\bar \sigma,\tau_0+2\bar \sigma]$, namely we can choose $S(\tau)$ to satisfy $\displaystyle \lim_{\tau\to -\infty} \sup_{s\in [0,\bar \sigma]}|S(\tau)-S(\tau+s)|=0$. Since $|S(\tau)-S(\tau+m\bar \sigma)|\leq \sum_{l=1}^m|S(\tau+(l-1)\bar \sigma)-S(\tau+l \bar \sigma)|$, we have $\displaystyle \lim_{\tau\to -\infty} \sup_{s\in [0,K]}|S(\tau)-S(\tau+s)|=0$ for any fixed $K\gg 1$. Thus taking $\tau_0$ negative enough,  one achieves that the flow $  S(\tau_0)\partial \overline{ \Omega}_\tau$ in the fixed amount of time interval $|\tau-\tau_0|\leq 2\varepsilon^{-1}$ converges to the static shrinker $\Gamma^\alpha_k$ in the $C^{2,\beta}$-topology. Then, by using the standard higher order regularity theory for parabolic PDEs, one can easily deduce the desired result.
   \end{proof}

\medskip

\begin{remark}
We notice that the \L{}ojasiewicz-Simon inequality in \cite{andrews1997monotone} implies that the rotation $S(\tau)$ in Theorem \ref{thm:ent-limit} has a limit, namely our ancient flow has the unique tangent flow at backward infinity, but it does not guarantee the exponentially fast convergence. On the other hand, the Allard-Almgren's method works for the $\alpha$-CSFs, and it yields both of the uniqueness of the tangent flow and the fast convergence. Hence, we will use the Allard-Almgren's method in Section \ref{sec:non-radial}, not considering the \L{}ojasiewicz-Simon inequality. However, We would like to quickly explain how to use the inequality for the readers who maybe interested in.

In \cite[Proposition 20]{andrews1997monotone}, the notation there $\frac{1}{2}\log \mathcal{Z}_{1/\alpha}^{\alpha/(\alpha-1)}$ denotes the entropy $\mathcal{E}_\alpha$. In addition, the entropy is an analytic functional. Thus, one can obtain the \L{}ojasiewicz-Simon inequality  as \cite[Proposition 20]{andrews1997monotone}. Therefore \cite[Proposition 21]{andrews1997monotone} implies that  the rescaled flow will stay in a neighborhood of the limit shrinker for the sufficiently negative $\tau$.
\end{remark}

\bigskip

\section{Spectral dichotomy}\label{sec:dich}

In this section, we establish some lemmas which are needed in the following sections. For example, we will use Lemma \ref{lem:MZ_v} in Section \ref{sec:non-radial}. Also, we establish a dichotomy Theorem \ref{thm:dichotomoy} and consider each case in Section \ref{sec:radial} and in Section \ref{sec:classification}.

We recall that a rescaled flow $\overline{\Gamma}_\tau$ has the support function $\bar u(\theta,\tau)$ satisfying 
\begin{align} 
    \bar u_\tau=-{\bar\kappa^\alpha}+\bar u.
\end{align} 
Then, Theorem \ref{thm:ent-limit} implies that given $ \varepsilon \in (0,\frac{1}{100})$ there exists some  $T\ll -1$ such that each $\tau_0\leq T$ has some positive $h\in C^{\infty}(\mathbb{S}^1)$ satisfying
\begin{align}\label{eq:shrinker}
   h_{\theta\theta}+h=h^{-1/\alpha},
   \end{align}
   and 
   \begin{equation}\label{eq:local_conv}
\sup_{|\tau-\tau_0|\leq \varepsilon^{-1}}   \|v(\cdot,\tau)\|_{C^{20}(\mathbb{S}^1)} \leq \varepsilon
   \end{equation}
   where  $v=\bar u-h$. If the rotation $S(\tau)$ in Theorem \ref{thm:ent-limit} is a constant function or  $h\equiv 1$, then by taking $h$ as the support function of the limit shrinker we have
   \begin{equation}\label{eq:global_conv}
\lim_{ \tau\to -\infty}   \|v(\cdot,\tau)\|_{C^{20}(\mathbb{S}^1)} =0.
   \end{equation}   
   Furthermore, the difference $v$ satisfies
   \begin{align}\label{eq:v-flow}
   \begin{split}
   v_\tau=&-{(h_{\theta\theta}+h+v_{\theta\theta}+v)^{-\alpha}}+(h+v)\\
   =&-(h^{-\frac{1}{\alpha}}+v_{\theta\theta}+v)^{-\alpha}+h+v=:\mathcal{L}v+E (v),
   \end{split}
   \end{align}
where $\mathcal{L}$ is the linearized operator at $v=0$, namely
   \begin{align}\label{eq:L-def}
   \mathcal{L} v= \alpha h^{1+\frac{1}{\alpha}}({v_{\theta\theta}+v})+v.
  \end{align}   
We also use $\mathcal{L}_{\Gamma_k^\alpha}$ for clarification to denote the linearized operator above, when we consider the operators for different shrinkers.
  
Thus, the remaining quadratic error term is given by
\begin{align}\label{def:E}
    E(v)=-{(h^{-\frac{1}{\alpha}}+v_{\theta\theta}+v)^{-\alpha}}+h-\alpha h^{1+\frac{1}{\alpha}}(v_{\theta\theta}+v).
\end{align}

We now introduce the space  $L^2_h(\mathbb{S}^1)=\{f:\mathbb{S}^1\to \mathbb{R}:\|f\|_h^2<+\infty\}$, where $\|f\|_h^2=\int_{\mathbb{S}^1}f^2h^{-1-\frac{1}{\alpha}}$. The space is equipped with the inner product
   \begin{align}\label{eq:h-inner}
   (f,g)_h=\int_{\mathbb{S}^1}fgh^{-1-\frac{1}{\alpha}}d\theta.
   \end{align}
   Since $h>0$,  this norm is equivalent to the standard $L^2$ norm over $\mathbb{S}^1$. More importantly, $\mathcal{L}$ is a self-adjoint operator on $L^2_{h}$ over the compact space $\mathbb{S}^1$. Thus we have eigenfunctions $\{\varphi_j\}_{j \in \mathbb{N}}$ such that $\text{span}\{\varphi_1,\varphi_2,\cdots\}=L^2_h$, $(\varphi_i,\varphi_j)_h=\delta_{ij}$, and $\mathcal{L}\varphi_j+\lambda_j \varphi_j=0$, where the eigenvalues $\{\lambda_j\}_{j\in \mathbb{N}}$ satisfy  $\lambda_j \leq  \lambda_{j+1}$ and $\lim_{j\to +\infty}\lambda_j=+\infty$.

Given $f:\mathbb{S}^1\to \mathbb{R}$, we define projection operators $P_j$ by $P_jf=(f,\varphi_j)_h\varphi_j$ and
\begin{align}
&P_{<\lambda}=\sum_{\{j:\lambda_j<\lambda\}}P_j, && P_{=\lambda}=\sum_{\{j:\lambda_j=\lambda\}}P_j, && P_{>\lambda}=\sum_{\{j:\lambda_j>\lambda\}}P_j.
\end{align}
For $\lambda=0$, we will frequently use the following notations for the purpose of brevity. 
\begin{align}\label{def:proj}
&P_+=\sum_{\{j:\lambda_j<0\}}P_j, && P_*=\sum_{\{j:\lambda_j=0\}}P_j,&&   P_-=\sum_{\{j:\lambda_j>0\}}P_j.
\end{align}
We denote the projection to the negative eigensapce by $P_+$ instead of $P_-$, because $\|P_+v\|_h$ would increase in time due to its unstability. 

\bigskip

\begin{proposition}[non-vanishing]\label{prop:non-vanish}
Let $v$ be  a solution to \eqref{eq:v-flow} for some positive $h\in C^{\infty}(\mathbb{S}^1)$ satisfying \eqref{eq:shrinker}. Then, $\|v(\cdot,\tau)\|_{h}>0$ holds for every $\tau$ unless $v\equiv 0$. 
\end{proposition}

\begin{proof}
Suppose that $\|v(\cdot,\tau_0)\|_{h}=0$ at some $\tau_0\in \mathbb{R}$ and denote by $\overline{\Omega}_\tau$ the region enclosed by the rescaled flow $\overline{\Gamma}_\tau$. Since the entropy is invariant under rotations, $\overline{\Gamma}_{\tau_0}$ has the same entropy to the backward limit shrinker, namely 
\begin{equation}
\mathcal{E}_\alpha(\overline{\Omega}_{\tau_0})=\lim_{\tau\to -\infty}\mathcal{E}_\alpha(\overline{\Omega}_{\tau}).
\end{equation}
Therefore, Proposition \ref{prop:entropy-monotone} implies $v\equiv 0$.
\end{proof}

\begin{lemma}[error estimates]\label{lem:R<v}
Suppose that $v$ and $h$ are solutions   to \eqref{eq:v-flow} and  \eqref{eq:shrinker}, respectively, satisfying \eqref{eq:local_conv} for some $\tau_0\leq -1$ and $\varepsilon>0$.
Then, there exists some  $0<\varepsilon_0\ll 1$ depending on $\alpha,h$ such that if $\varepsilon\leq \varepsilon_0$ then for each $j\in \mathbb{N}$
\begin{align}
|(E,P_-v)_h|+\sum_{\lambda_j\leq 0} |(E,P_jv)_h|\leq C\|v\|_{C^4} \|v\|_h^2,
\end{align}
holds for $\tau\in [\tau_0-\varepsilon^{-1},\tau_0+\varepsilon^{-1}]$ and $E$ defined in \eqref{def:E}, where $C$ is some  constant depending on $\alpha, h$ and $\{\varphi_j:\lambda_j\leq 0\}$.
\end{lemma}

\begin{proof} We consider a function $f(x)$ defined by 
\begin{equation}\label{eq:error_f}
x^2f(x)=-(1+x)^{-\alpha}+1-\alpha x.
\end{equation}
By the Taylor expansion of $(1+x)^{-\alpha}$, $f(x)$ is smooth near $x=0$ . Hence, by the definition of $E$ in \eqref{eq:v-flow}, we get $E=h^{1+\frac{1}{\alpha}}(v+v_{\theta\theta})^2\rho$, where $\rho(v)=h^{\frac{1}{\alpha}}f(h^{\frac{1}{\alpha}}(v+v_{\theta\theta}))$ is a smooth positive function.  Since $v\to 0$ in the $C^{20}$-topology, there exists a constant $C$ only depending on $h,f$ such that $|\partial_\theta^m\rho|\leq C$ holds for $m=0,1,2,3$ and sufficiently negative $\tau$.

Now, by using $E=(v+v_{\theta\theta})^2h^{1+\frac{1}{\alpha}}\rho$ we can obtain
\begin{align}
\begin{split}\label{eq:evv}
(E,v)_h=&\iS (v_{\theta\theta}+v)^2v\rho d\theta=\iS v_{\theta\theta}^2v\rho+v^2(v+2v_{\theta\theta})\rho d\theta\\
=&\iS v_{\theta\theta}^2v\rho d\theta +O(1)\|v\|_{C^2}\|v\|_{h}^2.
\end{split}
\end{align}
Here $O(1)$ denotes a quantity bounded by some numeric constant $C$ for negative enough $\tau$. In addition,
\begin{align}
\iS v_{\theta\theta}^2v\rho d\theta=& -\iS v_{\theta\theta\theta}v_{\theta}v\rho+v_{\theta\theta}v_\theta^2\rho+v_{\theta\theta}v_{\theta}v\rho_{\theta}\\
=&-\iS \tfrac12v_{\theta\theta\theta}(\partial_\theta v^2)\rho+\tfrac13 \partial_\theta (v_{\theta}^3)\rho+\tfrac12v_{\theta\theta}(\partial_\theta v^2)\rho_\theta d\theta.
\end{align}
Hence, by using integration by parts we have
\begin{align}
\begin{split}
\iS v_{\theta\theta}^2v\rho d\theta=&\frac13\iS v_{\theta}^3\rho_\theta d\theta+O(1)\|v\|_{C^4}\|v\|_h^2\\
=&-\frac13\iS 2v_{\theta\theta}v_{\theta}v\rho_\theta+v_{\theta}^2v\rho_{\theta\theta}d\theta+O(1)\|v\|_{C^4}\|v\|_h^2\\
=&-\frac13\iS v_{\theta\theta}(\partial_\theta v^2)\rho_\theta+\tfrac12 v_{\theta}(\partial_\theta v^2)\rho_{\theta\theta}d\theta+O(1)\|v\|_{C^4}\|v\|_h^2\\
=&O(1)\|v\|_{C^4}\|v\|_h^2.
\end{split}
\end{align}
Combine the above inequalities,
\begin{align}\label{eq:evh}
|(E,v)_h|\leq C\|v\|_{C^4}\|v\|_h^2.
\end{align}
On the other hand, we have
\begin{align}
(E,P_jv)_h=\iS (v_{\theta\theta}+v)^2P_jv\rho d\theta=\iS v_{\theta\theta}^2 P_jv\rho d\theta+O(1)\|v\|_{C^2}\|v\|_h^2.
\end{align}
Moreover, using integration by parts
\begin{align}
\begin{split}
\iS v_{\theta\theta}^2 P_jv\rho d\theta=&-\iS v_{\theta\theta\theta}v_{\theta}(P_jv)\rho+v_{\theta\theta}v_{\theta}[(P_jv)\rho]_\theta\ d\theta\\
=&-\iS v_\theta[v_{\theta\theta\theta}(P_jv)\rho]+\tfrac12 \partial_\theta (v_\theta^2)[(P_jv)\rho]_\theta d\theta\\
=&\iS v[v_{\theta\theta\theta}(P_jv)\rho]_\theta+\tfrac12 v_\theta^2[(P_jv)\rho]_{\theta\theta}d\theta\\
=&O(1)\|v\|_{C^4}\|v\|_h\|P_jv\|_{H^1}+\iS\tfrac12 v_\theta^2[(P_jv)\rho]_{\theta\theta}d\theta.
\end{split}
\end{align}
For the second term, we have 
\begin{align}
\begin{split}
&\iS\tfrac12 v_\theta^2[(P_jv)\rho]_{\theta\theta}d\theta\\
=&\iS  v_\theta [\tfrac12 v_\theta(P_jv)_{\theta\theta} \rho+ v_\theta(P_jv)_{\theta} \rho_\theta+ \tfrac12 v_\theta(P_jv) \rho_{\theta\theta}] d\theta\\
=&-\iS  v  \left[\tfrac12 v_{\theta\theta}(P_jv)_{\theta\theta} \rho+\tfrac12 v_\theta(P_jv)_{\theta\theta\theta} \rho+\tfrac32 v_{\theta}(P_jv)_{\theta\theta} \rho_\theta+ v_{\theta\theta}(P_jv)_{\theta} \rho_{\theta}\right.\\
&\qquad\qquad \left.+\tfrac32 v_\theta(P_jv)_\theta \rho_{\theta\theta}+\tfrac12 v_{\theta\theta}(P_jv)\rho_{\theta\theta}+\tfrac12 v_\theta(P_jv)\rho_{\theta\theta\theta}\right] d\theta\\
=&\,O(1)\|v\|_{C^2}\|v\|_h\|P_jv\|_{H^3}.
\end{split}
\end{align}
Thus,
\begin{equation}\label{eq:P_error}
|(E,P_jv)_h|\leq C\|v\|_{C^4}\|v\|_h\|P_jv\|_{H^3}.
\end{equation}
We observe that $P_jv=(v,\varphi_j)_h\varphi_j$ gives $\partial_\theta^m P_jv=(v,\varphi_j)_h \partial_\theta^m \varphi_j$, and thus
\begin{align}
\|\partial_\theta^m P_jv\|_{L^2}^2\leq C_j \| P_jv\|_{L^2}^2,
\end{align}
where $C_j$ only depends on $m$ and  $ \varphi_j$. Therefore,
\begin{equation}
|(E,P_jv)_h| \leq C\|v\|_{C^4}\|v\|_{h}\|P_jv\|_{H^3}\leq C_j\|v\|_{C^4}\|v\|_{h}^2.
\end{equation}
Thus, combining \eqref{eq:evh} and the above estimates for all $\lambda_j\leq 0$ yields
\begin{equation}
|(E,P_-v)_h| \leq C\|v\|_{C^4}\|v\|_{h}^2.
\end{equation}
\end{proof}

In the special case $h\equiv 1$, we can leverage the previous method to achieve higher order estimates.  

\begin{lemma}[higher order error estimates]\label{lem:derivative_error}
Suppose that  an ancient solution  $v$ to \eqref{eq:v-flow} with $h\equiv 1$  converges to $0$ in the $C^{20}$-topology  as $\tau\to -\infty$.
Then, for each $m\leq \{0,\cdots,9\}$ there exists some negative enough $T$ such that the following holds for $\tau \leq T$.
\begin{align}
|(\partial_\theta^m E,P_-\partial_\theta^mv)_{L^2}|+\sum_{\lambda_j\leq 0 } |(\partial_\theta^m E,P_j\partial_\theta^mv)_{L^2}|\leq C\| v\|_{C^{m+4}} \|\partial_\theta^mv\|_{L^2},
\end{align}
where $C$ is some  constant depending on $\alpha,m$ and $\{\varphi_j:\lambda_j\leq 0\}$.
\end{lemma}
\begin{proof}
As in the proof of Lemma \ref{lem:R<v}, we can consider 
\begin{equation}
E(v)=(v+v_{\theta\theta})^2f(v+v_{\theta\theta}),
\end{equation}
where $f$ is a smooth function defined in a neighborhood of $0$, which is defined in \eqref{eq:error_f}. In addition, for each $m\leq 9$ we define a function $\rho_m$ of $v_{\theta\theta}+v$ by
\begin{equation}
\partial_\theta^m E(v)=(v+v_{\theta\theta})^2\rho_m.
\end{equation}
Then, there exist some constant $C$ and negative time $T$ such that
\begin{equation}
|\partial_\theta^l\rho_m| \leq C,
\end{equation}
holds for $\tau\leq T$, $m\in \{0,\cdots, 9\}$, and $l\in \{0,\cdots, 3\}$. 

Next, $h\equiv 1$ implies $\partial_\theta \mathcal{L} =\mathcal{L}\partial_\theta$ and thus we have
\begin{equation}\label{eq:evolution_derivatives}
\partial_\tau \partial_\theta^m v= \mathcal{L}  \partial_\theta^m v+ \partial_\theta^m E(v).
\end{equation}
Hence, as we obtained \eqref{eq:evh} in the proof of Lemma \ref{lem:R<v}, we can derive
\begin{equation}\label{eq:high-error-L2}
|(\partial_\theta^m E,\partial_\theta^m v)_{L^2}|\leq C\| v\|_{C^{m+4}}\|\partial_\theta^m v\|_{L^2}^2.
\end{equation}
 In addition, for each $j$ we can have
\begin{equation}
|(\partial_\theta^m E,P_j\partial_\theta^m v)_{L^2}|\leq C\| v\|_{C^{m+4}}\|\partial_\theta^m v\|_{L^2}\|P_j\partial_\theta^m v\|_{H^3}.
\end{equation}
Since $\|P_j\partial_\theta^m v\|_{H^3} \leq C_j\|\partial_\theta^m v\|_{L^2}$ for some constant $C_j$  depending on $\varphi_j$, we have
\begin{equation}\label{eq:high-error-decompose}
|(\partial_\theta^m E,P_j\partial_\theta^m v)_{L^2}|\leq C_j\| v\|_{C^{m+4}}\|\partial_\theta^m v\|_{L^2}^2.
\end{equation}
Therefore, combining \eqref{eq:high-error-L2} and \eqref{eq:high-error-decompose} completes the proof.
\end{proof}

\begin{lemma}\label{lem:MZ_v} We recall $v,h,\tau_0,\varepsilon$ in Lemma \ref{lem:R<v}. Then, there exist some constant $C,\varepsilon_0 $ depending on $h,\alpha$  such that if $\varepsilon\leq \varepsilon_0$ then
\begin{align}
 \tfrac{d}{d\tau}\|P_+v \|_{h}^2 &\geq -2\lambda_I \|P_+v \|_{h}^2-C\|v \|_{C^{4}} \|v \|_{h}^2, \label{eq:diff-ieq-P-}\\
 \left| \tfrac{d}{d\tau} \|P_*v\|_{h}^2\right|&\leq C\|v \|_{C^{4}}\|v \|_{h}^2,  \label{eq:V0-1}\\
     \tfrac{d}{d\tau} \|P_-v \|_{h}^2 &\leq  -2\lambda_{J} \|P_-v \|_{h}^2+C\|v \|_{C^{4}}\|v \|^2_{h},\label{eq:V-flow}
\end{align}
hold for $\tau \in [\tau_0-\varepsilon^{-1},\tau_0+\varepsilon^{-1}]$,  where $\lambda_I$ denotes the greatest negative eigenvalue and $\lambda_J$ denotes the least positive eigenvalue.
\end{lemma}

\begin{proof}
As in \cite[Lemma 5.5]{angenent2019unique}, by using Lemma \ref{lem:R<v} we have 
\begin{align}
\begin{split}
\tfrac12 \tfrac{d}{d\tau}\|P_+v\|_{h}^2=& (P_+v,(P_+v)_\tau)_h=(P_+v,v_\tau)_h= (P_+v,\mathcal{L} v+E)_h\\
\geq& -\lambda_I \|P_+v(\cdot,\tau)\|_{h}^2-C\|v\|_{C^4} \|v(\cdot,\tau)\|_{h}^2.
\end{split}
\end{align}
In the same manner, we can obtain the other inequalities.
\end{proof}

\begin{theorem}[spectral dichotomy]\label{thm:dichotomoy} Suppose that  an ancient solution  $v$ to \eqref{eq:v-flow} with $h\equiv 1$  converges to $0$ in the $C^{20}$-topology  as $\tau\to -\infty$, and $v$ is not identically zero. Then, there exist  some constant $C$ and negative enough $T$ such that either
\begin{align}\label{eq:neutral_dominate}
\|P_-   v\|_{C^8(\mathbb{S}^1)}+\|P_+   v\|_{C^8(\mathbb{S}^1)}+\left|\tfrac{d}{d\tau}\|P_* v\|_{L^2(\mathbb{S}^1)}\right| =o( \|P_* v\|_{L^2(\mathbb{S}^1)})  , 
\end{align}
or 
\begin{equation}\label{eq:unstable_dominate}
\|P_-   v\|_{L^2(\mathbb{S}^1)}+\|P_*   v\|_{L^2(\mathbb{S}^1)}\leq C \|v\|_{C^4(\mathbb{S}^1)} \|P_+ v\|_{L^2(\mathbb{S}^1)}
\end{equation}
holds for $\tau  \leq T$.
\end{theorem}

\begin{proof}
By the non-vanishing proposition \ref{prop:non-vanish}, we have $\|v\|_{L^2}>0$ for all $\tau$. In addition, we have the inequalities in Lemma \ref{lem:MZ_v} for $\tau\leq T$ for some negative enough $T\ll -1$. Therefore, by Lemma \ref{lem:MZ}, either \eqref{eq:unstable_dominate} or 
\begin{equation} \label{eq:neutral_dominate_I}
\|P_-   v\|_{L^2(\mathbb{S}^1)}+\|P_+   v\|_{L^2(\mathbb{S}^1)}= o(\|P_* v\|_{L^2(\mathbb{S}^1)}),
\end{equation}
holds. Suppose \eqref{eq:neutral_dominate_I} holds.  Then, we have $\|P_*v\|_{L^2}>0$ and thus Lemma \ref{lem:MZ_v} implies
\begin{align}
\left|\tfrac{d}{d\tau}\|P_* v\|_{L^2(\mathbb{S}^1)}\right| =o( \|P_* v\|_{L^2(\mathbb{S}^1)}). 
\end{align}

On the other hand, as the proof of Lemma \ref{lem:MZ_v}, \eqref{eq:evolution_derivatives} and  Lemma \ref{lem:derivative_error} give us
\begin{align}
 \tfrac{d}{d\tau}\|P_+\partial_\theta^m v \|_{L^2}^2 &\geq -2\lambda_I \|P_+\partial_\theta^m v \|_{L^2}^2-C\|v \|_{C^{m+4}} \|\partial_\theta^m v \|_{L^2}^2, \\
 \left| \tfrac{d}{d\tau} \|P_*\partial_\theta^m v\|_{L^2}^2\right|&\leq C\|v \|_{C^{m+4}}\| \partial_\theta^m v \|_{L^2}^2, \\
     \tfrac{d}{d\tau} \|P_- \partial_\theta^m v \|_{L^2}^2 &\leq  -2\lambda_{J} \|P_- \partial_\theta^m v \|_{L^2}^2+C\|v \|_{C^{m+4}}\|\partial_\theta^m v \|^2_{L^2},
\end{align}
for each $m\in \{0,\cdots,9\}$ and $\tau\leq T$, where $T\ll-1$ is some negative enough time.

Moreover, we have $\|\partial_\theta^m v \|^2_{L^2}>0$ for all $\tau$. If  $\|\partial_\theta^m v(\cdot,\tau_0) \|_{L^2}=0$ for some $\tau_0$, then $v(\cdot,\tau_0)$ is a constant and thus $\overline{\Gamma}_{\tau_0}$ is a circle. So, we can derive a contradiction by using the entropy as the proof of Proposition \ref{prop:non-vanish}. 

In addition, there exists some some constant $C$ depending on $m$ and $\{\varphi_j:\lambda_j<0\}$ such that $\|P_+ \partial_\theta^m  v\|_{L^2 } \leq C \|P_+  v\|_{L^2}$. Since $1\not \in \ker \mathcal{L}$, $C^{-1}\|P_*   v\|_{L^2 } \leq \|P_* \partial_\theta^m v\|_{L^2}\leq C\|P_*   v\|_{L^2 }$ also holds for some constant $C$ depending on $m$ and $\{\varphi_j:\lambda_j=0\}$. Hence,  \eqref{eq:neutral_dominate_I} implies 
\begin{equation} 
 \|P_+ \partial_\theta^m  v\|_{L^2(\mathbb{S}^1)}\leq C \|P_+    v\|_{L^2(\mathbb{S}^1)}= o(\|P_*  v\|_{L^2(\mathbb{S}^1)})=o(\|P_* \partial_\theta^m v\|_{L^2(\mathbb{S}^1)}).
\end{equation}
Therefore, Lemma \ref{lem:derivative_error} yields
\begin{align} 
\|P_-  \partial_\theta^m  v\|_{L^2(\mathbb{S}^1)}+\|P_+  \partial_\theta^m v\|_{L^2(\mathbb{S}^1)}=o( \|P_* \partial_\theta^m v\|_{L^2(\mathbb{S}^1)})=o( \|P_*  v\|_{L^2(\mathbb{S}^1)})
\end{align}
for $m\in \{0,\cdots,9\}$. Hence, combining the Sobolev inequality and \eqref{eq:neutral_dominate_I} leads to
\begin{align}
\|P_-   v\|_{C^8(\mathbb{S}^1)}+\|P_+   v\|_{C^8(\mathbb{S}^1)} =o( \|P_* v\|_{L^2(\mathbb{S}^1)}). 
\end{align}
This completes the proof.
\end{proof}

\section{Radial asymptotic behavior}\label{sec:radial}

In this section, we will show that a rescaled ancient flow with the radial asymptotic behavior converges exponentially fast to the unit circle.  We state the goal of this section below.

\begin{theorem}[fast convergence to the circle]\label{thm:non-exist2}
Suppose that a rescaled ancient solution $\overline{\Gamma}_\tau$ smoothly converges to the unit circle $\Gamma^\alpha_\infty$  as $\tau\to -\infty$. Then, there exist some $C>0$ and $T\ll -1$ such that \eqref{eq:unstable_dominate} and
\begin{equation}
\|v(\cdot,\tau)\|_{C^4(\mathbb{S}^1)}\leq C e^{-\lambda_I\tau}
\end{equation}
hold for $\tau \leq T$, where $1+v(\cdot,\tau)$ is the support function of  $\overline{\Gamma}_\tau$ and $\lambda_I$ is the greatest negative eigenvalue of the operator $\mathcal{L}=\alpha \partial_\theta^2+\alpha+1$.
\end{theorem}

We recall the evolution equation \eqref{eq:v-flow} of $v$ with $h\equiv 1$ that
\begin{equation}\label{eq:v-eq-round}
\partial_\tau v=\mathcal{L}v +E(v),
\end{equation}
where
\begin{equation}
\mathcal{L}v=\alpha v_{\theta\theta}+(\alpha+1) v,
\end{equation}
and
\begin{equation}\label{eq:Error_defintion}
E(v)=-(1+v_{\theta\theta}+v)^{-\alpha}+1-\alpha(v_{\theta\theta}+v).
\end{equation}

\bigskip

To show the exponential decay of $v$, we assume the neutral mode dominance \eqref{eq:neutral_dominate} in Theorem \ref{thm:dichotomoy}, towards a contradiction. First, we should have $\ker \mathcal{L}\neq \emptyset$ unless $v\equiv 0$. Thus, Theorem \ref{thm:sep-L} implies that there exists some integer $k\geq 3$ such that
\begin{equation}\label{eq:power_integer}
\alpha=1/(k^2-1).
\end{equation}
Next, we consider the Fourier expansion of  $v$.
\begin{align}\label{eq:v-fourier}
v(\theta,\tau)= A_0(\tau)+\sum_{m=1}^\infty A_m(\tau)\cos m\theta+B_m(\tau)\sin m\theta,
\end{align}
where 
\begin{align}
    A_0=\frac{1}{2\pi}\iS v(\theta,\tau)d\theta,\quad A_m=\frac{1}{\pi}\iS v\cos m\theta d\theta, B_m=\frac{1}{\pi}\iS v\sin m\theta d\theta.
\end{align}
We also define $\rho(\tau)$ and $Q(\tau)$ by
\begin{align}\label{eq:def_Q}
&\rho=A_k^2+B_k^2, && Q=(A_k^2-B_k^2)A_{2k}+2A_kB_kB_{2k}.
\end{align}

In order to exclude the trivial case $v\equiv 0$, we assume that
\begin{align}\label{eq:non-vanishment}
\lim_{\tau\to -\infty}\rho(\tau)=0,\qquad \text{but}\quad \rho(\tau)>0
\end{align}
holds for negative enough $\tau$. Moreover, to have the neutral mode dominance \eqref{eq:neutral_dominate}, we assume that
\begin{equation}\label{eq:magnitude_derivative}
\tfrac{d}{d\tau}\rho(\tau)=o(\rho),
\end{equation}
and
\begin{equation}\label{eq:neutral_dominate_II}
\|P_-   v\|_{C^8}+\|P_+  v\|_{C^8} =o( \rho^{\frac{1}{2}}). 
\end{equation}
We will verify in the proof of Theorem \ref{thm:non-exist2} that these assumptions are required for a contradiction argument.

For the convenience of notation, only in this section, we abuse the notation of the projection operators in \eqref{def:proj} as $P_0v=A_0$ and $P_mv=A_m\cos m\theta+B_m\sin m\theta$ for $m\in \mathbb{N}^*$. However, we will use the same definition of $P_+$ and $P_-$, namely
\begin{align}
&P_-v=\sum_{m>k}P_mv, && P_+v= \sum^{k-1}_{m=0}P_m v.
\end{align}
Notice that  $P_k$ denotes the projection to neutral space, namely $P_*=P_k$. Notice that $\|P_*v\|_{L^2}^2=\rho$.
Furthermore, we will relabel the eigenvalues of $\mathcal{L}_{\Gamma^\alpha_\infty}$ by 
\begin{equation}
\lambda_l=\alpha(l^2-1)-1,
\end{equation}
with $l\in \mathbb{N}^*\cup\{0\}$ so that $P_l$ denotes the projection to the eigenspace of $\lambda_l$.

In this section, we will show that $P_{2k}v$ and $P_0v$ dominate $P_-v$ and $P_+v$, respectively. Thus, we consider the following projections for convenience.
\begin{align}
&\tilde P_-v:=P_-v-P_{2k}v, && \tilde P_+v:=P_+v-P_0v.
\end{align}

We begin by observing the slow decay property of $\rho$.
\begin{proposition}[slow decay]\label{prop:non_exp.decay-rho}
If we have \eqref{eq:magnitude_derivative} and \eqref{eq:non-vanishment}, then  
\begin{equation}
\lim_{\tau\to -\infty}\rho(\tau) e^{-\delta\tau}=+\infty
\end{equation} 
holds for every $\delta>0$.
\end{proposition}

\begin{proof}
Combining \eqref{eq:magnitude_derivative} and \eqref{eq:non-vanishment} yields $\frac{d}{d\tau}\log \rho =o(1)$. Hence, given $\delta>0$ there exists some $T$ such that $\frac{d}{d\tau}\log \rho <\frac{\delta}{2}$ holds for $\tau \leq T$. Thus,
\begin{equation}
\log (\rho(\tau)e^{-\delta \tau})-C=-\int_{\tau}^T \frac{d}{ds} \log (\rho(s)e^{-\delta s}) ds\geq \int_{\tau}^T \frac{\delta}{2} \, ds,
\end{equation} 
yields the desired result.
\end{proof}

Next, we observe the error $E(v)$.

\begin{lemma}\label{lem:error_expansion}
Under the conditions \eqref{eq:power_integer} and \eqref{eq:neutral_dominate_II}, we have
\begin{align}\label{eq:tv2-2}
  E (v)=-\sum_{j=2}^n\binom{-\alpha}{j} (v_{\theta\theta}+v)^j +O( \rho^{\frac{n+1}{2}})
\end{align} 
for each $n\in \{2,3,4\}$, where 
\begin{equation}
\binom{-\alpha}{n}:= \frac{(-\alpha)(-\alpha-1)\cdots(-\alpha-n+1)}{n!}.
\end{equation}
\end{lemma}

\begin{proof}
Since $v=P_-v+P_kv+P_+v$, the condition \eqref{eq:neutral_dominate_II} yields
\begin{align}
v_{\theta\theta}+v=(\partial_\theta^2+1)P_kv+o(\rho^{\frac12}).
\end{align}
Hence, $(\partial_\theta^2+1)P_kv=-(k^2-1)P_kv=O(\rho^{\frac{1}{2}})$ implies $v_{\theta\theta}+v=O(\rho^{\frac{1}{2}})$. Thus, by using  the definition of $E$ in \eqref{eq:Error_defintion} we can obtain the desired result.
\end{proof}

\begin{lemma}\label{lem:higher_order_error-1}
Under the conditions \eqref{eq:power_integer} and \eqref{eq:neutral_dominate_II}, we have
\begin{align}
  \partial_\theta^m E (v)=-\tfrac12\alpha(\alpha+1)\partial_\theta^m F(\theta,\tau) +o( \rho),
\end{align}
for each $m\in \{ 0,\cdots,6\}$, where
\begin{equation}\label{def:F1}
F(\theta,\tau)=\frac{\rho(\tau)}{2\alpha^2}+\frac{A_k^2(\tau)-B_k^2(\tau)}{2\alpha^2}\cos 2k\theta+\frac{A_k(\tau)B_k(\tau)}{\alpha^2}\sin 2k\theta.
\end{equation}
\end{lemma}

\begin{proof}
The condition \eqref{eq:neutral_dominate_II} and the definition of $\rho$ in \eqref{eq:def_Q} yield
\begin{align}
\partial_\theta^m( v_{\theta\theta}+v)=\partial_\theta^m (\partial_\theta^2+1)P_kv+o(\rho^{\frac12})=O(\rho^{\frac{1}{2}}),
\end{align}
for each $m\leq 6$. Therefore, we have
\begin{align}
\begin{split}\label{ptmE}
\partial_\theta^mE&=-\tfrac{1}{2}\alpha(\alpha+1) \partial_\theta^m( v_{\theta\theta}+v)^2+o(\rho)\\
&=-\tfrac{1}{2}\alpha(\alpha+1)\partial_\theta^m [(\partial_\theta^2+1)P_kv]^2+o(\rho).
\end{split}
\end{align}
Since \eqref{eq:power_integer} implies
\begin{equation}
(\partial_\theta^2+1)P_kv=-\tfrac{1}{\alpha}[A_k\cos k\theta+B_k\sin k\theta],
\end{equation}
we have the desired result.
\end{proof}

\begin{lemma}\label{lem:rho-nonint}
Under the conditions \eqref{eq:power_integer}, \eqref{eq:non-vanishment}, \eqref{eq:magnitude_derivative}, and \eqref{eq:neutral_dominate_II}, we have
\begin{align}
&|A_0|+|A_{2k}|+|B_{2k}|=O(\rho) , &&     \|\tilde P_+v\|_{C^5}+\|\tilde P_-v\|_{C^5}=o(\rho).
\end{align}
\end{lemma}
\begin{proof} 
By using \eqref{eq:v-eq-round} and  Lemma \ref{lem:error_expansion}, we have 
\begin{equation}
\frac{d}{d\tau}A_0= \frac{1}{2\pi}\int v_\tau d\theta =-\lambda_0A_0+O(\rho).
\end{equation}
Thus, Proposition \ref{prop:non_exp.decay-rho} and Lemma \ref{lem:ODE-2} with $\tilde{\rho}=\rho$ imply $|A_0|=O(\rho)$. In the same manner, we have $|A_{2k}|+|B_{2k}|=O( \rho)$.

Next, for each $m\in \{1,\cdots, 6\}$  we calculate
\begin{align}
\tfrac{1}{2}\tfrac{d}{d\tau}\|\partial_\theta^m \tilde P_-v\|^2_{L^2}=&(\partial_\theta^m \tilde P_-v,\partial_\theta^m  v_\tau)_{L^2}=(\partial_\theta^m \tilde P_-v,\mathcal{L} \partial_\theta^m  v+\partial_\theta^m  E (v))_{L^2} \notag\\
\leq& -\lambda_{k+1} \|\partial_\theta^m \tilde P_-v \|^2_{L^2}+\|\partial_\theta^m \tilde P_-v\|_{L^2}\|  \tilde P_-\partial_\theta^m E\|_{L^2}.\label{ptmP-}
\end{align}
Since Lemma \ref{lem:higher_order_error-1} implies $\tilde{P}_- \partial_\theta^m E(v)=o(\rho)$, by using $\lambda_{k+1}>0$ we have
\begin{align}\label{eq:tildeP-}
 \tfrac{1}{\lambda_{k+1}} \tfrac{d}{d\tau}\|\partial_\theta^m \tilde P_-v\|_{L^2}\leq - \|\partial_\theta^m \tilde P_-v \|_{L^2}+o(\rho),
\end{align}
almost everywhere in $\tau$. Therefore, Proposition \ref{prop:non_exp.decay-rho} and Lemma \ref{lem:ODE-1}  with $\tilde{\rho}=\rho$   imply $\|\partial_\theta^m\tilde P_-v\|_{L^2}\leq o(\rho)$. Since $\iS\partial_\theta^{m-1}\tilde P_- vd\theta=0$, the Sobolev embedding theorem yields 
\begin{equation}
\|\tilde{P}_-v\|_{C^5}=o(\rho).
\end{equation}
Similarly, we compute
\begin{align}
\tfrac{1}{2}\tfrac{d}{d\tau}\|\partial_\theta^m \tilde P_+v\|^2_{L^2} \geq  -\lambda_{k-1} \|\partial_\theta^m \tilde P_+v \|^2_{L^2}-\|\partial_\theta^m \tilde P_+v\|_{L^2}\|  \partial_\theta^m E\|_{L^2},
\end{align}
and then by using $\lambda_{k-1}<0$ we have
\begin{align}\label{eq:tP+v}
\tfrac{1}{\lambda_{k-1}} \tfrac{d}{d\tau} \|\partial_\theta^m\tilde P_+v\|_{L^2} \leq - \|\partial_\theta^m\tilde P_+ v\|_{L^2}+o(\rho).
\end{align}
Hence, Proposition \ref{prop:non_exp.decay-rho} and Lemma \ref{lem:ODE-1} with $\tilde{\rho}=\rho$ imply  $\|\partial_\theta^m\tilde P_+v\|_{L^2}=o(\rho)$. Thus, 
\begin{equation}
\|\tilde{P}_+v\|_{C^5}=o(\rho).
\end{equation}
This completes the proof.
\end{proof}

Now we can improve the estimates in Lemma \ref{lem:higher_order_error-1}.
\begin{lemma}\label{lem:higher_order_error}
Under the conditions \eqref{eq:power_integer}, \eqref{eq:non-vanishment}, \eqref{eq:magnitude_derivative}, and \eqref{eq:neutral_dominate_II}, we have
\begin{align}\label{pmtE3/2}
  \partial_\theta^m E (v)=-\tfrac12\alpha(\alpha+1)\partial_\theta^m F(\theta,\tau) +O( \rho^{\frac{3}{2}}),
\end{align}
for each $m\in \{ 0,1,2,3\}$, where $F$ is defined in \eqref{def:F1}.
\end{lemma}

\begin{proof}By using Lemma \ref{lem:rho-nonint}, as in the proof of Lemma \ref{lem:higher_order_error-1} we can obtain
\begin{equation}
\partial_\theta^m( v_{\theta\theta}+v)=-\tfrac{1}{\alpha}\partial_\theta^m [A_k\cos k\theta+B_k\sin k\theta]+O(\rho),
\end{equation}
for each $m\in \{0,1,2,3\}$. Thus, by using \eqref{eq:tv2-2}, we can derive the desired result, as done in \eqref{ptmE}.
\end{proof}

\begin{lemma}\label{lem:P-+}
Under the conditions \eqref{eq:power_integer}, \eqref{eq:non-vanishment}, \eqref{eq:magnitude_derivative}, and \eqref{eq:neutral_dominate_II}, we have
\begin{align}\label{eq:rho1}
\|\tilde P_+v\|_{C^2}+\|\tilde P_-v\|_{C^2}=O(\rho^{\frac{3}{2}}).
\end{align}
\end{lemma}
\begin{proof}
We repeat the computation in the proof of Lemma \ref{lem:rho-nonint} by using \eqref{pmtE3/2} to achieve
\begin{align}
     \tfrac{1}{\lambda_{k+1}} \tfrac{d}{d\tau}\|\partial_\theta^m \tilde P_-v\|_{L^2}\leq  -\|\partial_\theta^m \tilde P_-v \|_{L^2}+O(\rho^{3/2}),\\
     \tfrac{1}{\lambda_{k-1}} \tfrac{d}{d\tau} \|\partial_\theta^m\tilde P_+v\|_{L^2} \leq  -\|\partial_\theta^m\tilde P_+ v\|_{L^2}+O(\rho^{3/2}).
\end{align}
Therefore, Proposition \ref{prop:non_exp.decay-rho} and Lemma \ref{lem:ODE-basic} with $\tilde{\rho}=\rho^{3/2}$ yield  the desired result.
\end{proof}
\begin{proposition}
Under the conditions \eqref{eq:power_integer}, \eqref{eq:non-vanishment}, \eqref{eq:magnitude_derivative}, and \eqref{eq:neutral_dominate_II}, we have
\begin{align}
\tfrac{d}{d\tau} A_0=&-\lambda_0 A_0-\tfrac{\alpha+1}{4\alpha}\rho+O(\rho^2),\label{eq:2A0}\\
\tfrac{d}{d\tau} A_{2k}=&-\lambda_{2k}A_{2k}-\tfrac{\alpha+1}{4\alpha}(A_k^2-B_k^2)+O(\rho^2),\label{eq:2A2k}\\
\tfrac{d}{d\tau} B_{2k}=&-\lambda_{2k}B_{2k}-\tfrac{\alpha+1}{2\alpha }A_kB_k+O(\rho^2),\label{eq:2B2k}
\end{align}
and
\begin{align}
  \tfrac{1}{1+\alpha}\tfrac{d}{d\tau} A_k= &   (A_0-\tfrac{ 2+\alpha   }{8\alpha^2}\rho  )A_k-\tfrac{  4  +3\alpha}{2\alpha}[A_kA_{2k}+B_kB_{2k}] +O(\rho^{\frac52}),\label{eq:Ak}\\
  \tfrac{1}{1+\alpha}  \tfrac{d}{d\tau} B_k= &    (A_0-\tfrac{ 2+\alpha   }{8\alpha^2}\rho  ) B_k-\tfrac{  4+3\alpha }{2\alpha}[A_kB_{2k}-B_kA_{2k}] +O(\rho^{\frac52}).\label{eq:Bk}
\end{align}
\end{proposition}

\begin{proof}
We define $I,J,K,V$ by
\begin{align}
&V=v_{\theta\theta}+v, && I=P_{k}V, && J=(P_0+P_{2k})V, && K=(\tilde{P}_+ +\tilde{P}_-)V.
\end{align}
Then, Lemma \ref{lem:rho-nonint}  and Lemma \ref{lem:P-+} yield
\begin{align}
I&= -\tfrac{1}{\alpha} [A_k\cos k\theta+B_k\sin k\theta]=O(\rho^{\frac{1}{2}}),\\
J&=A_0-\tfrac{4+3\alpha}{\alpha} [A_{2k}\cos 2k\theta+B_{2k}\sin 2k\theta]=O(\rho),\\
K&=(\partial_\theta^2+1)(\tilde{P}_+ v +\tilde{P}_-  v)=O(\rho^{\frac{3}{2}}).
\end{align}
Notice that we used  $k^2=\alpha^{-1}+1$.  Thus, $V=I+J+K$ yields
\begin{equation}
\frac{1}{2\pi}\iS V^2d\theta=\frac{1}{2\pi}\iS (I^2+2IJ)d\theta+O(\rho^2)=\frac{\rho}{2\alpha^2}+O(\rho^2),
\end{equation}
and
\begin{equation}
\frac{1}{2\pi}\iS V^3 d\theta=\frac{1}{2\pi}\iS I^3d\theta+O(\rho^2)=O(\rho^2).
\end{equation}
Since $\frac{d}{d\tau} A_0(\tau)=\frac{1}{2\pi}\iS v_\tau d\theta$, combining the identities above with Lemma \ref{lem:error_expansion}, we get 
\begin{align}
\frac{d}{d\tau} A_0=-\lambda_0A_0+ \frac{1}{2\pi}\iS  E d\theta =-\lambda_0 A_0-\binom{-\alpha}{2}\frac{\rho}{2\alpha^2}+O(\rho^2)
\end{align}
which yields the first equation \eqref{eq:2A0}.

Similarly, we can calculate
\begin{align}
\frac{1}{\pi}\iS V^2\cos 2k\theta d\theta &=\frac{A_k^2-B_k^2}{2\alpha^2}+O(\rho^2),\\
\frac{1}{\pi}\iS V^2\sin 2k\theta d\theta &= \frac{A_k B_k}{\alpha^2}+O(\rho^2),
\end{align}
and
\begin{equation}
 \iS V^3\cos 2k\theta d\theta= \iS V^3\sin 2k\theta d\theta=O(\rho^2).
\end{equation}
Thus, we obtain  \eqref{eq:2A2k} and \eqref{eq:2B2k} in the same manner.

\bigskip

To derive equation \eqref{eq:Ak}, we observe
\begin{equation}
\iS I^2\cos k\theta d\theta=\iS J^2\cos k\theta d\theta=\iS IK\cos k\theta d\theta=0.
\end{equation}
We notice that we obtain the last identity above by using 
\begin{equation}
I\cos k\theta \in \text{span}\{1,\cos 2k\theta,\sin 2k\theta\}.
\end{equation}
Therefore, we have
\begin{align}
\begin{split}
\frac1\pi\iS V^2\cos k\theta d\theta&=\frac2\pi\iS IJ\cos k\theta d\theta+O(\rho^{\frac52})\\
&=-\tfrac{2}{\alpha}  A_0A_k+\tfrac{ 4+3\alpha}{\alpha^2} [A_kA_{2k}+B_kB_{2k}]+O(\rho^{\frac52}).
\end{split}
\end{align}
Also, we can obtain
\begin{align}
\frac1\pi\iS V^3\cos k\theta d\theta &=\frac1\pi\iS(I^3+3I^2J)\cos k\theta d\theta+O(\rho^{\frac{5}{2}})=-\frac{3}{4\alpha^3} A_k\rho +O(\rho^{\frac52}),
\end{align}
and
\begin{align}
 \iS V^4\cos k\theta d\theta = \iS I^4\cos k\theta d\theta+O(\rho^{\frac{5}{2}})=O(\rho^{\frac{5}{2}}).
\end{align}
Thus, combining the identities above with Lemma \ref{lem:error_expansion} implies \eqref{eq:Ak}. Similarly, we can obtain
\begin{align}
 \frac1\pi\iS V^2\sin k\theta d\theta &=  -\tfrac{2}{\alpha}  A_0B_k+\tfrac{ 4+3\alpha}{\alpha^2} [A_kB_{2k}-B_kA_{2k}]+O(\rho^{\frac52}), \notag\\
 \frac1\pi\iS V^3\sin k\theta d\theta &=-\tfrac{3}{4\alpha^3} B_k\rho +O(\rho^{\frac52}),\\
 \iS V^4\sin k\theta d\theta &=O(\rho^{\frac{5}{2}}). \notag
\end{align}
Therefore, we have the last equation \eqref{eq:Bk}.
\end{proof}

\bigskip

\begin{proposition}
Under the conditions \eqref{eq:power_integer}, \eqref{eq:non-vanishment}, \eqref{eq:magnitude_derivative}, and \eqref{eq:neutral_dominate_II}, we have
\begin{align}
\tfrac{1}{1+\alpha} \tfrac{d}{d\tau} \rho=(2A_0-\tfrac{2+\alpha }{4\alpha^2}\rho)\rho- \tfrac{4+3\alpha}{\alpha} Q +O(\rho^3)\label{eq:rho}
\end{align}
and the quantity $Q=(A_k^2-B_k^2)A_{2k}+2A_kB_kB_{2k}$ in \eqref{eq:def_Q} satisfies 
\begin{align}\label{eq:Q}
\tfrac{1}{1+\alpha} \tfrac{d}{d\tau} Q=(2A_0-\tfrac{ \alpha+2 }{4\alpha^2}\rho) Q-  \tfrac{4+3\alpha}{\alpha}\rho[A_{2k}^2+B_{2k}^2]-\tfrac{\lambda_{2k}}{1+\alpha} Q-\tfrac{ 1}{4\alpha}\rho^2+O(\rho^3 ).
\end{align}
\end{proposition}

\begin{proof}
For the purpose of brevity, we define the operator
\begin{equation}
D_m=\tfrac{1}{1+\alpha}\tfrac{d}{d\tau}-m(A_0-\tfrac{ 2+\alpha   }{8\alpha^2}\rho  )
\end{equation}
for $m=1,2$. Then, we can rewrite \eqref{eq:Ak} and \eqref{eq:Bk} as follows.
\begin{align}
  D_1 A_k= &    -\tfrac{  4+3\alpha }{2 \alpha}[A_kA_{2k}+B_kB_{2k}] +O(\rho^{\frac52}), \label{eq:A_k-re}\\
  D_1 B_k= &    -\tfrac{  4+3\alpha }{2\alpha}[A_kB_{2k}-B_kA_{2k}] +O(\rho^{\frac52}). \label{eq:B_k-re}
\end{align}
Therefore, we can obtain the first equation \eqref{eq:rho} from the following
\begin{align}
D_2\rho= D_2 (A_k^2+B_k^2)= - \tfrac{4+3\alpha}{\alpha}Q +O(\rho^3). 
\end{align}

Next, by using \eqref{eq:A_k-re} and \eqref{eq:B_k-re}, we have
\begin{align}
D_2(A_k^2-B_k^2)&= -\tfrac{4+3\alpha}{\alpha}\rho A_{2k} +O(\rho^3),\\
D_2(A_kB_k)&= -\tfrac{4+3\alpha}{2\alpha}\rho B_{2k}+O(\rho^3).
\end{align}
Thus, combing the equations above with \eqref{eq:2A2k} and \eqref{eq:2B2k} yields 
\begin{align}
D_2 Q=-\tfrac{4+3\alpha}{\alpha} \rho (A_{2k}^2+B_{2k}^2)+\tfrac{1}{1+\alpha} [ (A_k^2-B_k^2)A_{2k}'+2A_kB_kB_{2k}' ] +O(\rho^3).
\end{align}
This completes the proof.
\end{proof}

\begin{proposition}\label{prop:Q-value}
Under the conditions \eqref{eq:power_integer}, \eqref{eq:non-vanishment}, \eqref{eq:magnitude_derivative}, and \eqref{eq:neutral_dominate_II}, we have
\begin{align}
&\tfrac{d}{d\tau}\rho=O(\rho^2), & Q=-\tfrac{k^2-1}{12}\rho^2+O(\rho^3).
\end{align}
\end{proposition}
\begin{proof}
Combining Lemma \ref{lem:rho-nonint} and \eqref{eq:Q} yields
\begin{equation}\label{eq:Q-equation-1}
\tfrac{d}{d\tau} Q=-\lambda_{2k}Q-\tfrac{\alpha+1}{4\alpha}\rho^2+O(\rho Q)+O(\rho^3).
\end{equation}
Since we have $Q=o(1)$ and $\rho=o(1)$, we can reduce \eqref{eq:Q-equation-1} to
\begin{equation}
 \tfrac{1}{\lambda_{2k}} \tfrac{d}{d\tau} Q= -Q +o(\rho).
\end{equation}
Hence, Proposition \ref{prop:non_exp.decay-rho} and Lemma \ref{lem:ODE-1} with $f=\rm Q$ and $\tilde{\rho}=\rho$ yields $\pm Q\leq o(\rho)$, namely $Q=o(\rho)$. Therefore, \eqref{eq:Q-equation-1} implies
\begin{equation}
   \tfrac{d}{d\tau} Q= -\lambda_{2k} Q +O(\rho^2).
\end{equation}
Thus, Proposition \ref{prop:non_exp.decay-rho} and Lemma \ref{lem:ODE-2} with  $\tilde{\rho}=\rho^2$ yields $Q=O(\rho^2)$. Then,  Lemma \ref{lem:rho-nonint} and  \eqref{eq:rho}  lead to $ \frac{d}{d\tau} \rho=O(\rho^2)$. Combining these facts with the equation of $Q$ above yields
\begin{equation}
 \tfrac{d}{d\tau}(Q \rho^{-2})=-  \lambda_{2k}Q \rho^{-2}-\tfrac{\alpha+1}{4\alpha }   +O(\rho ).
\end{equation}
Hence,  Proposition \ref{prop:non_exp.decay-rho} and Lemma \ref{lem:ODE-2} with $f=Q \rho^{-2}+\tfrac{\alpha+1}{4\alpha\lambda_{2k}}$ imply
\begin{equation}
Q\rho^{-2}+\tfrac{\alpha+1}{4 \alpha \lambda_{2k}}=O(\rho).
\end{equation}
Therefore, combining with $\lambda_{2k}=\alpha(4k^2-1)-1=4+3\alpha-1=3(\alpha+1)$ and $\alpha^{-1}=k^2-1$ yields the desired result.
\end{proof}

\bigskip

\begin{proposition}\label{prop:neutral-contradiction}
There exists no smooth ancient solution $v$ to \eqref{eq:v-eq-round} satisfying the conditions \eqref{eq:power_integer}, \eqref{eq:non-vanishment}, \eqref{eq:magnitude_derivative}, and \eqref{eq:neutral_dominate_II}.
\end{proposition}

\begin{proof}
We begin by combining   \eqref{eq:2A0}, Lemma \ref{prop:Q-value}, and Lemma \ref{lem:rho-nonint} to obtain
\begin{equation}
\tfrac{d}{d\tau}(A_0\rho^{-1})=-\lambda_0A_0\rho^{-1}-\tfrac{\alpha+1}{4\alpha}+O(\rho).
\end{equation}
Thus, combining Proposition \ref{prop:non_exp.decay-rho}, Lemma \ref{lem:ODE-2} with $f=A_0 \rho^{-1}+\tfrac{\alpha+1}{4\alpha\lambda_{0}}$, $\lambda_0=-\alpha-1$, and $\alpha^{-1}=k^2-1$ yields 
\begin{equation}
A_0=-\tfrac{\alpha+ 1}{4\alpha \lambda_0 }\rho+O(\rho^2)=\tfrac{k^2-1}{4}\rho+O(\rho^2).
\end{equation}
Hence, \eqref{eq:rho} and Proposition \ref{prop:Q-value} imply
\begin{align}
\tfrac{\alpha}{1+\alpha} \tfrac{d}{d\tau} \rho=\tfrac12 \rho^2- \tfrac{2+\alpha }{4\alpha} \rho^2+ \tfrac{4+3\alpha}{12\alpha} \rho^2 +O(\rho^3)=\left(\tfrac12 -\tfrac{1}{6\alpha}\right)\rho^2+O(\rho^3).
\end{align}
Therefore, $\alpha^{-1}=k^2-1$ and $k\geq 3$ yield
\begin{align}
\tfrac{1}{k^2}\tfrac{d}{d\tau} \rho=\left(\tfrac12 -\tfrac{k^2-1}{6}\right)\rho^2+O(\rho^3)\leq -\tfrac{5}{6}\rho^2+O(\rho^3).
\end{align}
Thus, there exists some $T\ll -1$ satisfying $ \frac{d}{d\tau} \rho \leq 0$ for $\tau\leq T$, namely $\rho(\tau) \geq \rho(T)$ holds for $\tau \leq T$. This contradicts the condition \eqref{eq:non-vanishment}.
\end{proof}

\bigskip

\begin{proof}[Proof of Theorem \ref{thm:non-exist2}]
We may assume that $v$ is not identically zero. Towards a contradiction, we suppose that \eqref{eq:neutral_dominate} holds. Then, Proposition \ref{prop:non-vanish} implies that $\|P_*v\|_{L^2}>0$ holds for sufficiently negative $\tau$, and thus $\ker \mathcal{L}\neq \emptyset$. Therefore, Theorem \ref{thm:sep-L} says that \eqref{eq:power_integer} holds for an integer $k\geq 3$.

Since $\rho=\|P_*v\|_{L^2}^2$, we have \eqref{eq:non-vanishment} by Proposition \ref{prop:non-vanish}. In addition, \eqref{eq:neutral_dominate} already guarantees \eqref{eq:magnitude_derivative} and \eqref{eq:neutral_dominate_II}. This contradicts Proposition \ref{prop:neutral-contradiction}. 

Now, by Theorem \ref{thm:dichotomoy} we have the unstable dominance \eqref{eq:unstable_dominate} 
\begin{equation}
\|P_-   v\|_{L^2(\mathbb{S}^1)}+\|P_*   v\|_{L^2(\mathbb{S}^1)}\leq C \|v\|_{C^4(\mathbb{S}^1)} \|P_+ v\|_{L^2(\mathbb{S}^1)}
\end{equation}
for negative enough $\tau$. So, Proposition \ref{prop:non-vanish} implies $\|P_+ v\|_{L^2}\geq (1-o(1))\|v\|_{L^2}>0$. Thus, we can divide the inequality \eqref{eq:diff-ieq-P-} by $\|P_+ v\|_{L^2}^2$ so that we have
\begin{equation}
\tfrac{d}{d\tau}\log \|P_+ v\|_{L^2(\mathbb{S}^1)}  \geq  - \lambda_I  - C\|  v\|_{C^4(\mathbb{S}^1)}.
\end{equation}
Hence, we have $\frac{d}{d\tau}\log \|P_+ v\|_{L^2}  \geq  - \frac{1}{2} \lambda_I$ for sufficiently negative $\tau$. This yields 
\begin{equation}
\|P_+v\|_{L^2 }\leq Ce^{-\frac{1}{2}\lambda_I \tau},
\end{equation}
and consequently  $\|v\|_{L^2}\leq Ce^{-\frac{1}{2}\lambda_I \tau}$. Now, we remind that \eqref{eq:v-eq-round} can be considered as a linear equation
\begin{equation}\label{eq:linear_eq_b}
    v_\tau=(\alpha + b)v_{\theta\theta}+(\alpha+1+b) v,
\end{equation}
where $\|b\|_{C^3} =o(1)$. Therefore, the standard interior regularity theory for parabolic PDEs leads to $\|v\|_{C^4}\leq Ce^{-\frac{1}{2}\lambda_I \tau}$. Thus, by using  
\begin{equation}
\tfrac{d}{d\tau}\log \|P_+ v\|_{L^2(\mathbb{S}^1)}  \geq  - \lambda_I  - Ce^{-\frac{1}{2}\lambda_I \tau},
\end{equation}
we can obtain 
\begin{equation}
\|v\|_{L^2(\mathbb{S}^1)} \leq (1+o(1))\|P_+v\|_{L^2(\mathbb{S}^1)}\leq Ce^{-\lambda_I \tau}.
\end{equation}
Therefore, we apply the interior parabolic regularity theorems to the linear equation \eqref{eq:linear_eq_b} so that we complete the proof.
\end{proof}

\bigskip

\section{Non-radial asymptotic behavior}\label{sec:non-radial}

In this section, we will show that a rescaled ancient flow asymptotic to $k$-fold shrinkers converges exponentially fast.

We recall the kernel $\ker \mathcal{L}_{\Gamma^\alpha_k}=\text{span}\{h_\theta\}$ from Theorem \ref{thm:sep-L}. Since $h_\theta$ corresponds to the rotation of $\Gamma^{\alpha}_k$, the shrinker $\Gamma^{\alpha}_k$ is  an \textit{integrable critical point} of the entropy $\mathcal{E}_\alpha$. This notion of integrability and the fast convergence were pioneered by \citet{allard1981radial} in their study of tangent cones of minimal surfaces with isolated singularities.

\begin{definition}[$\varepsilon$-shrinking flow]
We say that a rescaled ancient flow $\overline{\Gamma}_\tau $ with the support function $\bar u(\theta,\tau) $  is $\varepsilon$-close to $\Gamma^{\alpha}_k$ with the support function $h$ (up to rotation) at time $\tau_0$ if there exists some $\theta_0\in [0,2\pi)$ such that  
\begin{align}
\|\bar u_{\theta_0}(\cdot,\tau)-h \|_{C^{4,\beta}(\mathbb{S}^1)}\leq \varepsilon,
\end{align}
holds for $\tau\in [\tau_0-1,\tau_0+1]$, where $\bar u_{\theta_0}(\theta,\tau)=\bar u(\theta_0+\theta,\tau)$ and $\beta=\frac{1}{2}$.
\end{definition}

\bigskip

\begin{lemma}\label{lem:half}
Suppose that a rescaled flow $\overline{\Gamma}_\tau$ converges to the shrinker $\Gamma_k^\alpha$ with $k\geq 3$ up to rotations as $\tau\to-\infty $  in $C^{4,\beta}$-sense where $\beta=\frac{1}{2}$. Then, there exist some $L\gg 1$, $T\gg -1$, and $0<\varepsilon_0 \ll 1$ such that if $\bar u(\cdot,\tau)$ is $\varepsilon$-close to $\Gamma^{\alpha}_k$ at every $\tau\leq T_0$ for some $ \varepsilon\leq \varepsilon_0$ and $T_0\leq T$, then 
 $\bar u(\cdot,\tau)$ is $\frac{\varepsilon}{2}$-close to $\Gamma^{\alpha}_k$ at every $\tau\leq T_0-L$.
\end{lemma}
\begin{proof}
By our assumption $\bar u$ is $\varepsilon$-close to $\Gamma^{\alpha}_k$ up to rotations for $\tau\leq T_0$, namely there exist some $\theta_0,\theta_1$ such that
\begin{align}
    \|\bar u(\cdot,s)-h_{-\theta_0} \|_{C^{4,\beta}}<\varepsilon,\quad s\in[\tau-1,\tau+1],\\
    \|\bar u(\cdot,s)-h_{-\theta_1} \|_{C^{4,\beta}}<\varepsilon,\quad s\in[\tau-3,\tau-1],
\end{align}
where $h_{-\theta_i}(\theta)=h(\theta-\theta_i)$  and $h$ is the support function of $\Gamma_k^\alpha$. This means
\begin{align}
    \|h_{-\theta_0}-h_{-\theta_1}\|_{C^{4,\beta}}<2\varepsilon,
\end{align}
and thus we have
\begin{align}
    \|\bar u(\cdot,s)-h_{-\theta_0}\|_{C^{4,\beta}}<3\varepsilon,\quad s\in[\tau-3,\tau+1].
\end{align}
Since $\|\bar u(\cdot,s)-h_{-\theta_0}\|_{C^{4,\beta}}=\|\bar u_{\theta_0}(\cdot,s)-h\|_{C^{4,\beta}}$, iterating this process yields that given integer $L\gg 1$ and $v_{\theta_0}:=\bar u_{\theta_0}-h$,
\begin{equation}\label{eq:v_theta_0-rough_bound}
\|v_{\theta_0}(\cdot,s)\|_{C^{4,\beta}}<(1+2L)\varepsilon,
\end{equation}
holds for $s\in[\tau-1-2L,\tau+1]$.  We denote  $x=\|P_*v_{\theta_0}\|_{h}^2$,  $y=\|P_-v_{\theta_0}\|_{h}^2$, and $z=\|P_+v_{\theta_0}\|_{h}^2$. Then, by Lemma \ref{lem:MZ_v}  we have 
 \begin{align}
 \begin{split}
 z'-2\lambda z\geq -C\sigma(x+y),\\
 |x'|\leq C\sigma(x+y+z),\\
 y' +2\lambda y\leq C\sigma(x+z),
 \end{split}
 \end{align}
for sufficiently negative $\tau$, where $\lambda=\min\{|\frac{1}{2}\lambda_j|:\lambda_j\neq 0\}$ and $\sigma(\tau)=\|v_{\theta_0}(\cdot,\tau)\|_{C^{4,\beta}}$. Notice that we used the fact $\sigma(\tau) \ll \lambda$ for $\tau\ll -1$. Therefore, combining  \eqref{eq:v_theta_0-rough_bound} and Lemma \ref{lem:bdd-MZ} yields 
\begin{align} 
\|P_+v_{\theta_0}(\cdot,s)\|_{h}+\|P_-v_{\theta_0}(\cdot,s)\|_{h}\leq &CL\varepsilon \|P_*v_{\theta_0}(\cdot,s)\|_{h}+C e^{-\frac{\lambda}{8}L}\varepsilon,
\end{align}
for $s\in [\tau-L-2,\tau-L+1]$. Since $\|P_*v_{\theta_0} \|_{h}\leq \| v_{\theta_0} \|_{h}\leq \| v_{\theta_0} \|_{C^{4,\beta}}$,  
\begin{align}\label{eq:+-0}
\|P_+v_{\theta_0}(\cdot,s)\|_{h}+\|P_-v_{\theta_0}(\cdot,s)\|_{h}\leq CL^2\varepsilon^2+CLe^{-\frac{\lambda}{8}L}\varepsilon,
\end{align}
holds for $s\in [\tau-L-2,\tau-L+1]$. 

\bigskip

Now, we want to rotate $\overline{\Gamma}_{\tau-L}$ a little bit so that the neutral mode of the rotated curve is small. We define
\begin{equation}
\theta_0'=\theta_0+(v_{\theta_0}(\cdot,\tau-L),h_\theta)_h \|h_\theta\|_{h}^{-2}.
\end{equation}
We recall $\ker\mathcal{L}_{\Gamma^{\alpha}_k}=\text{span}\{h_\theta\}$, which implies
\begin{equation}\label{eq:neutral.proj_rotation}
P_*v_{\theta_0}(\cdot,\tau-L)=\left(v_{\theta_0}(\cdot,\tau-L),\tfrac{h_\theta}{\|h_\theta\|_{h}} \right)_h\tfrac{h_\theta}{\|h_\theta\|_{h}}=(\theta_0-\theta_0')h_\theta.
\end{equation}
On the other hand, \eqref{eq:v_theta_0-rough_bound} yields
\begin{equation}\label{eq:angle-correction}
|\theta_0'-\theta_0|\leq \|v_{\theta_0}(\cdot,\tau-L)\|_{h}\leq CL\varepsilon.
\end{equation}
Now, given $\theta_1\in [0,2\pi)$ and $s\in [\tau-L-2,\tau-L+1]$, we have
\begin{align}\label{baru'}
\begin{split}
&\left|\left(\bar u_{\theta_0'}(\theta_1  ,s)-\bar u_{\theta_0 }(\theta_1  ,s)\right)-\left(h(\theta_1+\theta_0'-\theta_0)-h(\theta_1)\right)\right|\\
&=\left|\int_{\theta_1 }^{\theta_1+\theta_0'-\theta_0}\tfrac{\partial}{\partial \theta} [\bar u_{\theta_0}(\theta ,s)-h(\theta)]d\theta \right| \leq  |\theta_0'-\theta_0|\|v_{\theta_0}(\cdot,s)\|_{C^1}.
\end{split}
\end{align}
Moreover, we have
\begin{align}
&\left( h(\theta_1+\theta_0'-\theta_0)-h(\theta_1)\right)-(\theta_0'-\theta_0)h_\theta(\theta_1)\\
&=\int_{\theta_1}^{\theta_1+\theta_0'-\theta_0}h_\theta(\theta)-h_\theta(\theta_1) d\theta =\int_{\theta_1}^{\theta_1+\theta_0'-\theta_0}\int^{\theta}_{\theta_1} h_{\theta\theta}(\omega) d\omega d\theta,
\end{align}
which implies
\begin{equation}
\left|\left(h(\theta_1+\theta_0'-\theta_0)-h(\theta_1)\right)-(\theta_0'-\theta_0)h_\theta(\theta_1)\right| \leq  \|h\|_{C^2}|\theta_0-\theta_0'|^2.
\end{equation}
Therefore, combining the above inequalities  with \eqref{eq:v_theta_0-rough_bound}, \eqref{eq:angle-correction}, and \eqref{baru'} implies
\begin{align}
\|\bar u_{\theta_0'}(\cdot ,s)-\bar u_{\theta_0 }(\cdot ,s)-(\theta_0'-\theta_0)h_\theta\|_{L^{\infty}}\leq CL^2\varepsilon^2.
\end{align}
Since $\ker \mathcal{L}_{\bar\Gamma_k^\alpha}=\{h_\theta\}$ implies $P_+h_\theta=P_-h_\theta=0$, combining the inequality above with $\bar u_{\theta_0'}-\bar u_{\theta_0 } =v_{\theta_0'} -v_{\theta_0 }$ yields
\begin{equation}
\|(P_++P_-)\left(v_{\theta_0}(\cdot,s)- v_{\theta_0'}(\cdot,s)\right)\|_{h}\leq CL^2\varepsilon^2,
\end{equation}
Thus, by \eqref{eq:+-0}, for $s\in [\tau-L-2,\tau-L+1]$ we obtain
\begin{align}\label{eq:+-0'}
\|P_+v_{\theta_0'}(\cdot,s)\|_{h}+\|P_-v_{\theta_0'}(\cdot,s)\|_{h}\leq C_0L^2\varepsilon^2+C_0Le^{-\frac{\lambda}{8}L}\varepsilon.
\end{align}
In the same manner, by using $P_*h_\theta=h_\theta$ we have
\begin{align}
\| P_*( v_{\theta_0}(\cdot,s)-  v_{\theta_0'}(\cdot,s)) -(\theta_0'-\theta_0) h_\theta \|_h\leq CL^2\varepsilon^2,
\end{align}
for $s\in [\tau-L-2,\tau-L+1]$. Hence, \eqref{eq:neutral.proj_rotation} implies 
\begin{equation}\label{eq:P_0(t-L)}
 \|  P_*v_{\theta_0'}(\cdot,\tau-L)  \|_h\leq CL^2\varepsilon^2.
\end{equation}
Let $\hat x(s)= \|  P_*v_{\theta_0'}(\cdot,s)  \|_h^2$. We recall from Lemma \ref{lem:MZ_v} that
\begin{equation}
|\hat x'|\leq C\|v_{\theta_0'}\|_{C^4} \|v_{\theta_0'}\|_h^2,
\end{equation}
and thus \eqref{eq:v_theta_0-rough_bound} yields
\begin{equation}
|\hat x'| \leq CL\epsilon\left(\hat x+\|P_+v_{\theta_0'} \|_{h}^2+\|P_-v_{\theta_0'}\|_{h}^2\right).
\end{equation}
Thus, \eqref{eq:+-0'} implies
\begin{equation}
|\hat x'|\leq CL\varepsilon (\hat x+ L^4\varepsilon^4+L^2e^{-\frac{\lambda}{4}L}\varepsilon^2),
\end{equation}
namely 
\begin{equation}
\left|\tfrac{d}{ds}\log (\hat x+ L^4\varepsilon^4+L^2e^{-\frac{\lambda}{4}L}\varepsilon^2) \right|\leq CL\varepsilon
\end{equation}
holds for $s\in [\tau-L-2,\tau-L+1]$. Therefore, by using \eqref{eq:P_0(t-L)} we integrate the inequality above from $\tau-L$ to $s$ so that we have
\begin{equation}
\|  P_*v_{\theta_0'}(\cdot,s)  \|_h^2\leq e^{CL\varepsilon}(CL^4\varepsilon^4+ L^2e^{-\frac{\lambda}{4}L}\varepsilon^2)
\end{equation}
for $s\in [\tau-L-2,\tau-L+1]$. Hence, combining with \eqref{eq:+-0'} gives
\begin{equation}\label{eq:L2_bound_rot}
\|   v_{\theta_0'}(\cdot,s) \|_h\leq C(1+e^{CL\varepsilon})(L^2\varepsilon^2+ Le^{-\frac{\lambda}{8}L}\varepsilon )
\end{equation}
for $s\in [\tau-L-2,\tau-L+1]$.  Since \eqref{eq:v-flow} can be written as
\begin{equation}
    \partial_\tau v_{\theta_0'}=-( ( v_{\theta_0'})_{\theta\theta}+v_{\theta_0'})^{-\alpha}+(h_{\theta\theta}+h)^{-\alpha}+v_{\theta_0'},
\end{equation}
by \eqref{eq:v_theta_0-rough_bound} with $L\ll \varepsilon^{-1}$,  there are a smooth function $a(\theta,\tau)$ and some constant $\Lambda>0$ only depending on $\alpha, h$ such that $ \| a\|_{C^{2,\beta}}\leq \Lambda$, $a\geq \Lambda^{-1}$, and
\begin{equation}
\tfrac{\partial}{\partial \tau} v_{\theta_0'}=a \tfrac{\partial^2}{\partial \theta^2}  v_{\theta_0'} +(a+1)v_{\theta_0'},    
\end{equation}
hold for $\tau \in [\tau-1-2L,\tau+1]$.  Therefore, the $L^2$ bound \eqref{eq:L2_bound_rot} and  the standard interior regularity theory for linear parabolic PDEs yield
\begin{equation}
\|   v_{\theta_0'}(\cdot,s) \|_{C^{4,\beta}}\leq C(1+e^{CL\varepsilon})(L^2\varepsilon+ Le^{-\frac{\lambda}{8}L})\varepsilon,
\end{equation}
for $s\in [\tau-L-1,\tau-L+1]$. Therefore, by choosing large enough $L$ and small enough $\varepsilon_0$ we can obtain the desired result that
\begin{equation}
\|   v_{\theta_0'}(\cdot,s) \|_{C^{4,\beta}}\leq \tfrac12\varepsilon,
\end{equation}
for $s\in [\tau-L-1,\tau-L+1]$.
\end{proof}

\bigskip

\begin{theorem}\label{thm:k-exp} Suppose that the rescaled flow $\overline{\Gamma}_\tau$ converges to the shrinker $\Gamma_k^\alpha$ with $k\geq 3$ up to rotations as $\tau\to-\infty $  in $C^{4,\beta}$-sense, where $\beta=\frac{1}{2}$. Then, there exist some $\delta>0$, $T\ll -1$, and a fixed rotation $S_0\in SO(2)$ such that
\begin{equation}
\|\bar u(\cdot,\tau)-\bar h\|_{C^{4,\beta}}\leq e^{\delta\tau}
\end{equation}
holds for $\tau \leq T$, where $\bar u(\cdot,\tau)$ and $\bar h$ are the support functions of  $\overline{\Gamma}_\tau$ and $S_0\Gamma^{\alpha}_k$, respectively.
\end{theorem}

\begin{proof}
We recall $L,\varepsilon_0,T$ in Lemma \ref{lem:half} and choose $T_0\leq T$ such that $\bar u(\cdot,\tau)$ is $\varepsilon_0$-close to $\Gamma^{\alpha}_k$ up to rotations at every $\tau \leq T_0$. Then, we apply Lemma \ref{lem:half} repeatedly so that we can obtain a sequence $\{\theta_i\}_{i\in \mathbb{N}}$ satisfying 
\begin{equation}
\|\bar u(\cdot,\tau)-h_{-\theta_i}\|_{C^{4,\beta}}\leq 2^{-i}\varepsilon_0
\end{equation}
holds for $ \tau\leq \tau_0-iL$, where $h_{\mu}(\theta):=h(\theta+\mu)$.  Because the symmetry of $h$, the angle $\theta_i$ is determined up to $2\pi/k$. Hence, $h_{-\theta_i}$ is a Cauchy sequence in $C^{4,\beta}$-sense as follows;
\begin{equation}
\|h_{-\theta_i}-h_{-\theta_{i+j}}\|_{C^{4,\beta}}\leq \sum_{l=i}^{i+j}2^{-l}\varepsilon_0\leq 2^{-i+1}\varepsilon_0.
\end{equation}
Thus, modulo $2\pi/k$, $\{\theta_i\}_{i\in \mathbb{N}}$ has the limit $\bar{\theta}$ and we have
\begin{equation}
\|\bar u(\cdot,\tau)-h_{-\bar \theta}\|_{C^{4,\beta}}\leq 2^{-i+2}\varepsilon_0,
\end{equation}
for $ \tau\leq \tau_0-iL$.  Choosing $\delta:=(2L)^{-1}\log 2$ and applying the inequality above for $\tau \in [\tau_0-(i+1)L,\tau_0-iL]$ yield
\begin{equation}
\|\bar u(\cdot,\tau)-{ h_{-\bar \theta}}\|_{C^{4,\beta}}\leq 2^{-i+2}\varepsilon_0=2\varepsilon_0 e^{-2(i+1)L\delta} \leq 2\varepsilon_0 e^{2\delta(\tau-\tau_0)}.
\end{equation}
Therefore, by choosing negative enough $T$ we complete the proof.
\end{proof}

\section{Classification of ancient solutions}\label{sec:classification}

\subsection{Summary for small powers}
 To begin with, we summarize the results in previous sections to show the exponential convergence of ancient rescaled flows with small powers.

\begin{proposition}[]\label{prop:unique_tangent.flow}
Let $\overline{\Gamma}_\tau$ be a rescaled closed smooth ancient $\alpha$-curve shortening flow with $\alpha\in (0,\frac{1}{3})$. Then, there exists a fixed rotation $S\in SO(2)$ such that $S\overline{\Gamma}_\tau$ converges exponentially fast to a shrinker $\Gamma_k^\alpha$ in the $C^{10}$-topology as $\tau \to -\infty$, where $k\in \{\infty\}\cup \{3,4,\cdots\}$. Moreover, there exist some constants $C,\delta>0$ and negative time $T$ such that the difference $v=\bar u(\cdot,\tau)-h$  satisfies
\begin{equation}\label{eq:exp_decay}
\|v\|_{C^4(\mathbb{S}^1)}\leq Ce^{\delta\tau}
\end{equation}
for $\tau \leq T$, where $\bar u(\cdot,\tau)$ and $h$ are the support functions of $S\overline{\Gamma}_\tau$ and $\Gamma_k^\alpha$, respectively.  
\end{proposition}

\begin{proof}
Combining Theorem \ref{thm:ent-limit} and Theorem \ref{thm:k-exp} yields the convergence of the rescaled flow $\overline{\Gamma}_\tau$ to a shrinker. So, there are a rotation $S$ and $ k\in \{\infty\}\cup \{3,4,\cdots\} $ such that $S\overline{\Gamma}_\tau$ converges to $\Gamma_k^\alpha$. Then, Theorem \ref{thm:non-exist2} and Theorem \ref{thm:k-exp} imply the exponential convergence of $\bar u(\cdot,\tau)$ to $h$.
\end{proof}

We recall that the difference $v=\bar u-h$ satisfies  \eqref{eq:v-flow}
\begin{equation}\label{eq:v-flow2}
v_\tau=\mathcal{L}v+E(v),
\end{equation}
where
\begin{align}\label{def:E2}
    E(v)=-{(h^{-\frac{1}{\alpha}}+v_{\theta\theta}+v)^{-\alpha}}+h-\alpha h^{1+\frac{1}{\alpha}}(v_{\theta\theta}+v).
\end{align}
We also recall the eigenfunctions $\{\varphi_i\}_{i\in\mathbb{N}}$ and the eigenvalues $\{\lambda_i\}_{i\in\mathbb{N}}$ such that $(\varphi_i,\varphi_j)_h=\delta_{ij}$, $\text{span}\{\varphi_i\}_{i\in\mathbb{N}}=L^2_h$, $\mathcal{L}\varphi_i+\lambda_i\varphi_i=0$, $\lambda_i\leq \lambda_{i+1}$, and $\lim_{i\to +\infty}\lambda_i=+\infty$. Finally, we recall the projections $P_if=(f,\varphi_i)_h\varphi_i$ and
\begin{align}
&P_{<\lambda}=\sum_{\{i:\lambda_i<\lambda\}}P_i, && P_{=\lambda}=\sum_{\{i:\lambda_i=\lambda\}}P_i, && P_{>\lambda}=\sum_{\{i:\lambda_i>\lambda\}}P_i.
\end{align}

\begin{notation}
We introduce some notations which will be used in this section \ref{sec:classification}.
\begin{enumerate}
\item  $\{\bar \lambda_i\}_{i\in\mathbb{N}}$ is the strictly increasing sequence satisfying 
\begin{align}
\{\bar \lambda_i: i\in \mathbb{N}\}=\{ \lambda_i: i\in \mathbb{N}\}.
\end{align}
\item $\bar I\in \mathbb{N}$ is defined by $\bar \lambda_{\bar I}=\lambda_I$, where $I$ denotes the Morse index of $\mathcal{L}$.
\item $d_i$ denotes the dimension of the eigenspace associated with the eigenvalue $\bar \lambda_i$. In addition, we define $D_i=\sum_{m=1}^id_m$.  
\item We define vector valued eigenfunctions $\{\vec{\varphi}_i\}_{i\in\mathbb{N}}$ by 
\begin{equation}
\vec{\varphi}_i=(\varphi_{1+D_{i-1}},\cdots,\varphi_{D_i}).
\end{equation}
\item For each  $i\leq \bar I$, the projection $\bar \pi_i:\mathbb{R}^I\to \mathbb{R}^I$ is defined by
\begin{equation}
\bar \pi_i(\boldsymbol{a})=(0,\cdots,0,a_{1+D_{i-1}},\cdots,a_I),
\end{equation}
where $\boldsymbol{a}=(a_1,\cdots,a_I)\in \mathbb{R}^I$.
\item By abuse of notation, we would denote $(\vec{a}_1,\cdots,\vec{a}_{\bar I})\in \mathbb{R}^I$ for $\vec{a}_i\in \mathbb{R}^{d_i}$. Also, we denote $(\vec{b}_1,\vec{a}_i,\vec{b}_2)\in \mathbb{R}^I$ for $\vec{b}_1\in \mathbb{R}^{D_{i-1}}$, $\vec{a}_i\in \mathbb{R}^{d_i}$, $\vec{b}_2\in \mathbb{R}^{\bar I-D_i}$.
\end{enumerate}
\end{notation}

 \bigskip
 
 Now, we recall the ancient $\alpha$-curve shortening flow with $\alpha<\frac{1}{3}$ constructed in \cite{choisun}.
\begin{proposition}[Choi-Sun \cite{choisun}]\label{prop:existence}
There is a function $\mathcal{S}$ from $\mathbb{R}^I$ to the set of ancient solutions to \eqref{eq:v-flow2}, a time function $T:\mathbb{R}^+\to \mathbb{R}$, and $\beta\in (0,1)$ with the following significance. 
\begin{enumerate}[(i)]
\item $\mathcal{S}(\boldsymbol{0})\equiv 0$.
\item Given $\boldsymbol{a} \in B_r^I(0)$,  $\mathcal{S}(\boldsymbol{a})$ is a smooth ancient solution to \eqref{eq:v-flow2} whose first singular time is greater than $T(r)$, Namely, $\mathcal{S}(\boldsymbol{a})(\cdot,\tau)$ exists for $\tau \leq T(r)$. 
\item For each $r>0$, $\mathcal{S}|_{T(r)}: B_r^I(0)\to C^{2,\beta}(\mathbb{S}^1\times (-\infty,T(r)]) $ is a continuous map, where $\mathcal{S}|_{T(r)}(\boldsymbol{a})$ is the restriction of the function $\mathcal{S}(\boldsymbol{a})$ to $\mathbb{S}^1\times (-\infty,T(r)]$.
\item  If $\bar \pi_{j+1}(\boldsymbol{a})=\bar \pi_{j+1}(\boldsymbol{b})$ for some $j < \bar I$, then 
\begin{equation}
\lim_{\tau\to -\infty}\left\| e^{\bar \lambda_j \tau}[\mathcal{S}( \boldsymbol{a} )-\mathcal{S}( \boldsymbol{b})]- (\vec{a}_j-\vec{b}_j)\cdot \vec{\varphi}_j\right\|_h=0
\end{equation}
holds, where $\boldsymbol{a}=(\vec{a}_1,\cdots,\vec{a}_{\bar I})$ and $\boldsymbol{b}=(\vec{b}_1,\cdots,\vec{b}_{\bar I})$. Namely, $\mathcal{S}$ is injective.
\end{enumerate}  
\end{proposition}

\begin{proof}
The map $\mathcal{S}$ in this proposition is the same map $\mathcal{S}$ in \cite[Theorems 3.2 \& 3.5]{choisun}. Indeed, the theorems 3.2 \& 3.5 in \cite{choisun} already imply this proposition, but we explain the reason for readers' convenience. To begin with, we recall some notations in \cite{choisun}.
\begin{enumerate}
\item $L:=\lfloor \lambda_1/\lambda_I \rfloor$ is the greatest integer less than or equal to $\lambda_1/\lambda_I$.
\item $J^{(l)}:=\{m:(l+1)\lambda_I<\lambda_m \leq l\lambda_I\}$ for $l=1,\cdots,L$. 
\item $\delta_l$ is any real number satisfying $(l+1)\lambda_I<-\delta_l <\min\{\lambda_m:m\in J^{(l)}\}$, and $X^{(l)}:=X^{\delta_l}$ is the Banach space equipped with the norm $\|\cdot\|_{\mathcal{C}^{2,\beta,\delta_l}}$ defined on the equation (3.2) in \cite{choisun} 
\item $\iota^{(l)}(\boldsymbol{a})=\sum_{j\in J^{(l)}} a_j e^{- \lambda_j \tau}  \varphi_j$, where $\boldsymbol{a}=(a_1,\cdots,a_I)$.
\end{enumerate}
We also define a time function
\begin{equation}\label{eq:time_definition}
T(r)=\tfrac{1}{2(1+\alpha)}\min\left\{\log\tfrac{\varepsilon_0^2}{r^2},0\right\},
\end{equation}
by using the equation (3.17) in \cite{choisun}.

Next, we recall some properties of $\mathcal{S}$ in \cite{choisun}. \cite[Theorem 3.2]{choisun} says that given $\boldsymbol{a}\in \overline{B_{\varepsilon_0}^I}(0)$ there exists a sequence of smooth functions $\{v^{(l)}_{\boldsymbol{a}}\}_{l=1}^{L}$  such that $\mathcal{S}(\boldsymbol{a})= \sum_{j=1}^L v^{(j)}_{\boldsymbol{a}}$ and the following hold for each $1\leq l \leq L$. 
\begin{enumerate}[(a)]
\item $v^{(j)}_{\boldsymbol{a}}$ is defined on $\mathbb{S}^1\times (-\infty,T(|\boldsymbol{a}|)]$.
\item $\sum_{j=1}^l v^{(j)}_{\boldsymbol{a}}$ is a solution to \eqref{eq:v-flow2}. 
\item $v^{(l)}_{\boldsymbol{a}}-\iota^{(l)}(\boldsymbol{a})\in X^{(l)}$, namely $\displaystyle\lim_{\tau\to -\infty}e^{-\delta_l\tau} \|v^{(l)}_{\boldsymbol{a}}-\iota^{(l)}(\boldsymbol{a})\|_h=0$.
\item $\displaystyle\lim_{\tau\to -\infty}e^{\lambda_m\tau}(v^{(l)}_{\boldsymbol{a}}(\cdot,\tau),\varphi_m)_h=a_m$ for every $m\in J^{(l)}$.
\end{enumerate}
Notice that we obtained (a) by \eqref{eq:time_definition} and $|\boldsymbol{a}|\leq \varepsilon_0$.

For $|\boldsymbol{a}|>\varepsilon_0$, one can consider $\tilde{\boldsymbol{a}}=(\tilde{a}_1,\cdots,\tilde{a}_I)$ where $\tilde{a}_m= e^{-\lambda_mT(|\boldsymbol{a}|)}a_m$ as in the proof of \cite[Theorem 3.5]{choisun}. Then, by using $|\tilde{\boldsymbol{a}}|\leq \varepsilon_0$ we can find $\{v^{(l)}_{\tilde{\boldsymbol{a}}}\}_{l=1}^{L}$ satisfying the properties above. Then, $v^{(l)}_{\boldsymbol{a}}(\cdot,\tau)=v^{(l)}_{\tilde{\boldsymbol{a}}}(\cdot,\tau-T(|\boldsymbol{a}|))$ satisfy the desired properties (a)-(d), and moreover the map $\mathcal{S}$ is defined again by $\mathcal{S}(\boldsymbol{a})= \sum_{j=1}^L v^{(j)}_{\boldsymbol{a}}$.

\bigskip

Now, we are ready to verify this proposition \ref{prop:existence}.  First of all, we can easily obtain (ii), (iii) from \cite[Theorem 3.2 \& 3.5]{choisun} and the above properties (a), (b).

Secondly, in order to show (iv), we observe that $v^{(l)}_{\boldsymbol{a}}$ is uniquely determined by $a_{m_0},\cdots,a_I$ where $\lambda_{m_0-1}\leq (l+1)\lambda_I<\lambda_{m_0}$. To be specific, as in the proof of \cite[Theorem 3.2]{choisun}, $v^{(1)}_{\boldsymbol{a}}-\iota^{(1)}(\boldsymbol{a})$ is the unique solution to the equation (3.7) in \cite{choisun}. Hence, $v^{(1)}_{\boldsymbol{a}}$ is uniquely determined by $\iota^{(1)}(\boldsymbol{a})$.  Then, $v^{(l)}_{\boldsymbol{a}}$ is inductively determined by $\iota^{(1)}(\boldsymbol{a}),\cdots,\iota^{(l)}(\boldsymbol{a})$ due to the equation (3.11) in \cite{choisun}. Therefore, if $\bar \pi_{j_0+1}(\boldsymbol{a})=\bar \pi_{j_0+1}(\boldsymbol{b})$ with  $(l_0+1)\lambda_I< \bar \lambda_{j_0}\leq  l_0\lambda_I $, then we have $v^{(l)}_{\boldsymbol{a}}=v^{(l)}_{\boldsymbol{b}}$ for $l<l_0$. On the other hand, if $l\geq l_0$, the property (c) and $\delta_l>-\lambda_{j_0}$ yield
\begin{equation}
0=\lim_{\tau\to -\infty}e^{-\delta_l\tau}\|v^{(l)}_{\boldsymbol{a}}-v^{(l)}_{\boldsymbol{b}}-\iota^{(l)}(\boldsymbol{a}-\boldsymbol{b})\|_h=\lim_{\tau\to -\infty}e^{\bar\lambda_{j_0}\tau}\|v^{(l)}_{\boldsymbol{a}}-v^{(l)}_{\boldsymbol{b}}-\iota^{(l)}(\boldsymbol{a}-\boldsymbol{b})\|_h.
\end{equation}
Hence, summing them up yields $e^{\bar \lambda_{j_0}\tau}\|\mathcal{S}(\boldsymbol{a})-\mathcal{S}(\boldsymbol{b})-\sum_{l=l_0}^L\iota^{(l)}(\boldsymbol{a}-\boldsymbol{b})\|_h\to 0$. We remind that $\bar \pi_{j_0+1}(\boldsymbol{a})=\bar \pi_{j_0+1}(\boldsymbol{b})$ implies $a_{i}=b_{i}$ for $\lambda_i>\bar \lambda_{j_0}$. Also, $e^{(\bar \lambda_{j_0} -\lambda_i)\tau}(a_i-b_i)\varphi_i\to 0$ holds for $\lambda_i<\bar \lambda_{j_0}$. Thus, $e^{\bar \lambda_{j_0}\tau}\sum_{l=l_0}^L\iota^{(l)}(\boldsymbol{a}-\boldsymbol{b})\to (\vec{a}_{j_0}-\vec{b}_{j_0})\cdot\vec{\varphi}_{j_0}$. This completes the proof of (iv).

Finally, (i) and $v^{(l)}_{\boldsymbol{0}}\equiv 0$ are obvious by $\iota^{(l)}(\boldsymbol{0})\equiv 0$ and (3.7), (3.11) in \cite{choisun}.
\end{proof}

\bigskip

\begin{remark}\label{rmk:I_3_paramater}
Given $k\neq \infty$, the $k$-fold shrinker is unique up to rotations by \cite{andrews2003classification}, and thus Proposition \ref{prop:existence} yields an $I$-parameter family of ancient \textit{rescaled} $\alpha$-curve shortening flows up to rotations. However, \cite[Proposition 3.4]{choisun} says that the first three parameters $(a_1,a_2,a_3)$ are determined by the choice of the rescaling center. Hence, Proposition \ref{prop:existence} indeed yields an $(I-3)$-parameter family of ancient \textit{non-rescaled} flows up to rigid motions and dilation.

Similarly, for $k=\infty$, there is an $(I-4)$-parameter family of ancient \textit{non-rescaled} flows up to rigid motions and dilation, because the unit circle $\Gamma_{\infty}^\alpha$ is invariant under rotations.

Hence, by remembering Theorem \ref{thm:sep-L}, the rescaled flows obtained from Proposition \ref{prop:existence} correspond the non-rescaled flows in Theorem \ref{thm:existence}. Therefore, we can prove Theorem \ref{thm:subaffine.classification} by showing that any rescaled $\alpha$-CSF converging to $\Gamma_k^\alpha$ has the support function $\bar u=h+\mathcal{S}(\boldsymbol{a})$ for some $\boldsymbol{a}\in \mathbb{R}^I$, where $h$ is the support function of $\Gamma_k^\alpha$.
\end{remark}

\subsection{Asymptotic behavior of difference}

In this subsection, our goal is to prove the following technical theorem.
\begin{theorem}\label{thm:asy_diff}
Suppose that $v_1,v_2$ are smooth ancient solutions to \eqref{eq:v-flow2} undergoing exponential decay \eqref{eq:exp_decay}.  Then, there exist a negative eigenvalue $\bar \lambda_i$ and a non-zero vector $0\neq \vec{a}_i\in \mathbb{R}^{d_i}$ such that the difference $w=v_1-v_2$ satisfies 
\begin{align}\label{eq:rep_coe}
 \lim_{\tau\to -\infty} \left\|e^{\bar\lambda_i\tau} w-\vec{a}_i\cdot\vec{\varphi}_i \right\|_h=0, 
\end{align}
unless $w\equiv 0$.
\end{theorem}

Since each $v_i$ ($i=1,2$) satisfies \eqref{eq:v-flow2}, we have
\begin{equation}
\partial_\tau v_i-\mathcal{L}v_i=E(v_i)=h^{1+\frac{2}{\alpha}}(\partial_\theta^2v_i+v_i)^2f(h^{\frac{1}{\alpha}}(\partial_\theta^2v_i+v_i)),
\end{equation}
where $f$ is given in \eqref{eq:error_f}. Hence, the difference $w=v_1-v_2$ satisfies
\begin{equation}\label{def:F2}
w_\tau-\mathcal{L}w=E(v_1)-E(v_2)=:F(w;v_1,v_2)=:(\partial_\theta^2 w +w )G(v_1,v_2)h^{1+\frac{1}{\alpha}},
\end{equation}
where
\begin{align}\label{def:G}
\begin{split}
G:=&\,h^{\frac{1}{\alpha}}(\partial_\theta^2v_1+v_1+\partial_\theta^2v_2+v_2) f(h^{\frac{1}{\alpha}}(\partial_\theta^2v_1+v_1))\\
&+ h^{\frac{2}{\alpha}}(\partial_\theta^2v_2+v_2)^2\int_0^1 f'(h^{\frac{1}{\alpha}}(s (\partial_\theta^2v_1+v_1) +(1-s)(\partial_\theta^2v_2+v_2)))ds. \end{split}
\end{align}

\begin{lemma}\label{lem:F_total_error}
 Suppose that $v_1,v_2$ are smooth ancient solutions to \eqref{eq:v-flow2}, and given $\varepsilon>0$ there is $T_\varepsilon$ such that the difference $w=v_1-v_2$ satisfies $\|v_1(\cdot,\tau)\|_{C^4},\|v_2(\cdot,\tau)\|_{C^4}\leq \varepsilon$ for $\tau \leq T_\varepsilon$. Then, there are some constants $\varepsilon_0,C_0$ only depending on $\alpha,h$ such that  
 \begin{align}\label{eq:F_error_est}
|(F, w)_h|(\tau) \leq C_0(\|v_1\|_{C^4}+\|v_2\|_{C^4})\sup_{s\leq \tau}\|w(\cdot,s)\|_h^2 ,
\end{align}
holds for $\tau \leq T_{\varepsilon_0}$.
\end{lemma}

\begin{proof}
While $\|v_1\|_{C^4},\|v_2\|_{C^4}$ are small enough, $w$ solves the linear PDE \eqref{def:F2}
\begin{equation}
w_\tau=\mathcal{L}w+h^{1+\frac{1}{\alpha}}G (w_{\theta\theta}+w).
\end{equation}
 Hence, the standard interior regularity theory for linear parabolic PDEs yields
\begin{equation}
    \|w(\cdot,\tau)\|_{H^1}\leq C\sup_{s\leq \tau}\|w(\cdot,s)\|_h
\end{equation}
for some constant $C$ only depending on $\alpha,h$. On the other hand,
\begin{align}\label{eq:F_integration}
    (F,w)_h=\iS  ( w_{\theta\theta} +w )w G d\theta=\iS - w_{\theta}^2 G -w_\theta w G_\theta  +w^2G d\theta,
\end{align}
yields
\begin{equation}\label{eq:F_entire_error}
    |(F,w)_h|\leq C \|G\|_{C^1}\|w\|_{H^1}^2\leq C(\|v_1\|_{C^4}+\|v_2\|_{C^4}) \|w\|_{H^1}^2.
\end{equation}
Therefore, combining the two inequalities above completes the proof.
\end{proof}

\bigskip

\begin{lemma}\label{lem:fast_decay}
Suppose that $v_1,v_2$ are smooth ancient solutions to \eqref{eq:v-flow2} undergoing exponential decay \eqref{eq:exp_decay}. Then, for every negative value $\lambda<\lambda_1$, the difference $w=v_1-v_2$ satisfies 
\begin{equation}
\lim_{\tau\to -\infty}e^{\lambda\tau}\|w(\cdot,\tau)\|_{h}=+\infty,
\end{equation}
unless $v_1 \equiv v_2 $.
\end{lemma}

\begin{proof}
Suppose that $\limsup_{\tau\to -\infty} e^{\lambda\tau}\|w\|_h \leq C$ holds for some $\lambda< \lambda_1$. Then, we have $\delta'=\min\{\delta,\lambda_1-\lambda\}>0$. We recall $\varepsilon_0,T_{\varepsilon_0}$ in Lemma \ref{lem:F_total_error}, and then we find some $C_1$ and $T_1\leq T_{\varepsilon_0}$ such that 
$\|v_1\|_{C^4}+\|v_2\|_{C^4} \leq C_1 e^{\delta'\tau}$ and $e^{\lambda_1\tau}\|w\|_h\leq C_1 e^{\delta'\tau}$
hold for $\tau \leq T_1$. Hence, Lemma \ref{lem:F_total_error} yields
\begin{equation}\label{eq:Fw_est}
    |(F, w)_h|  \leq C_0C_1e^{\delta'\tau}\sup_{s\leq \tau} \|w(\cdot,s)\|_h^2.
\end{equation}
Since we have
\begin{equation}
\tfrac{1}{2}\tfrac{d}{d\tau}\|w\|_h^2=(w,\mathcal{L}w+F)_h\leq  -\lambda_1  \|w\|^2_h +(w,F)_h,
\end{equation}
combining with \eqref{eq:Fw_est} and $e^{\lambda_1\tau}\|w\|_h\leq C_1 e^{\delta'\tau}$ yields
\begin{equation}
\tfrac{1}{2}\tfrac{d}{d\tau}e^{2\lambda_1 \tau}\|w\|_h^2\leq   C_0C_1e^{(2\lambda_1+\delta') \tau}\sup_{s\leq \tau}\|w(\cdot,s)\|_h^2 \leq  C_0C_1^3  e^{3\delta'\tau}
\end{equation}
for $\tau \leq T_1$. Remembering $ \displaystyle \lim_{\tau\to -\infty} e^{\lambda_1\tau}\|w\|_h\leq \lim_{\tau\to -\infty} C_1e^{\delta'\tau}= 0$, we have 
\begin{equation}
  e^{  2\lambda_1  \tau}\|w\|_h^2=\int_{-\infty}^{\tau} \tfrac{d}{ds}e^{2\lambda_1 s}\|w(\cdot,s)\|_h^2ds \leq       \frac{2C_0C_1^3}{3\delta'} e^{3\delta'\tau}=\frac{2^2C_0C_1^3}{3!\delta'} e^{3\delta'\tau}
\end{equation}
for $\tau\leq T_1$. We repeat this process so that we have
\begin{equation}
  \tfrac{1}{2}\tfrac{d}{d\tau}e^{2\lambda_1 \tau}\|w\|_h^2\leq   C_0C_1e^{(2\lambda_1+\delta') \tau}\sup_{s\leq \tau}\|w(\cdot,s)\|_h^2 \leq  \frac{2^2C_0^2C_1^4}{3!\delta'} e^{4\delta'\tau}, 
\end{equation}
and thus
\begin{equation}
    e^{  2\lambda_1  \tau}\|w\|_h^2=\int_{-\infty}^{\tau} \tfrac{d}{ds}e^{2\lambda_1 s}\|w(\cdot,s)\|_h^2ds \leq       \frac{2^3C_0^2C_1^4}{4!(\delta')^2} e^{4\delta'\tau}. 
\end{equation}
By iterating the same process, we obtain
\begin{equation}
    e^{  2\lambda_1  \tau}\|w\|_h^2  \leq       \frac{2^{k-1}C_0^{k-2}C_1^k}{k!(\delta')^{k-2}} e^{k\delta'\tau}
\end{equation}
for $2\leq k \in \mathbb{N}$ and $\tau \leq T_1$. Passing $k\to \infty$ yields $w\equiv 0$ for $\tau \leq T_1$, and therefore $w\equiv 0$ holds for all $\tau$.
\end{proof}

\bigskip

\begin{lemma}\label{lem:FP-}
 Suppose that $v_1,v_2$ are smooth ancient solutions to \eqref{eq:v-flow2}  undergoing exponential decay \eqref{eq:exp_decay}, and the difference $w=v_1-v_2$ satisfies $\limsup_{\tau\to -\infty} e^{\lambda \tau}\| w\|_h \leq C$ for some $C >0$ and $\lambda <0$. Then, there are some constant $C$ depending on $\alpha, h$ and some $T\ll -1$ such that
 \begin{align}\label{eq:F_error_est1}
|(F,P_-w)_h|+\sum_{\lambda_j\leq 0} |(F,P_{=\lambda_j}w)_h|\leq Ce^{(\delta -2\lambda) \tau} ,
\end{align}
and
\begin{equation}\label{eq:F_projection}
 \sum_{\lambda_j\leq 0} |(F,\varphi_j )_h|\leq Ce^{\delta \tau}  \|w\|_{h},
\end{equation}
hold for $\tau\leq T$ and $F$ given in \eqref{def:F2}.
\end{lemma}

\begin{proof}
As \eqref{eq:F_integration}, by using integration by parts, we can obtain
 \begin{equation}
    |(F,\varphi_j )_h|\leq \left| \int wG\varphi_j+w(G\varphi_j)_{\theta\theta}d\theta \right|\leq  C \|G\|_{C^2}\|\varphi_j \|_{H^2} \|w\|_{h}.
\end{equation}
Since $\varphi_j$ is a unit eigenfunction with the eigenvalue $\lambda_j\leq 0$, we have 
\begin{equation}
\|\varphi_j \|_{H^2}\leq C(1+\|\mathcal{L}\varphi_j\|_h)\leq C(1+|\lambda_j|).    
\end{equation}
Notice that we know $\varphi_1=h/\|h\|_h $ and $\lambda_1=-1-\alpha$ by the Courant nodal domain theorem. (See also Proposition 2.2 in \cite{choisun}.) Therefore, $\lambda_j\leq 0$ and  implies $|\lambda_j|\leq |\lambda_1|=1+\alpha$. Thus, combining the inequalities above with $\|G\|_{C^2}\leq Ce^{\delta\tau}$ yields \eqref{eq:F_projection}. 

Next, \eqref{eq:F_projection} implies
 \begin{align}
\sum_{\lambda_j\leq 0} |(F,P_{=\lambda_j}w)_h|\leq Ce^{\delta \tau}  \|w\|_{h}^2\leq  Ce^{(\delta -2\lambda) \tau}.
\end{align}
Hence, combining with Lemma \ref{lem:F_total_error} completes the proof of \eqref{eq:F_error_est1}.
\end{proof}

\bigskip

Let $v_1,v_2$ be smooth ancient solutions to \eqref{eq:v-flow2} satisfying \eqref{eq:exp_decay}. Then, there is $\bar \lambda \in [-\infty,-\delta]$ such that the difference $w=v_1-v_2$ satisfies 
\begin{equation}\label{eq:optimal_decay}
\bar \lambda:= \inf \{ \lambda \in \mathbb{R}:  \limsup_{\tau\to -\infty}e^{\lambda\tau}\|w(\cdot,\tau)\|_{h}<+\infty \}.
\end{equation}
Suppose that the difference $w=v_1-v_2$ is not identically zero, Then, by Lemma \ref{lem:fast_decay} there exists a negative eigenvalue $\bar \lambda_j<0$ such that 
\begin{equation}\label{eq:optimal_eigenvalue}
\bar \lambda_j \leq \bar \lambda < \bar \lambda_{j+1}\leq 0.
\end{equation}
By using $\bar \lambda_j$, we define
\begin{align}\label{eq:W_proj_definition}
&    W_+:=\| P_{<\bar \lambda_j}w\|_h^2, &&  W_*:=\| P_{=\bar \lambda_j}w\|_h^2, &&  W_-:=\| P_{>\bar \lambda_j}w\|_h^2.
\end{align}

\bigskip

\begin{proposition}\label{prop:W_ODEs}
Suppose that $v_1,v_2$ are smooth ancient solutions to \eqref{eq:v-flow2} satisfying \eqref{eq:exp_decay}, with the difference $w=v_1-v_2 \not \equiv 0$.  We recall  $\bar \lambda_j,\bar \lambda, W_+,W_*,W_-$ in \eqref{eq:optimal_decay}, \eqref{eq:optimal_eigenvalue}, and \eqref{eq:W_proj_definition}. Then, there are some constants $C>0$ and $T\ll -1$ such that
\begin{align}
    \tfrac{d}{d\tau}W_- + 2\bar \lambda_{j+1} W_-& \leq C e^{(\frac{4}{5}\delta-2\bar \lambda )\tau},\\
    \left|\tfrac{d}{d\tau}W_*+2\bar \lambda_{j} W_*\right| &\leq  C e^{(\frac{4}{5}\delta-2\bar \lambda )\tau},\\
    \tfrac{d}{d\tau}W_+ + 2\bar \lambda_{j-1} W_+ &\geq - C e^{(\frac{4}{5}\delta-2\bar \lambda )\tau},
\end{align}
hold for $\tau \leq T$.
\end{proposition}
\begin{proof}
By definition of $\bar \lambda$, $\|w(\cdot,\tau)\|_h\leq C e^{-(\bar\lambda+\frac{\delta}{10})\tau}$ holds for some $C$ and sufficiently negative $\tau$. Next, we observe 
\begin{align}
    \tfrac12\tfrac{d}{d\tau}W_*= (P_{=\bar\lambda_j}w,\mathcal{L} w+ F)_h=-\bar\lambda_jW_*+(P_{=\bar\lambda_j}w, F)_h.
\end{align}
Then, \eqref{eq:F_projection} yields $|(P_{=\bar\lambda_j}w,  F)_h|\leq Ce^{\delta\tau}\|w\|_h^2\leq Ce^{(\frac{4}{5}\delta-2\bar\lambda)\tau}$ for sufficiently negative $\tau$. The other two inequalities can be shown similarly. (See the proof of Lemma \ref{lem:MZ_v}.)
\end{proof}
\bigskip

\begin{lemma}\label{lem:first_decay_rate}
Suppose that $v_1,v_2$ are smooth ancient solutions to \eqref{eq:v-flow2} satisfying \eqref{eq:exp_decay} with the difference $w=v_1-v_2 \not \equiv 0$. We recall $\bar \lambda,\bar \lambda_j$ in \eqref{eq:optimal_decay} and \eqref{eq:optimal_eigenvalue}. Then, we have $\bar \lambda=\bar \lambda_j$, and there are some constants $C,\delta'>0$ and $T\ll -1$ such that 
\begin{equation}\label{w-Pw-2/5}
    \|w-P_{=\bar \lambda_j}w\|_h \leq C e^{(\delta'-\bar \lambda_j)\tau}
\end{equation}
holds for $\tau \leq T$.
\end{lemma}

\begin{proof} Given any $\delta' \in (0, \bar \lambda_{j+1}-\bar \lambda] $, by Proposition \ref{prop:W_ODEs}, there are some $C,T$ such that 
\begin{align}
        \tfrac{d}{d\tau} e^{(2\bar \lambda +2\delta')\tau }W_- \leq -2(\bar \lambda_{j+1}-\bar \lambda-\delta')e^{(2\bar \lambda +2\delta')\tau }W_- +Ce^{(\frac{4}{5}\delta  +2\delta'  )\tau} \leq Ce^{(\frac{4}{5}\delta  +2\delta'  )\tau}
\end{align}
holds for $\tau\leq T$. Since $ e^{2\bar \lambda+2\delta'}W_- \leq  e^{2\bar \lambda+2\delta'}\|w\|_h^2 \to 0$ as $\tau \to -\infty$, we have
\begin{equation}
e^{(2\bar \lambda +2\delta')\tau }W_-(\tau)=\int_{-\infty}^{\tau}\tfrac{d}{ds} e^{(2\bar \lambda +2\delta')s }W_-(s)ds\leq \int_{-\infty}^{\tau}Ce^{(\frac{4}{5}\delta  +2\delta'  )s}ds =Ce^{(\frac{4}{5}\delta  +2\delta'  )\tau}.
\end{equation}
Namely,
\begin{equation}\label{eq:P>_est}
 \|P_{>\bar\lambda_j}w\|_h=: \sqrt{W_-}  \leq Ce^{(\frac{2}{5}\delta  -\bar\lambda  )\tau}.
\end{equation}

Next, we consider $W_+$. If $\bar \lambda_j= \lambda_1$ then we have $W_+=0$. So, we may assume $\bar \lambda_j > \lambda_1$ and choose any $0< \delta' \leq  \min\{\bar \lambda-\bar \lambda_{j-1},\frac{1}{5}\delta \}$. Then, by Proposition \ref{prop:W_ODEs},
\begin{align}
        \tfrac{d}{d\tau} e^{(2\bar \lambda -2\delta')\tau }W_+ \geq 2(\bar \lambda -\delta' -\bar \lambda_{j-1})e^{(2\bar \lambda +2\delta')\tau }W_+ -Ce^{(\frac{4}{5}\delta  -2\delta'  )\tau} \geq -Ce^{(\frac{4}{5}\delta  -2\delta'  )\tau}
\end{align}
holds for some constant $C$ and negative enough $\tau$. Therefore,
\begin{equation}
    e^{(2\bar \lambda -2\delta')\tau }W_+(\tau)-e^{(2\bar \lambda -2\delta')\tau_0 }W_+(\tau_0) \leq \int^{\tau_0}_{\tau}Ce^{(\frac{4}{5}\delta  -2\delta'  )s}ds\leq C. 
\end{equation}
This implies $e^{(2\bar \lambda -2\delta')\tau }W_+ \leq C$. Thus,
\begin{equation}\label{eq:P<_est}
    \|P_{<\bar\lambda_j}w\|_h=: \sqrt{W_+}  \leq Ce^{( \delta'  -\bar\lambda  )\tau}.
\end{equation}
Hence, combining with \eqref{eq:P>_est} yields \eqref{w-Pw-2/5}.

Finally, if $\bar \lambda \neq \bar \lambda_j$, then as \eqref{eq:P<_est} we can show
 \begin{equation}
    \|P_{=\bar\lambda_j} w\|_h=: \sqrt{W_*}  \leq Ce^{( \delta'  -\bar\lambda  )\tau}
\end{equation}
for some $\delta'>0$. This contradicts the definition of $\bar \lambda$ and \eqref{w-Pw-2/5}. Thus, $\bar \lambda = \bar \lambda_j$.
 \end{proof}

\bigskip

\begin{corollary}\label{cor:limit-exists}
Suppose that $v_1,v_2$ are smooth ancient solutions to \eqref{eq:v-flow2} satisfying \eqref{eq:exp_decay} with the difference $w=v_1-v_2 \not \equiv 0$. There exists a constant $a_0>0$ satisfying
\begin{align}\label{elth}
    \lim_{\tau\to -\infty}e^{\bar\lambda_j\tau}\|P_{=\bar\lambda_j}w\|_h=a_0,
\end{align}
where $\bar \lambda_j$ defined in \eqref{eq:optimal_eigenvalue}. 
\end{corollary} 
\begin{proof}
 Proposition \ref{prop:W_ODEs} and $\bar \lambda = \bar\lambda_j$ in Lemma \ref{lem:first_decay_rate}  imply
\begin{equation}
    \tfrac{d}{d\tau}e^{ 2\bar \lambda_j \tau} W_* = O(e^{\frac{4}{5}\delta  \tau}). 
\end{equation}
Then, for any $\tau'<\tau$, we have
\begin{align}\label{eq:Cauchy_seq}
    |e^{2\bar\lambda_j\tau}W_*(\tau)-e^{2\bar\lambda_j\tau'}W_*(\tau')|\leq C\int_{\tau'}^{\tau}e^{\frac{4}{5}\delta \tau}\leq Ce^{\frac{4}{5}\delta\tau}.
\end{align}
Therefore, $e^{2\bar\lambda_j\tau}W_*(\tau)$ is a Cauchy sequence. Hence, there is $a_0 \geq 0$ such that
\begin{equation}
    \lim_{\tau\to -\infty}e^{2\bar\lambda_j\tau}W_*(\tau)=a_0.
\end{equation}

Towards a contradiction, we suppose   $a_0=0$. Then, \eqref{eq:Cauchy_seq} yields
\begin{equation}
    e^{ 2\bar \lambda_j \tau}    W_*(\tau)\leq Ce^{\frac{4}{5}\delta  \tau}.
\end{equation}
Combining with Lemma \ref{lem:first_decay_rate} contradicts the definition of $\bar\lambda$.
\end{proof}

\begin{proof}[Proof of Theorem \ref{thm:asy_diff}]
Using \eqref{w-Pw-2/5} and \eqref{elth}, we can obtain 
\begin{align}\label{eq:w_h.estimate}
    \|w(\cdot,\tau)\|_h\leq \|w-P_{=\bar\lambda_j}w\|_h+\|P_{=\bar\lambda_j}w\|_h\leq 2a_0 e^{-\bar\lambda_j\tau},
\end{align}
for sufficiently negative $\tau$. On the other hand, we have
\begin{align}
\tfrac{d}{d\tau}e^{\bar\lambda_j\tau}(w,\vec{\varphi}_j)_h=e^{\bar\lambda_j\tau}(\bar\lambda_j w+ \mathcal{L}w+F ,\vec{\varphi}_j)_h=e^{\bar\lambda_j\tau}(F ,\vec{\varphi}_j)_h.
\end{align}
Hence, combining with  \eqref{eq:F_projection} and \eqref{eq:w_h.estimate} leads to
\begin{align}
\left|\tfrac{d}{d\tau}e^{\bar\lambda_j\tau}(w,\vec{\varphi}_j)_h\right|\leq C e^{\delta \tau+\bar\lambda_j\tau}\|w\|_{h}\leq Ce^{\delta\tau}.
\end{align}
Thus, $e^{\bar\lambda_j\tau}(w,\vec{\varphi}_j)_h$ has the backward limit $\vec{a}\in \mathbb{R}^{d_j}$. Also,  \eqref{elth} implies $\vec{a}\neq 0$.
\end{proof}

 \subsection{Classification}

 In this section, we classify convex closed ancient $\alpha$-curve shortening flows with $\alpha<\frac{1}{3}$ by proving Theorem \ref{thm:subaffine.classification}.

\begin{proof}[Proof of Theorem \ref{thm:subaffine.classification}]
Suppose that $\overline{\Gamma}_\tau$ is a smooth convex closed rescaled ancient $\alpha$-CSF with $\alpha<\frac{1}{3}$. By Proposition \ref{prop:unique_tangent.flow}, $\overline{\Gamma}_\tau$ converges to a shrinker as $\tau\to -\infty$. We may rotate the flow, in order to make it converge to $\Gamma^\alpha_k$. Then, as discussed in Remark \ref{rmk:I_3_paramater}, it is enough to show that there exists a certain $\boldsymbol{a}'=(\vec{a}_1',\cdots,\vec{a}_{\bar I}')$ such that $\bar u -h=\mathcal{S}(\boldsymbol{a}')$ holds for sufficiently negative time, where $\bar u, h$ are the support functions of $\overline{\Gamma}_\tau,\Gamma_k^\alpha$, respectively.

Since $v:=\bar u-h$ satisfies \eqref{eq:exp_decay}, we can apply Theorem \ref{thm:asy_diff} for $w=v-\mathcal{S}(\boldsymbol{0})$. Notice that we have $\mathcal{S}(\boldsymbol{0})\equiv 0 $ by Proposition \ref{prop:existence}. Hence, we can find $\bar \lambda_i<0$ and $\vec{a}_i\neq 0$ satisfying \eqref{eq:rep_coe}, unless $v\equiv \mathcal{S}(\boldsymbol{0})\equiv 0$. If $ i=\bar I $, then we set $\vec{a}'_{\bar I}=\vec{a}_i$. Otherwise, we set $\vec{a}_{\bar I}'=0$. Then, we have
\begin{equation}\label{eq:first_rep}
\lim_{\tau\to -\infty}\|e^{\bar \lambda_{\bar I} \tau} v-\vec{a}_{\bar I}'\cdot \vec{\varphi}_{\bar I}\|_h=0.
\end{equation}
Now, we define $\boldsymbol{a}_{\bar I}'=(0,\cdots,0,\vec{a}_{\bar I}')$ and consider $\mathcal{S}(\boldsymbol{a}_{\bar I}')=\sum_{l=1}^Lv^{(l)}(\boldsymbol{a}_{\bar I}')$. Then, (4) and (c) in Proposition \ref{prop:existence} imply that 
\begin{align}
    \lim_{\tau\to -\infty}\|e^{\bar{\lambda}_{\bar I}\tau}\mathcal{S}(\boldsymbol{a}_{\bar I}')-\vec{a}_{\bar I}'\cdot \vec{\varphi}_{\bar I}\|_h=0.
\end{align}
Combining this with \eqref{eq:first_rep}, we obtain
\begin{equation}\label{eq:first_approx}
\lim_{\tau\to -\infty}e^{\bar \lambda_{\bar I} \tau}\|v-\mathcal{S}(\boldsymbol{a}_{\bar I}')\|_h=0.
\end{equation}

Next, we again apply Theorem \ref{thm:asy_diff}  for $w=v-\mathcal{S}(\boldsymbol{a}_{\bar I}')$. Then, we can find  $\bar \lambda_i<0$ and $\vec{a}_i\neq 0$ satisfying \eqref{eq:rep_coe}, unless $v=\mathcal{S}(\boldsymbol{a}_{\bar I}')$. In addition,  \eqref{eq:first_approx} implies $i\leq \bar I-1$. If $i=\bar I-1$ then we set $\vec{a}'_{\bar I-1}=\vec{a}_i$. Otherwise, we set $\vec{a}_{\bar I-1}'=0$. Then, as above we can obtain
\begin{equation}
\lim_{\tau\to -\infty}e^{\bar \lambda_{\bar I-1} \tau}\|v-\mathcal{S}(\boldsymbol{a}_{\bar I-1}')\|_h=0,
\end{equation}
where $\boldsymbol{a}_{\bar I-1}'=(0,\cdots,0,\vec{a}_{\bar I-1}',\vec{a}_{\bar I}')$. 

We repeat this process at most $\bar I$-times so that we find either $v\equiv\mathcal{S}(\boldsymbol{0})\equiv 0$, $v\equiv\mathcal{S}(\boldsymbol{a}_{i}')$ for a certain  $2\leq i\leq \bar I$, or
\begin{equation}
\lim_{\tau\to -\infty}e^{\bar \lambda_{1} \tau}\|v-\mathcal{S}(\boldsymbol{a}_{1}')\|_h=0.
\end{equation}
In the last case, Theorem \ref{thm:asy_diff}  with $w=v-\mathcal{S}(\boldsymbol{a}_{1}')$ concludes $ v\equiv \mathcal{S}(\boldsymbol{a}_{1}')$. This completes the proof.
\end{proof}

\appendix
\section{Morse index and kernel}

\begin{theorem}[\cite{choisun}]\label{thm:sep-L} 
Suppose that $0<\alpha \neq \frac13$. For the linearized operator \eqref{eq:L-def}, we have
\begin{enumerate}
\item The Morse index of $\mathcal{L}_{\Gamma^{\alpha}_k}$ is $2k-1$, and  $\ker \mathcal{L}_{\Gamma^{\alpha}_k}=\textup{span}\{h_\theta\}$, where $h$ is the support function of $\Gamma^{\alpha}_k$.
\item The Morse index of $ \mathcal{L}_{\Gamma^{\alpha}_\infty}$ is $2\lceil\sqrt{1+1/\alpha}\rceil-1$.\footnote{$\lceil x\rceil$ denotes the least integer greater than or equal to $x$.}  If $\alpha=\frac{1}{k^2-1}$, then $\ker \mathcal{L}_{\Gamma^{\alpha}_\infty}=\textup{span}\{\cos k\theta,\sin k\theta\}$. Otherwise $\ker \mathcal{L}_{\Gamma^{\alpha}_\infty}=\emptyset$.
\end{enumerate}
\end{theorem}

\section{ODE lemmas}
The following lemma is used in Section \ref{sec:dich}. 
\begin{lemma}[\cite{merle1998optimal, choi2018ancient,choi2019ancient}]\label{lem:MZ} Suppose that $x, y, z:(-\infty, T] \rightarrow[0, \infty)$ are absolutely continuous functions satisfying
$ x+y+z>0$, $\displaystyle \liminf _{\tau \rightarrow-\infty} y(\tau)=0$, and 
\begin{align}
\begin{split}
\left|x^{\prime}\right| &\leq \varepsilon(x+y+z), \\
y^{\prime}  &\leq -y+ \varepsilon(x+z), \\
z^{\prime} & \geq z-\varepsilon(x+y),
\end{split}
\end{align}
for some $\varepsilon>0$. Then, there exist some positive universal constants $\varepsilon_{0},c $ such that if $\varepsilon \leq \varepsilon_{0}$  then either $x+y\leq c\varepsilon z$ holds on $(-\infty,T]$  or $y+z=o(x)$  as $\tau\to -\infty$ .
\end{lemma}

In Section \ref{sec:non-radial}, we need a version of the above lemma on a finite interval.
\begin{lemma}\label{lem:bdd-MZ}
Let $x, y, z:[-L, L] \rightarrow[0, \varepsilon)$ be absolutely continuous functions satisfying 
\begin{align*} 
\left|x^{\prime}\right| & \leq \sigma(x+y+z), \\
 y^{\prime}& \leq -y+\sigma(x+z),\\
 z^{\prime} & \geq z-\sigma(x+y),
\end{align*}
for some $\sigma \in (0,\frac{1}{100})$. Then, 
\begin{equation}
y+z\leq 8\sigma x+4\varepsilon e^{-\frac14L},
\end{equation}
holds for $s\in[-L/2,L/2]$.
\end{lemma}
\begin{proof}
We define $ \beta:=y- 3 \sigma(x+z)$. Then, $0<\sigma<\frac{1}{100}$ and $x,y,z\geq 0$ imply
\begin{equation}
\beta^{\prime} \leq-y+\sigma(x+z)+3 \sigma^{2}(x+y+z)-3 \sigma(z-\sigma(x+y)) \leq   -\tfrac{1}{2}\beta,
\end{equation}
namely $\frac{d}{ds}[ e^{\frac{s}{2}}\beta(s)]\leq 0$. Hence, remembering $x,y,z\in [0,\varepsilon)$,  if $s \geq -L/2$, then
\begin{equation}
\beta(s) \leq e^{-\frac{s}{2} }e^{ - \frac{L}{2} }\beta(-L)\leq \varepsilon e^{-\frac{L}{4}},
\end{equation}
holds. Namely, the following holds for $s\in [-L/2,L]$.
\begin{equation}\label{eq:ODE2-y-bound}
y\leq 3\sigma(x+z)+\varepsilon e^{-\frac{L}{4}}.
\end{equation}
Similarly,   $\gamma:=z-3\sigma(x+y)$ satisfies $\gamma'\geq \frac{1}{2}\gamma$, namely $\frac{d}{ds}[e^{-\frac{s}{2}}\gamma(s)]\geq 0$. Thus,
\begin{equation}
\gamma(s) \leq e^{ \frac{s}{2} }e^{ - \frac{L}{2} }\gamma(L)\leq \varepsilon e^{-\frac{L}{4}},
\end{equation}
holds  for $s\leq L/2$. Namely, the following holds for $s\in [-L,L/2]$.
\begin{equation}\label{eq:ODE2-z-bound}
z\leq 3\sigma(x+y)+\varepsilon e^{-\frac{L}{4}}.
\end{equation}
Since we have $\sigma <\frac{1}{100}$ and $y,z\geq 0$, adding \eqref{eq:ODE2-y-bound} and \eqref{eq:ODE2-z-bound} yields
\begin{equation}
y+z \leq 6\sigma x +\tfrac{1}{10}(y+z) +2\varepsilon e^{-\frac{L}{4}},
\end{equation}
for $|s|\leq L/2$. This completes the proof.
\end{proof}

\begin{lemma}\label{lem:ODE-basic}  $\tilde{\rho}, f:(-\infty,T]\to \mathbb{R}$ are absolutely continuous   functions satisfying
\begin{align}\label{eq:ODE-basic-assumption}
& \lim_{\tau\to -\infty} f(\tau)=0 , && \lim_{\tau\to -\infty}\tilde{\rho}(\tau)e^{-\frac{|\lambda|}{2}\tau}=+\infty, && |\tilde{\rho}'| \leq \tfrac{1}{2}|\lambda| \tilde{\rho},
\end{align}
for some $\lambda \neq 0$. Moreover, there exist  some constants  $K>0$ and $T_K\leq T$ satisfying
\begin{equation}\label{eq:ODE1-1}
 \lambda^{-1} f'\leq   -f +   K \tilde{\rho} , 
\end{equation}
for $\tau\in (-\infty,T_K]$. Then,  there exists some $T_*\leq T_K$ such that 
\begin{equation}
f_+:=\max\{f,0\}\leq 2K\tilde{\rho}
\end{equation}
holds for $\tau \leq T_*$.
\end{lemma}

\begin{remark}
Notice that $0\leq |\tilde{\rho}'|\leq \frac{1}{2}|\lambda| \tilde{\rho}$ implies $\tilde{\rho}\geq 0$. However, $f$ is not necessarily non-negative.
\end{remark}

\begin{proof}
Towards a contradiction, we suppose that there exists a decreasing sequence $\{\tau_i\}_{i\in \mathbb{N}}$ satisfying $\tau_i\leq T_K$, $\lim  \tau_i=-\infty$, and $\beta(\tau_i)>0$, where  $\beta:=f-2K \tilde{\rho}$. Then, 
\begin{equation}
 \lambda^{-1} \beta'=  \lambda^{-1} f'-2 \lambda^{-1} K\tilde\rho'\leq  - f  + 2 K \tilde{\rho}=  -\beta 
\end{equation}
holds for $\tau \leq T_K$, namely  
\begin{equation}\label{eq:ODE-lemma-beta}
\tfrac{1}{\lambda}\tfrac{d}{d\tau}[e^{ \lambda \tau}\beta(\tau) ] \leq 0.
\end{equation}
 Hence, if $\lambda>0$ then  \eqref{eq:ODE-lemma-beta} implies $e^{\lambda\tau}\beta(\tau)\geq e^{\lambda \tau_1}\beta(\tau_1)>0$  for $\tau\leq \tau_1$. Thus, \eqref{eq:ODE1-1} and $\beta>0$ yield 
\begin{equation}
\lambda^{-1} f'\leq -f+K \tilde{\rho}=-\beta-K \tilde{\rho}<0,
\end{equation} 
for $\tau\leq \tau_1$. Therefore, $f$ is  decreasing in $(-\infty,\tau_1]$, and this contradicts
\begin{equation}
0<\beta(\tau_1)\leq f(\tau_1)\leq \lim_{\tau \to -\infty}f(\tau)=0.
\end{equation}
Therefore, we have $\lambda<0$ and thus $e^{\lambda\tau}\beta$ is an increasing function by \eqref{eq:ODE-lemma-beta}. Hence, $\beta(\tau_i)>0$ implies $\beta(\tau)>0$ for  $\tau\in [\tau_i,T_K]$. Since $\tau_i$ diverges to $-\infty$, we have $\beta(\tau)>0$ for all $\tau\leq T_K$.  Thus, \eqref{eq:ODE1-1} yields
\begin{equation}
  f'\geq  -\lambda ( f-K\rho) =|\lambda|( \tfrac12 f +\tfrac12 \beta)>\tfrac{|\lambda|}{2} f.
\end{equation}
Therefore, for $\tau\leq T_K$ we have
\begin{equation}
 e^{-\frac{|\lambda|}{2}T_K}f(T_K) \geq e^{-\frac{|\lambda|}{2}\tau}f(\tau) \geq  2K e^{-\frac{|\lambda|}{2}\tau}\tilde{\rho}(\tau).  
\end{equation}
This contradicts the assumption that  $e^{-\frac{|\lambda|}{2}\tau}\tilde{\rho}\to +\infty$ as $\tau \to -\infty$.
\end{proof}

\begin{lemma}\label{lem:ODE-1}  $\tilde{\rho}, f:(-\infty,T]\to \mathbb{R}$ are absolutely continuous   functions satisfying \eqref{eq:ODE-basic-assumption} and $
\lambda^{-1} f'\leq   -f +o( \tilde{\rho})$ for some $\lambda\neq 0$. Then,   $f_+= o(\tilde{\rho})$ holds.
\end{lemma}

\begin{proof}
Given any small $\varepsilon>0$, we apply Lemma \ref{lem:ODE-basic} with $K=\varepsilon$.
\end{proof}

\begin{lemma}\label{lem:ODE-2}  $\tilde{\rho}, f:(-\infty,T]\to \mathbb{R}$ are absolutely continuous  functions satisfying \eqref{eq:ODE-basic-assumption} and $
  f'=-\lambda   f  +O( \tilde{\rho})$ for some $\lambda\neq 0$. Then, we have  $ f =O(\tilde{\rho}) $.
\end{lemma}

\begin{proof}
$f'=-\lambda   f  +O( \tilde{\rho})$ implies $|\lambda^{-1} f'+f |\leq K \tilde{\rho} $ for some $K>0$. Thus, Lemma \ref{lem:ODE-basic} implies $f\leq 2K\tilde{\rho}$ for negative enough $\tau$. In addition, we can apply Lemma \ref{lem:ODE-basic} for $-f$ so that we can obtain $-f\leq 2K\tilde{\rho}$  for negative enough $\tau$. This completes the proof.
\end{proof}

\section{Shrinkers in entropy order}
In this part, we investigate the entropies of closed shrinkers. Note that the entropy is invariant under scaling. 
%one can write $\mathcal{E}_\alpha(\Gamma_\alpha^k)=\mathcal{E}_\alpha(\hat \Gamma_\alpha^k)=\mathcal{E}_\alpha(\Gamma^{\alpha}_k)$ without ambiguity. 
We will show the following ordering of the entropy.
\begin{theorem}\label{thm:ent-cmp}Given $\alpha\in(0,1/8)$, the entropies of closed shrinkers satisfy
\begin{equation}
0=\mathcal{E}_\alpha( \Gamma^{\alpha}_\infty)>\mathcal{E}_\alpha( \Gamma^{\alpha}_{k_0})>\cdots>\mathcal{E}_\alpha( \Gamma^{\alpha}_{3}),
\end{equation}
where $k_0=\lceil\sqrt{1+1/\alpha}\rceil-1\geq 3$.  Here, $\lceil x\rceil$ denotes the smallest integer that is greater than or equal to $x$.
\end{theorem}

To begin with, we calculate the entropies of shrinkers by using their support functions.
\begin{proposition}\label{prop:ent-cmp-h}
The shrinker $\Gamma_k^\alpha$ with the support function $h$ satisfies
\begin{align}\label{ent-shrinker}
\mathcal{E}_\alpha( \Gamma^{\alpha}_k)=-\frac{1+\alpha}{2(1-\alpha)}\log\left(\fint_{\mathbb{S}^1}  h^{1-1/\alpha}(\theta)d\theta\right).
\end{align}
\end{proposition}
\begin{proof}
We recall from Proposition \ref{prop:ent-pt} the unique entropy point $z_e$ of $ \Omega^\alpha_k$, where $ \Gamma_k^\alpha=\partial  \Omega_k^\alpha$. If $z_e\neq 0$, then the rotation by $2\pi/k$ yields another entropy point due to the symmetry of $\Gamma_k^\alpha$. This contradicts the uniqueness of the entropy point and thus we have $z_e=0$. Hence, \eqref{ent-shrinker} follows from \eqref{def:ent-2} and the following identity
\begin{align}
    \mathcal{A}(\Omega^\alpha_k)=\frac{1}{2}\int_{\mathbb{S}^1}h(h+h_{\theta\theta})d\theta=\frac{1}{2}\int_{\mathbb{S}^1}h^{1-\frac{1}{\alpha}}(\theta)d\theta.
\end{align}
 \end{proof}

Next, we will recall the period function $\Theta(\alpha,r)$ from \cite[Section 2]{andrews2003classification}, and we will also recall a family of solutions $U(\alpha,r,\theta)$ to the shrinker equation \eqref{eq:h2} from the proof of Lemma 7.2 of \cite{andrews2003classification}. In order to explain $\Theta(\alpha,r)$ and $U(\alpha,r,\theta)$, we begin by considering the solution $u:[0,+\infty)\to \mathbb{R}_+$ to the initial value problem
\begin{align}
&u''+u=u^{-\frac{1}{\alpha}}, && u'(0)=0, && u(0)=u_+>1.
\end{align}
Since $u''(0)=u_+^{-\frac{1}{\alpha}}-u_+<0$, there exists some $\Theta $ such that $u'<0$ in $(0,\Theta)$ and $u'(\Theta)=0$. We denote $u_-=u(\Theta)$ and consider the ratio $r=\frac{u_+}{u_-}>1$. Then, by (2.4) and (2.6) in \cite{andrews2003classification}, $u_\pm$ and $r$ have the following relation for $\alpha\neq 1$.
\begin{equation}\label{eq:E(r)-def}
u_\pm^2+\tfrac{2\alpha}{1-\alpha}u_\pm^{1-\frac{1}{\alpha}}=E(\alpha, r):=\left(\tfrac{2\alpha(1-r^{\frac{\alpha-1}{\alpha}})}{(1-\alpha)(r^2-1)}\right)^{\frac{2\alpha}{\alpha+1}}\left(\tfrac{r^2-r^{\frac{\alpha-1}{\alpha}}}{1-r^{\frac{\alpha-1}{\alpha}}}\right).
\end{equation}
As mentioned in \cite{andrews2003classification}, $E(r)$ is an increasing function of $r$ for $r\geq 1$. In addition, we can observe that $F(u_+):=u_+^2+\tfrac{2\alpha}{1-\alpha}u_+^{1-\frac{1}{\alpha}}$ is also an increasing function of $u_+$ for $u_+>1$. Hence, we can consider $u_+>1$ as a function of $(\alpha, r)$, and we can parametrize $\Theta$ by $(\alpha,r)$ instead of $(\alpha,u_+)$. Indeed, $\Theta:(0,1)\times(1,\infty)\to \mathbb{R}$ satisfies
\begin{equation}
\Theta(\alpha,r)=\int_1^r\left(\tfrac{r^2-r^{\frac{\alpha-1}{\alpha}}}{1-r^{\frac{\alpha-1}{\alpha}}}-x^2-\tfrac{r^2-1}{1-r^{\frac{\alpha-1}{\alpha}}}x^{\frac{\alpha-1}{\alpha}} \right)^{-\frac{1}{2}}dx.
\end{equation}
Moreover, we denote by $U(\alpha,r,\theta)$ the solution to the equation
\begin{align}\label{eq:U-ode}
U_{\theta\theta}+U=U^{-\frac{1}{\alpha}}
\end{align}
satisfying $U_{\theta}(\alpha,r, 0)=0$ and  $U(\alpha,r, 0)=u_+(\alpha, r)>0$.
Then, we have  
\begin{equation}\label{eq:U-Neumann}
U_{\theta}(\alpha, r, \Theta(\alpha, r))=U_{\theta}(\alpha,r, 0) =0.
\end{equation}

In this appendix, we will fix $\alpha\in (0,\frac{1}{3})$ so that we can denote $\Theta(\alpha,r)$, $U(\alpha,r,\theta)$ and $E(\alpha, r)$ by $\Theta(r)$, $U(r,\theta)$ and $E(r)$, respectively.

Now, we define
\begin{equation}
\eta(\theta)=\tfrac{\partial}{\partial r}  U( r,\theta).
\end{equation}
Then, \eqref{eq:U-ode} implies
\begin{equation}\label{eq:eta-eq}
\eta_{\theta\theta}+\eta=-\tfrac{1}{\alpha}U^{-1-\frac{1}{\alpha}}\eta,
\end{equation}
and also differentiating \eqref{eq:U-Neumann} in $r$ yields
\begin{align}\label{ap:eta-bdry}
\eta_{\theta}(\Theta(r))+U_{\theta\theta}(r,\Theta(r)) \Theta_r(r)=\eta_\theta(0)=0.
\end{align}

Finally, we define $f:(1,\infty)\to \mathbb{R}$ by
\begin{align}\label{eq:f-entropy-def}
f(r)=\frac{1}{\Theta(r)}\int_0^{\Theta(r)}U(r,\theta)^{1-1/\alpha}d\theta.
\end{align}
Notice that $U(r,\Theta(r))\leq U(r,\theta)\leq U(r,0)$ and $r=U(r,0)/U(r,\Theta(r))$. Hence,
\begin{equation}\label{eq:f-limit-1}
\lim_{r\to 1^+}f(r)=1.
\end{equation}

\begin{lemma}\label{lem:f-monotone}
 Given $\alpha\in(0,\frac{1}{3})$, $f(r)$ is an increasing function.
\end{lemma}

\begin{proof}Taking the derivative of $f$ yields
\begin{align}
\frac{df(r)}{dr}=\Theta^{-1}\left[U(r,\Theta)^{1-1/\alpha}\frac{d\Theta}{dr} +\int_0^{\Theta}\left(1-\tfrac{1}{\alpha}\right)U^{-1/\alpha}\eta d\theta\right]-\Theta^{-1}\frac{d\Theta}{dr} f.
\end{align}
To calculate the integral on the RHS, we combine \eqref{eq:U-ode} and \eqref{eq:eta-eq} as follows.
\begin{align}
(1+\tfrac{1}{\alpha})\int_0^{\Theta}U^{-\frac{1}{\alpha} }\eta \,d\theta= \int_0^\Theta (U_{\theta\theta}+U)\eta-(\eta_{\theta\theta}+\eta)U\,d\theta
=U_\theta\eta-U\eta_{\theta}\big|_0^{\Theta}.
\end{align}
Hence, the Neumann conditions \eqref{eq:U-Neumann} and \eqref{ap:eta-bdry} imply
\begin{align}
(1+\tfrac{1}{\alpha})\int_0^{\Theta}U^{-1/\alpha}\eta d\theta=U(r,\Theta) U_{\theta\theta}(r,\Theta)\frac{d\Theta}{dr}.
\end{align}
Therefore, combining with the first equation yields
\begin{align}
f'=\Theta^{-1}\Theta'\left[U(r,\Theta)^{1-1/\alpha}+\tfrac{\alpha-1}{\alpha+1}U_{\theta\theta}(r,\Theta)U(r,\Theta)-f\right].
\end{align}
Using $U_{\theta\theta}=U^{-1/\alpha}-U$, we obtain the following identity
\begin{align}
U^{1-\frac{1}{\alpha} }+\tfrac{\alpha-1}{\alpha+1}U_{\theta\theta}U=  \tfrac{\alpha -1 }{\alpha+1}\left[\tfrac{2\alpha}{\alpha-1}U^{1-\frac{1}{\alpha}}-U^2\right].
\end{align}
Hence, by the definition \eqref{eq:E(r)-def} of $E(r)$, we have 
\begin{equation}\label{eq:f_r-entropy}
f' =\Theta^{-1} \Theta' \left[\tfrac{1-\alpha}{1+\alpha}E -f\right]=\tfrac{1-\alpha}{1+\alpha}\Theta^{-1} \Theta' \left[E -\tfrac{1+\alpha}{1-\alpha}f\right].
\end{equation}

On the other hand, by using \eqref{eq:U-ode} we can obtain 
\begin{equation}
 \tfrac{ \partial}{\partial \theta}(U^2+\tfrac{2\alpha}{1-\alpha}U^{1-\frac{1}{\alpha}})=2( U-U^{-\frac{1}{\alpha}}) U_\theta=-2U_{\theta\theta}U_\theta=-\tfrac{\partial}{\partial \theta} U_\theta^2.
\end{equation}
Therefore,
\begin{equation}
E(r)=U^2_\theta(r,\theta)+U^2(r,\theta)+\tfrac{2\alpha}{ 1-\alpha}U^{1-\frac{1}{\alpha}}(r,\theta),
\end{equation}
holds for $\theta\in [0,\Theta(r)]$, and thus we have
\begin{equation}
E(r)=\frac{1}{\Theta(r)}\int_0^{\Theta(r)}U^2_\theta(r,\theta)+U^2(r,\theta)+\tfrac{2\alpha}{1- \alpha}U^{1-\frac{1}{\alpha}}(r,\theta)\, d\theta.
\end{equation}
Hence, combining with \eqref{eq:f-entropy-def} and \eqref{eq:U-ode} leads to
\begin{equation}
E-\frac{1+\alpha}{1-\alpha}f=\frac{1}{\Theta }\int_0^{\Theta }U^2_\theta +U^2-U^{1-\frac{1}{\alpha}}\, d\theta=\frac{1}{\Theta}\int_0^\Theta \,U_\theta^2-U U_{\theta\theta}\,d\theta.
\end{equation}
Therefore, the Neumann condition \eqref{eq:U-Neumann} implies
\begin{equation}
E-\frac{1+\alpha}{1-\alpha}f=\frac{2}{\Theta }\int_0^{\Theta }U^2_\theta \, d\theta.
\end{equation}
Thus, \eqref{eq:f_r-entropy} yields
\begin{equation}
f'=\frac{2(1-\alpha)}{1+\alpha}\Theta^{-2} \Theta'\int_0^{\Theta }U^2_\theta \, d\theta.
\end{equation}
Since ${\frac{d\Theta}{dr}}>0$ holds for $\alpha\in(0,\frac{1}{3})$ by \cite[Corollary 5.6]{andrews2003classification}, we have the desired result.
\end{proof}
\bigskip

\begin{proof}[Proof of Theorem \ref{thm:ent-cmp}]
By \cite[Theorem 3.1]{andrews2003classification} and \cite[Corollary 5.6]{andrews2003classification}, for each $3\leq k \leq k_0 $ there exists $r_k$ satisfying $\Theta(r_k)=\pi/k$. Moreover, \cite[Corollary 5.6]{andrews2003classification} implies
\begin{equation}
1<r_{k_{0}}<\cdots<r_3<\infty.
\end{equation}
Hence, \eqref{eq:f-limit-1} and \eqref{lem:f-monotone} yield
\begin{equation}
0<\log f(r_{k_0})<\cdots<\log f(r_3).
\end{equation}

Now, we recall the definition of $\Gamma_k^\alpha$ that the support function $h(\theta)$ of $\Gamma_k^\alpha$ attains its maximum at $\theta=0$. Then, we have $h(\theta)=U(r_k,\theta)$. Therefore, Proposition \ref{prop:ent-cmp-h} and the $k$-fold symmetry of $ \Gamma_k^\alpha$ lead to
\begin{equation}
\mathcal{E}_\alpha( \Gamma^{\alpha}_k)=-\frac{1+\alpha }{2(1-\alpha)}\log f(r_k).
\end{equation}
Hence, the inequality above says 
\begin{equation}
\mathcal{E}_\alpha( \Gamma^{\alpha}_3)<\cdots<\mathcal{E}_\alpha( \Gamma^{\alpha}_{k_0})<0.
\end{equation}
Therefore, we can complete the proof by observing  $\mathcal{E}_\alpha( \Gamma^{\alpha}_\infty)=0$, because the unit circle $ \Gamma^{\alpha}_\infty$ has the support function $ h\equiv 1$.
\end{proof}

\small

\bibliographystyle{plainnat}
\bibliography{CSF-ref}

\end{document}